\documentclass[pdflatex,sn-mathphys-num]{sn-jnl}% Math and Physical Sciences Numbered Reference Style
\usepackage{graphicx}%
\usepackage{multirow}%
\usepackage{amsmath,amssymb,amsfonts}%
\usepackage{amsthm}%
\usepackage{mathrsfs}%
\usepackage[title]{appendix}%
\usepackage{xcolor}%
\usepackage{textcomp}%
\usepackage{manyfoot}%
\usepackage{booktabs}%
\usepackage{algorithm}%
\usepackage{algorithmicx}%
\usepackage{algpseudocode}%
\usepackage{listings}%
\usepackage{subcaption}
\newcommand{\dsp}{\displaystyle}
\theoremstyle{thmstyleone}%
\newtheorem{theorem}{Theorem}
\newtheorem{acknowledgement}[theorem]{Acknowledgement}

\newtheorem{corollary}[theorem]{Corollary}

\newtheorem{definition}[theorem]{Definition}
\newtheorem{example}[theorem]{Example}

\newtheorem{lemma}[theorem]{Lemma}
\newtheorem{notation}[theorem]{Notation}
\newtheorem{problem}[theorem]{Problem}
\newtheorem{proposition}[theorem]{Proposition}
\newtheorem{remark}[theorem]{Remark}

\begin{document}

\title{Cubic Oscillator: Geometric Approach and Zeros of Eigenfunctions}
\author{Faouzi Thabet, Gliia Braek, Marwa Mansouri, and Mondher Chouikhi }
\maketitle

\begin{abstract}
In this paper, we give a geometric approach to the cubic oscillator with three distinct turning points based on the $\mathcal{D\diagup SG}$\emph{\
correspondence }introduced in \cite{Thabet+al}. The existence of quantization conditions, depending on extra data for the potential, is related to some
particular critical graphs of the quadratic differential $\lambda ^{2}\left(z-a\right) \left( z^{2}-1\right) dz^{2}$ where $\lambda$ is a non vanishing complex number, $a\in \mathbb{C}\diagdown \left\{ -1,1\right\}$. We investigate this geometric approach on
two levels: the first level is studying an inverse spectral problem related to cubic oscillator. The second level describes the zeros locations of
eigenfunctions related to this oscillator. Our results may provide a geometric proof of some questions related to cubic potential case. 
\end{abstract}

\textit{MSC2020 Mathematics subject classifications:} 

\textit{Primary 34E46,34E65; Secondary 34E20,34L10}

\textit{Keywords and phrases: }Schr\"{o}dinger equation. PT-symmetry. WKB
method. Quadratic differentials. Stokes line. Half-plane domain. Band domain.
\tableofcontents

\section{Introduction}

This paper is a continuation of previous work highlights the deep
relations between properties of solutions of linear ordinary differential
equation (ODE) with polynomial coefficients in complex domain and the
structure of the critical graph of related quadratic differential. The \textit{Stokes} geometry is relevant in studying asymptotics behavior of solutions of certain eigenvalue problems (EP) related to these ODE. In fact, Stokes graph related to an ODE is crucial to determine the maximal domain of applicability of complex WKB method \cite{fed1} and \cite{masero}.\\
\newline
In this work, we improve this relationship between Stokes graph and solutions of ODE via the example: 
\begin{equation}
-y^{\prime \prime }(z)+\lambda ^{2}p_{a}(z)y(z)=0  \label{first ODE}
\end{equation}
where the potential $p_{a}(z)=(z^{2}-1)(z-a)$, $\lambda $ and $a$ are
complex numbers. This equation has unique (irregular) singularity at
infinity. An eigenvalue problem (EP) related to \eqref{first ODE} is given by boundary
conditions at this point. In \cite{Thabet+al}, a full classification of the Stokes graph related to \eqref{first ODE} was given via the so called $\mathcal{D\diagup SG}$\emph{\ correspondence}. The authors in \cite{Thabet+al} constructed a set of curves ${\Sigma }_{\lambda }$ that divide the complex plane $\mathbb{C}$ into $n_{\lambda }$ simply connected domains $\left( \Omega _{i}\right)_{1\leq i\leq n_{0}}$. The critical graph of the quadratic differential $
-\lambda ^{2}p_{a}(z)dz^{2}$ changes as $a$ goes from $\Omega _{i}$ to $\Omega_{j} $, with $i\neq j$, crossing ${\Sigma }_{\lambda }$. These
mutations of the critical graph affect the number of eigenvalue
problems. We will give a short summary of \cite{Thabet+al} in the Section ( \ref{sub1}).\\
\newline
Primary references, to the best of our knowledge, that build a bridge between the critical graph of quadratic differential and the asymptotics
study of such type of (EP) are \cite{fed1} and \cite{fed2}, by M.V.Fedoryuk. He used the complex WKB, reformulated from different point of view by Voros \cite{voros}, to study an eigenvalue problem with spectral parameter $\lambda $ and boundary conditions imposed on the real line:
\begin{equation}
 {-y^{\prime \prime }(z)+\lambda ^{2}P(z)y(z)=0}{,\quad y(-\infty
,\lambda )=y(+\infty ,\lambda )=0}  \label{fed equation}
\end{equation}
where $\lambda $ is a non-vanishing complex number and $P$ is an entire
function in the complex plane $
\mathbb{C}$. The author in \cite{fed1} established an algorithm that enables one to continue an asymptotic solution of \eqref{fed equation}, known in a domain or along a line (the real line in the case of \eqref{fed equation}), over the whole  complex $z-$plane. The boundary conditions are equivalent to $y\in L^{2}(\mathbb{R})$. The solutions used by Fedoryuk are \emph{subdominant} in some regions in the complex plane (\emph{canonical domains}, See  \ref{topology Stokes}). The spectrum of \eqref{fed equation}, namely the set of all non-vanishing complex numbers $\lambda $ such that there exists a non trivial $L^{2}(\mathbb{R})-$ solution, is discrete. In the case when $P$ is a polynomial with real coefficients and with at least two simple real zeros, a full asymptotic study of the spectrum is given in \cite[chapter 3 \S 5]{fed1}. In the case where $P$ is a complex polynomial, some additional conditions were established on the \textit{Stokes complexes }related to the (ODE). In fact, it was claimed that the spectrum of \eqref{fed equation}, for the case of complex polynomial, tends to accumulate near only finite rays in the complex plane with slope $
\tan \theta _{1}$, $\tan \theta_{2}...\tan \theta_{k}$ (see again \cite[
Chapter 3 \S 6]{fed1} and \cite{shapiroev}).\\
\newline
Conversely, we think that, for a given angle $\theta $, describing the
spectrum accumulation conditions \textit{( }$\theta $\textit{\ is an \emph{accumulation direction;} for definition see  \ref{accumulation direction})} will be very useful, for example to the exact WKB method based on Borel resummation technics \cite{kawai}. In fact, the Borel summability of asymptotics solutions, in a given direction, is guaranteed by the non existence of certain Stokes lines ( called \emph{finite Stokes}\textit{\ \emph{line} }or \emph{short trajectories}), for more details see again \cite
{kawai} and \cite{voros}.\\
\newline
To be more precise, we treat an inverse problem related to ODE with polynomial potential:

\begin{problem}
\label{main problem}For a given $\theta \in \lbrack 0,\pi \lbrack $ and $d$ integer $\geq 3$, can we construct a polynomial $P$ with degree $d$ and at least two distinct zeros such that the differential equation \eqref{first ODE} admits a subdominant solution in at least two different Stokes regions. What are the conditions for $\theta $ is an
"accumulation direction" (see again definition \ref{accumulation direction}) of the spectrum and what about the locations of the zeros of corresponding eigenfunctions?
\end{problem}
In this paper, we try to give answers to these questions for $d=3$. We use results from \cite{Thabet+al}, \cite{fed1} to construct a set of polynomials $p_{a}(z)=(z-a)(z^{2}-1)$ solving problem  \ref{main problem}. We will give a short summary of Fedoryuk's work in (see Subsection  \ref{sub2}). For the general theory of quadratic differentials the reader can see \cite{strebel, jenkins, jenk-spencer, bridgeland+smith}.

A similar result was obtained in \cite{shin2}. The author had studied an
inverse problem related to an eigenvalue problem of type Sturm-Liouville 
\cite[chapter 6]{sibuya}. The methods used there are analytic. In \cite[
theorem 1.6]{shin2}, the asymptotic distribution of the eigenvalue was
investigated to reconstruct \textit{"some"} coefficients of the polynomial
potential. Our method is different in the case $d=3$. We use geometric
method to construct a \textit{"full"} polynomial potential. We hope to
generalize this method to higher degree in forthcoming project.

The paper is organized as follows, in section(\ref{sub1}), we give a preliminary results from \cite{Thabet+al} to define the set ${\Sigma}_{\lambda }$ using a particular quadratic differential. These curves play a crucial role in the classification of (EP). We close this section by a short review of complex WKB method from \cite{fed1,{fed2}}.

In Section \ref{main parag}, the first main result Theorem\ref{main result1} was established and the $\mathcal{D\diagup SG}$ correspondence was investigated to give a geometric classification to eigenvalue related to cubic oscillator. In particular, a necessary condition to the accumulation of the spectrum was proven which answers a question in \cite{shapiroev}.

In Section \ref{zeros section}, the distribution of zeros of eigenfunctions was studied and the second main result (\ref{domain zeros}) was established. The zeros locations are some particular Stokes lines. Our results may be considered as a generalization of results obtained in \cite{eremenko3,shin2} in the case of PT-symmetric and self-adjoint potential arising from Sturm-Liouville problem. The results in Sections (\ref{main parag}, \ref{zeros section}) give a complete answer to  Problem \ref{main problem}. 

In Section \ref{applications}, we give two applications to our results. The first is the construction of cubic oscillator with a non trivial solution with prescribed number of zeros. This problem was introduced in \cite[problem $2.71$]{brann} in a general form. The second is the description of the exact zeros locations of a rescaled eigenfunction raising from a Sturm-Liouville problem studied in \cite{trinh}.
The Section \ref{proofs} is devoted to the proofs of our results. We close the paper by an Appendix  \ref{Appendix} that contains a proof to WKB formulas cited in subsection \ref{sub2}.

\section{Preliminaries and notations}

In all this paper we denote: $\overset{\bullet }{
\mathbb{C}}=\mathbb{C}
\setminus \left\{ -1,1\right\} $, $\mathbb{C}_{+}=\left\{ z\in 
\mathbb{C}\;\,\,\Im z>0\right\} $ and $
\mathbb{C}_{-}=\left\{ z\in \mathbb{C};\,\,\Im z<0\right\} $.

\subsection{Review of $\mathcal{D\diagup SG}$ correspondence\label{sub1}}

\subsubsection{A parametric Quadratic Differential}

In this subsection, we recall some notation and results from \cite{Thabet+al}. Recall that a quadratic differential in the Riemann sphere $
\widehat{\mathbb{C}}$ is given by an expression $\varpi (z)dz^{2}$, where $\varpi $ is a
rational function. By \textit{parametric quadratic differentials, }we mean a family of such
differentials that depend on extra data. Let $\mathcal{D}$ be the set of these
data.\textit{\ }We treat the case 
\begin{equation}\label{our quadratic}
    \dsp\varpi _{a,\lambda }(z) =-\lambda ^{2}(z-1)(z+1)(z-a)=-\lambda^{2}p_{a}(z); \quad
\mathcal{D} \mathcal{=}\left\{ \left( \theta ,a\right) \in \lbrack 0,\pi
\lbrack \times \overset{\bullet }{\,\,\mathbb{C}}\,\,\right\} ; 
\end{equation}

where $\lambda =r\exp (i\theta )\in \mathbb{C}^{\ast }$.\\
\newline
Critical points of $\varpi _{a,\lambda }$ are its zero's and poles in $
\widehat{\mathbb{C}}$. Zeros are called \emph{finite critical points}, while poles of order 2
or greater are called \emph{infinite critical points}. All other points of $\widehat{\mathbb{C}}$ are called \emph{regular points}. Finite critical points of the quadratic
differential $\varpi _{a,\lambda }$ are $\pm 1$ and $a$, as simple zeros;
while, by the change of variable $y=1/z$, $\varpi _{a,\lambda }$ has a
unique infinite critical point that is located at $\infty $ as a pole of
order $7$. \emph{Vertical trajectories }of the quadratic differential $\varpi_{a,\lambda }$ are the level curves defined by 
\begin{equation}
\mathcal{\Re }\int^{z}e^{i\theta }\sqrt{p_{a}(z)}dt=\emph{const;}
\label{eq integ}
\end{equation}
or equivalently

\begin{equation*}
\exp (2i\theta )p_{a}(z)dz^{2}<0.
\end{equation*}
If $z\left( t\right) ,t\in \mathbb{R}$ is a vertical trajectory, then the function 
\begin{equation*}
t\longmapsto \Im  \int^{t}e^{i\theta }\sqrt{p_{a}\left( z\left( s\right)
\right) }z^{\prime }\left( s\right) ds
\end{equation*}
is monotonic.\\
\newline
The \emph{horizontal trajectories }are obtained by replacing $\Re $ by $\Im  $
in equation \eqref{eq integ}. Horizontal and vertical trajectories of the
quadratic differential $\varpi _{a,\lambda }$ produce two pairwise
orthogonal foliations of the Riemann sphere $\widehat{\mathbb{C}}$. A trajectory (horizontal or vertical) passing through a critical point
of $\varpi _{a,\lambda }$ is called \emph{critical trajectory}. In
particular, if it starts and ends at a finite critical point, it is called 
\emph{finite critical trajectory} or \emph{short trajectory}, otherwise, we
call it an \emph{infinite critical trajectory}.

\emph{Stokes lines }of \eqref{first ODE} are critical vertical trajectories
of $\varpi _{a,\lambda }$ defined by the equation
\begin{equation*}
\mathcal{\Re }\int_{z_{0}}^{z}e^{i\theta }\sqrt{p_{a}(z)}dt=\emph{0}\text{ \
where }z_{0}\in \left\{ -1,1,a\right\} \emph{;}
\end{equation*}
and \emph{anti-Stokes lines }are critical horizontal trajectories of $\varpi
_{a,\lambda }$. The closure of the set of finite and infinite critical
trajectories, that we denote by $\Gamma _{a,\lambda }$ is called the \emph{
critical graph, }or \emph{Stokes graph}. The local and global structure of
the trajectories are studied in \cite{strebel, jenkins}.

\begin{remark}
\begin{enumerate}
\item By a linear transformation, we can reduce any cubic quadratic
differential $\varpi (z)=A(z-z_{0})(z-z_{1})(z-z_{2})$ to the form \eqref{our quadratic}.

\item From \eqref{eq integ}, the critical graph of a quadratic differential
is invariant under multiplication by a positive scalar, so we use the
notations $\varpi _{a,\lambda }(z)dz^{2}=\varpi _{a,\theta }(z)dz^{2}$ and $
\Gamma _{a,\lambda }=\Gamma _{a,\theta }$.

\item By the relations
\begin{align}
\varpi _{a,\theta }(z)dz^{2}& >0\Longleftrightarrow \varpi _{\theta +\frac{
\pi }{2},-a}(z)dz^{2}>0,  \label{symmetric relations} \\
\varpi _{a,\theta }(z)dz^{2}& >0\Longleftrightarrow \varpi _{\frac{\pi }{2}
-\theta ,-\overline{a}}(z)dz^{2}>0,  \notag
\end{align}
we may focus on the case $\theta \in $ $[0,\frac{\pi }{4}]$ without the loss of  generality.

\item The \emph{critical (orthogonal critical) directions }are given by
\begin{equation}
\begin{array}{c}
\alpha _{j}=\frac{2j\pi -2\theta }{5} \\ 
\alpha _{j}^{\bot }=\frac{(2j-1)\pi -2\theta }{5}
\end{array}
;\,\,j=0,...,4  \label{critical directions}
\end{equation}
There is a neighborhood $\mathcal{V}$ of $\infty $, such that each
horizontal (vertical) trajectory entering $\mathcal{V}$ stays in $\mathcal{V}
$ and tends to $\infty $ following one of $\alpha _{j}$ $\left( \alpha
_{j}^{\perp }\right) $.
\end{enumerate}
\end{remark}
The structure of the set $\widehat{\mathbb{C}}\setminus \Gamma _{a,\theta }$ depends on the local and global behaviors of
trajectories. It consists of a finite number of domains called\emph{\ domain
configurations, }or \emph{Stokes regions} of $\varpi _{a,\theta }$. Jenkins
Theorem (see \cite[Theorem3.5]{jenkins},\cite{strebel}) asserts that there
are two kinds of Stokes regions of $\varpi _{a,\theta }$ :

\begin{itemize}
\item Half-plane domain : It is swept by trajectories diverging to $\infty$
in its two ends, and along consecutive critical directions. Its boundary
consists of a path with two unbounded critical trajectories, and possibly a
finite number of short ones. It is conformally mapped to a vertical half
plane $\left\{ w\in\mathbb{C}:\Re w>c\right\} $ for some real $c$ by the function $\int_{z_{0}}^{z}e^{i
\theta}\sqrt{p_{a}\left( t\right) }dt$ with suitable choices of $z_{0}\in\mathbb{C}$ and the branch of the square root;

\item strip, or band domain : It is swept by trajectories which both ends
tend $\infty $. Its boundary consists of a disjoint union of two paths, each
of them consisting of two unbounded critical trajectories diverging to $
\infty $, and possibly a finite number of short trajectories. It is
conformally mapped to a vertical strip $\left\{ w\in \mathbb{C}:c_{1}<\Re w<c_{2}\right\} $ for some reals $c_{1}$ and $c_{2}$ by the
function $\int_{z_{0}}^{z}e^{i\theta }\sqrt{p_{a}\left( t\right) }dt$ with
suitable choices of $z_{0}\in \mathbb{C}$ and the branch of the square root.
\end{itemize}

\begin{remark}
\begin{enumerate}
\item The Stokes regions are unbounded and simply connected. We always order the half-plane domains \textit{\ }$(H_{k})_{0\leq k\leq 4}$ and  
\textit{strip domain }$(B_{k})_{0\leq k\leq 1}$ anticlockwise \textit{.} In the general case if $\deg p_{a}=d$ in \eqref{first ODE}, then the Stokes lines divide the $z$-plane into $d+2$ domains of half-plane type and $N$ strip domains such that $0\leq N\leq d-1$. The connected components of the Stokes graph are called the \emph{Stokes complexes }denoted $\Delta _{a,\theta }\mathit{.}
$

\item A big challenge in asymptotic theory of ODE in complex domains (in a
Riemann surface more generally) is to study Stokes graph mutations under
changes of parameters in the data space $\mathcal{D}$, see \cite{iwaki, chouikhiquartic, solynin, chouikhialgebraic}. In our work we
treat the cubic polynomial case. When the potential is a meromorphic
function, the situation is more complicated, see again \cite{solynin, chouikhialgebraic}.

\item The table below give a correspondance between the language used in the
theory of quadratic differential and the language used in WKB asymptotic
theory of ODE
\begin{equation*}
\begin{tabular}{|l|l|}
\hline
\textbf{Q.D language} & \textbf{complex} \textbf{WKB language} \\ \hline
finite critical point & turning point \\ \hline
infinite critical point & singular point \\ \hline
horizontal critical trajectory & anti-Stokes line \\ \hline
vertical critical trajectory & Stokes line \\ \hline
short trajectory & finite anti-Stokes line \\ \hline
short vertical trajectory & finite Stokes line \\ \hline
critical graph & Stokes graph \\ \hline
domain configuration & Stokes region \\ \hline
critical direction & anti-Stokes direction \\ \hline
orthogonal critical direction & Stokes direction \\ \hline
\end{tabular}
\end{equation*}
In particular for the traditional WKB analysis, turning points are only
zeros of the potential (quadratic differential) but in \cite{koike}, it is
revealed that simple poles also play a similar role as turning points which is clearly noted in our table.
\end{enumerate}
\end{remark}

\subsubsection{Level Sets\label{level sets}}

For $\theta \in \left[ 0,\pi /2\right[ $, we consider the sets
\begin{align*}
\Sigma _{1,\theta }& =\left\{ a\in\mathbb{C}\setminus \left] -\infty ,-1\right] :\Re \left( \int_{\left[ 1,a\right]
}e^{i\theta }\sqrt{p_{a}\left( z\right) }dz\right) =0\right\} ; \\
\Sigma _{-1,\theta }& =\left\{ a\in \mathbb{C}\setminus \left[ 1,+\infty \right[ :\Re \left( \int_{\left[ -1,a\right]
}e^{i\theta }\sqrt{p_{a}\left( z\right) }dz\right) =0\right\} ; \\
\Sigma _{\blacktriangle ,\theta }& =\left\{ a\in \mathbb{C}
\setminus \left[ -1,1\right] :\Re \left( \int_{\left[ -1,1\right]
}e^{i\theta }\sqrt{p_{a}\left( z\right) }dz\right) =0\right\} ; \\
\Sigma _{\theta }& =\Sigma _{\blacktriangle ,\theta }\cup \Sigma _{-1,\theta
}\cup \Sigma _{1,\theta };
\end{align*}
A full description of the three sets is given in \cite[Lemma 1 and
Proposition 2]{Thabet+al} for $\theta \in \left[ 0,\pi /2\right[ $, in
particular the set $\Sigma _{\theta }$ divide the complex plane $\mathbb{C}$ into $n_{\theta }$ simply connected domains $\left( \Omega _{i}\right)
_{1\leq i\leq n_{\theta }}$ and by \eqref{symmetric relations} we have for $\theta \in \left[ 0,\frac{\pi }{4}\right] $:

\begin{itemize}
\item If $\theta \neq 0$, $\Sigma _{1,\theta }$ is formed by three smooth curve $\Sigma _{1,\theta }^{\prime }$, $\Sigma _{1,\theta }^{\prime \prime }$
and $\Sigma _{1,\theta }^{\blacktriangle }$. The two first curves are locally
orthogonal at $z=1$ and 
\begin{eqnarray*}
\Sigma _{1,\theta }^{r} &=&\Sigma _{1,\theta }^{\prime }\cap \mathbb{C}_{+}=\left\{ z\in \Sigma _{1,\theta };\,\,\Im z>0,\Re z\geq
1\right\} \\
\Sigma _{1,\theta }^{l} &=&\Sigma _{1,\theta }^{\prime \prime }\cap \mathbb{C}_{+}=\left\{ z\in \Sigma _{1,\theta };\,\,\Im z>0,\Re z\leq
1\right\} .
\end{eqnarray*}
The third curve $\Sigma _{1,\theta }^{\blacktriangle }$ starting from some
point $s_{1,\theta }^{\blacktriangle }\in \left] -\infty ,-1\right] $. The
five rays defining $\Sigma _{1,\theta }$ diverge to $\infty$ in different directions (see Figure (\ref{theta=pi/4 1,-1})).
\begin{figure}[h]
    \centering
    \includegraphics[width=0.4\linewidth]{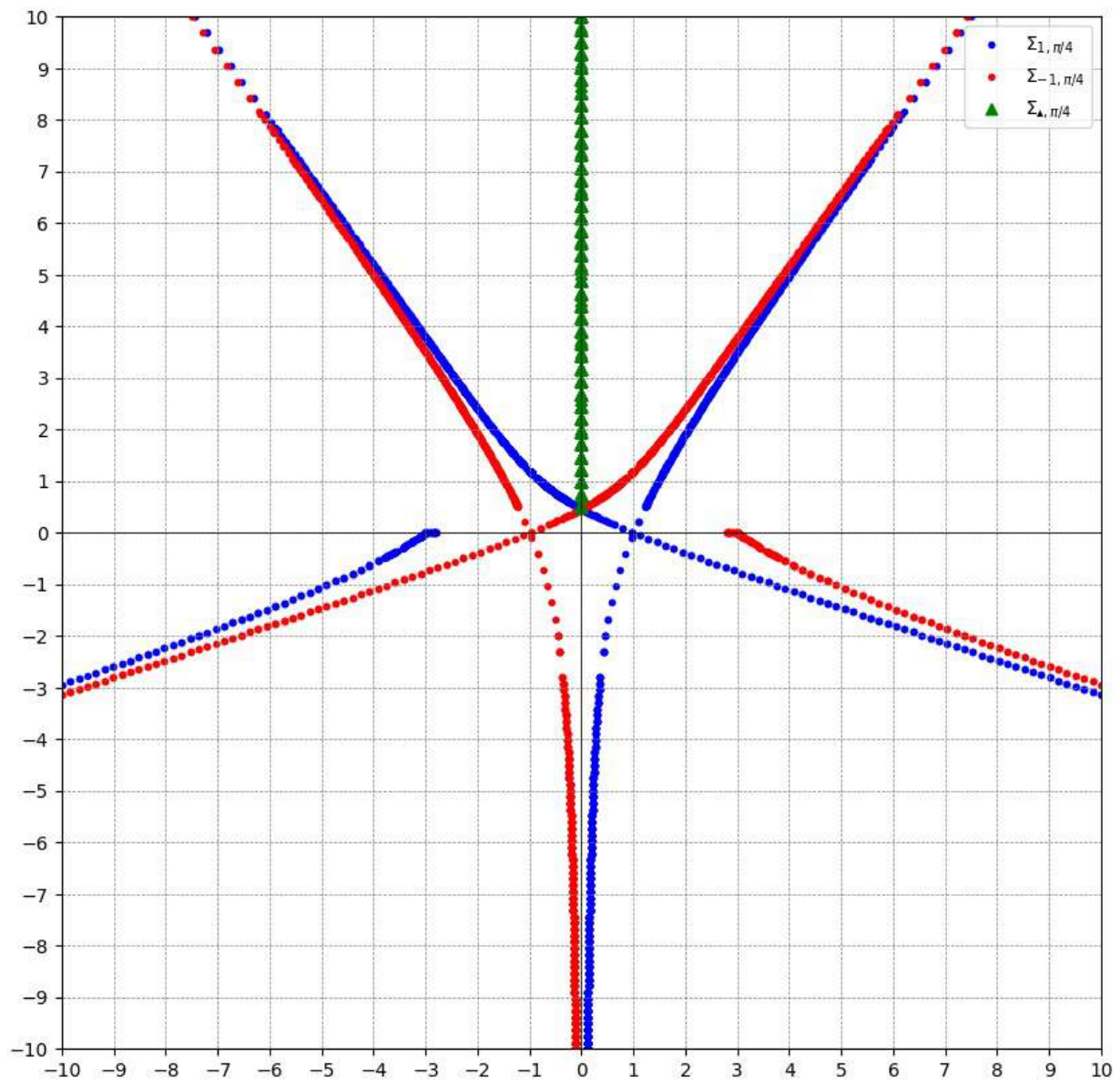}
    \caption{Exact plot of Sigma curves for $\theta=\frac{\pi}{4}$}
    \label{theta=pi/4 1,-1}
\end{figure}

\item If $\theta \neq 0$, $\Sigma _{-1,\theta }$ is formed by three smooth
curves $\Sigma _{-1,\theta }^{\prime }$, $\Sigma _{-1,\theta }^{\prime \prime}$ and $\Sigma _{-1,\theta }^{\blacktriangle }$. The two first curve are locally orthogonal at $z=-1$ and 
\begin{eqnarray*}
\Sigma _{-1,\theta }^{r} &=&\Sigma _{-1,\theta }^{\prime }\cap 
\mathbb{C}_{+}=\left\{ z\in \Sigma _{-1,\theta };\,\,\Im z>0,\Re z\geq
-1\right\} \\
\Sigma _{-1,\theta }^{l} &=&\Sigma _{-1,\theta }^{\prime \prime }\cap
\mathbb{C}_{+}=\left\{ z\in \Sigma _{-1,\theta };\,\,\Im z>0,\Re z\leq
-1\right\} .
\end{eqnarray*}
The third curve $\Sigma _{-1,\theta }^{\blacktriangle }$ starting from some point $s_{-1,\theta }^{\blacktriangle }\in \left[ 1,+\infty \right[ $. The five rays defining $\Sigma _{-1,\theta }$ diverge to $\infty$ in different directions (see Figure  \ref{theta=pi/4 1,-1}).

\item If $\theta =0$, $\Sigma _{-1,0}$ is formed only by $\Sigma _{-1,0}^{'}$
and $\Sigma _{-1,0}^{''}$ while $\Sigma _{1,0}^{\blacktriangle }$ is formed by two smooth curves starting from $s_{1,0}^{\blacktriangle }=-1$ and symmetric with real axis (see Figure \ref{theta zeros}). 
\begin{figure}[h]
    \centering
    \includegraphics[width=0.4\linewidth]{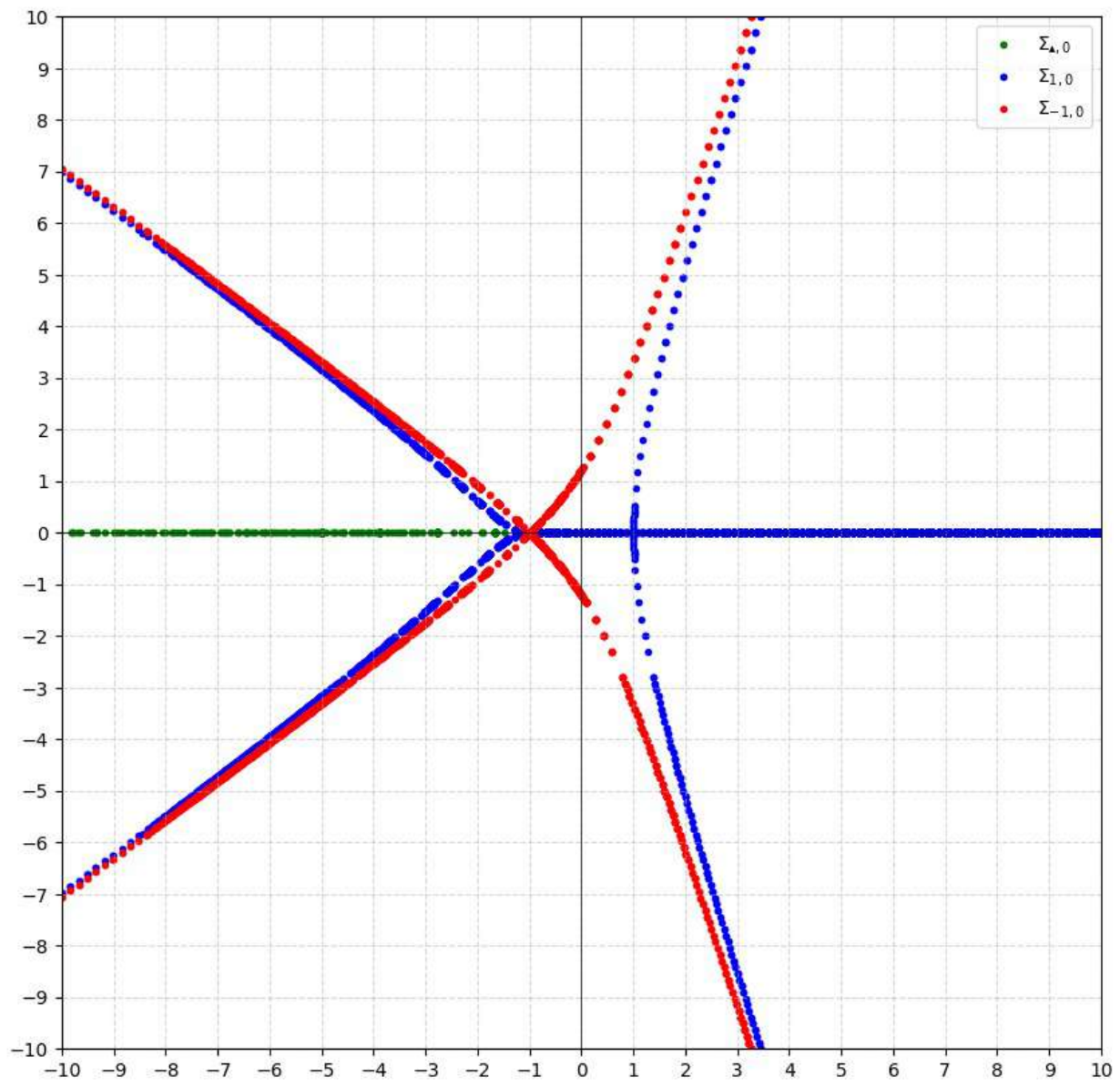}
    \caption{Excat plot of Sigma curves for $\theta=0$}
    \label{theta zeros}
\end{figure}

\item 
\begin{equation*}
\begin{tabular}{|l|l|l|l|l|}
\hline
& $\Sigma _{1,\theta }^{l}\cap \Sigma _{-1,\theta }^{r}$ & $\Sigma
_{1,\theta }^{l}\cap \Sigma _{-1,\theta }^{l}$ & $\Sigma _{1,\theta
}^{r}\cap \Sigma _{-1,\theta }^{r}$ & $\Sigma _{1,\theta }^{r}\cap \Sigma
_{-1,\theta }^{l}$ \\ \hline
$\theta \in \left] 0,\pi /8\right[ $ & $\left\{ t_{\theta }\right\} $ & $
\left\{ e_{\theta }\right\} $ & $\emptyset $ & $\emptyset $ \\ \hline
$\theta \in \left[ \pi /8,\pi /4\right] $ & $\left\{ t_{\theta }\right\} $ & 
$\emptyset $ & $\emptyset $ & $\emptyset $ \\ \hline
\end{tabular}
\ \ \ \ \ \ \ \ \ \ \ \ 
\end{equation*}
where $t_{\theta }$ and $e_{\theta }$ are two complex number varying respectively in two smooth curves.

\item The set $\Sigma _{\blacktriangle ,\theta }$ is a smooth curve included in the part of the upper half plane of the strip bounded by the segment $\left[ -1,1\right] $ and the lines $y=-\tan \left( 2\theta \right) \left(x\pm 1\right) $. It goes through $t_{\theta }$ and $e_{\theta }$ (if they exist); its two rays diverge  to $\infty $ in different directions, following the direction $\arg a=\pi -2\theta $, and it connects to the unique point $s_{\blacktriangle
,\theta }\in \left[ -1,0\right] $, such that 
\begin{equation*}
\Re \int_{-1}^{1}e^{i\theta }\left( \sqrt{p_{s_{\blacktriangle ,\theta
}}\left( t\right) }\right) _{+}dt=0.
\end{equation*}
In particular 
\begin{equation*}
\dsp\Sigma _{1,\theta }\cap \Sigma _{-1,\theta }\cap \left\{ a\in 
\mathbb{C}:\Im  a<0\right\} =\emptyset .
\end{equation*}
\begin{figure}[th]
    \centering
    \includegraphics[width=0.4\linewidth]{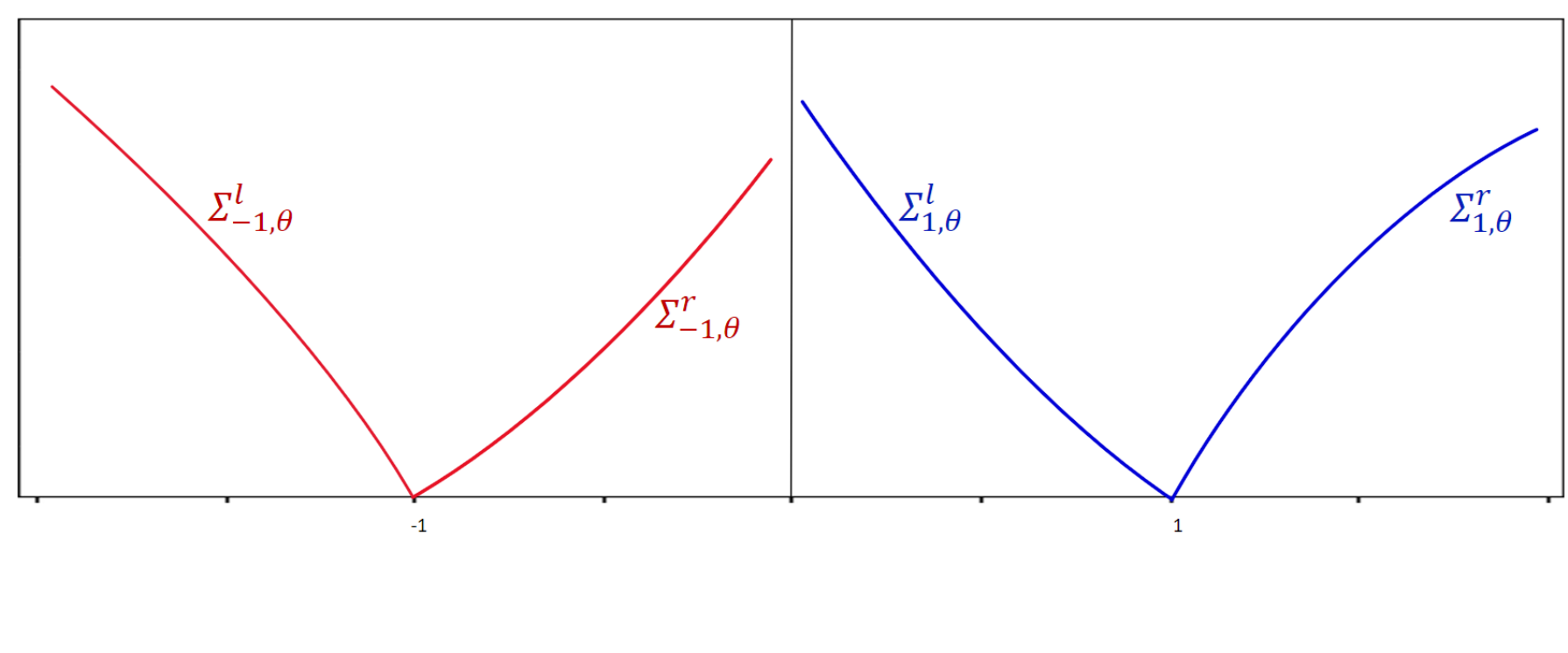}
    \caption{the orthogonal curves at $\pm1$; $\Sigma'_{\pm1,\theta}$}
    \label{fig orthogonal curves}
\end{figure}
\begin{figure}[th] 
  \begin{minipage}[b]{0.5\linewidth}\centering\includegraphics[width=0.5\textwidth]{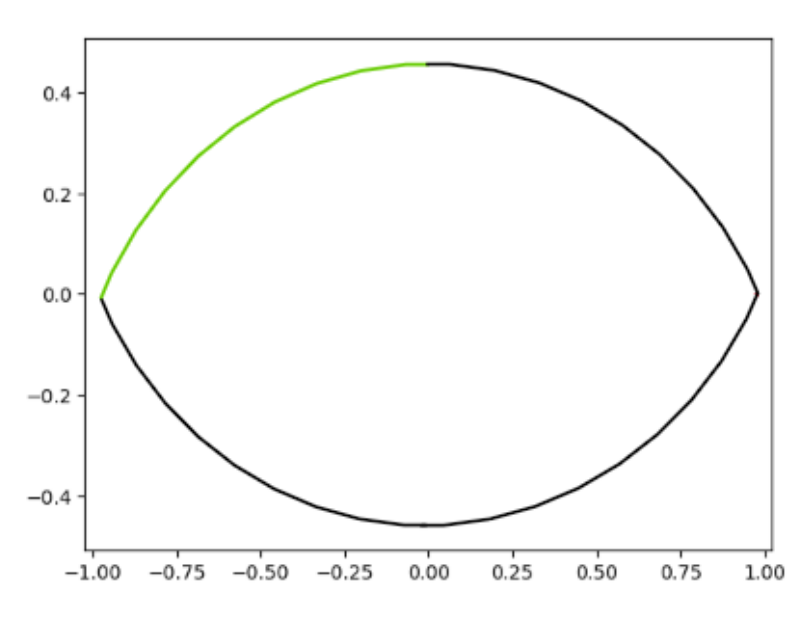}
\caption{Sets of $e_\theta$}
  \end{minipage}\hfill   \begin{minipage}[b]{0.5\linewidth}
  \centering
  \includegraphics[width=0.4\textwidth]{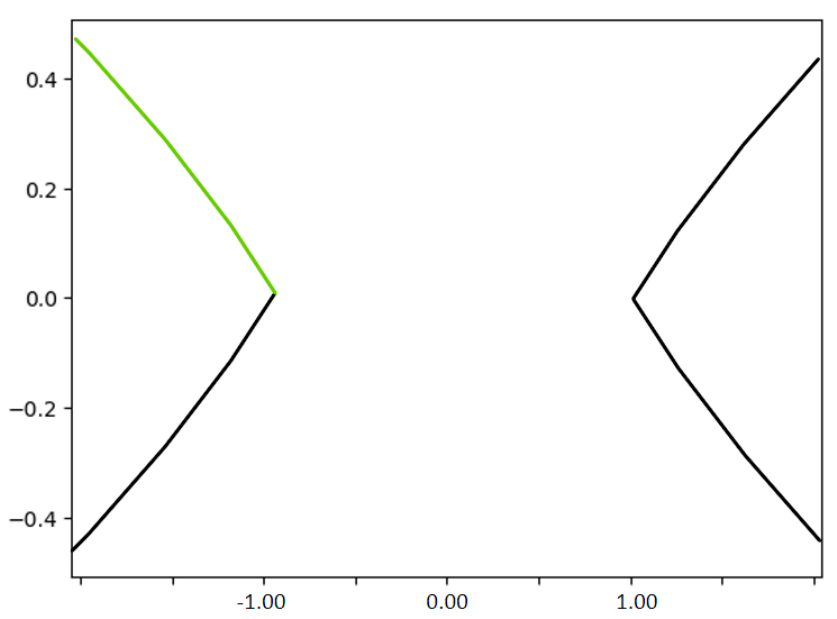}
\caption{Sets of $t_\theta$}
  \end{minipage}
\end{figure} 

\end{itemize}

\subsubsection{$\mathcal{D\diagup SG}$ correspondence}

Let 
\begin{equation}\label{chi theta}
\mathcal{\chi }_{\theta }=\left\{ a\in \overset{\bullet }{
\mathbb{C}};\,\,-\exp (2i\theta )\left( z-a\right) \left( z^{2}-1\right) dz^{2}
\text{ admit at least a short trajectories}\right\} ;
\end{equation}
by \eqref{symmetric relations} the set of data $\mathcal{D=}\left\{ \left(
\theta ,a\right) \in \lbrack 0,\frac{\pi }{2}[\times \overset{\bullet }{
\mathbb{C}}\,\,\right\} $. The study of the mutations of Stokes graphs as extra parameters vary in $\mathcal{D}$ is a highly attractive subject in
mathematical physics, see \cite{{Thabet+al}, {chouikhiquartic},
{iwaki}, {solynin}}.\\
Let $\theta \in \lbrack 0,\frac{\pi }{2}[,$

\begin{notation}
\label{Notations sigma}We denote by

\begin{itemize}
\item $\mathcal{S}_{1,\theta }:$ the sets of $\Sigma _{1,\theta }$ thet goes through $z=1$ minus the arc starting at $e_{\theta }$ (if it exists),
and diverging to $\infty $ in the upper half plane.
\item $\mathcal{S}_{-1,\theta }:$ the sets of $\Sigma _{-1,\theta }$ that goes through $-1$.
\item $\mathcal{S}_{\blacktriangle ,\theta }:$ the part of $\Sigma
_{\blacktriangle ,\theta }$ starts at $t_{\theta }$ and diverges to $
\infty$.

\item It was shown in \cite{Thabet+al} that $\mathcal{\chi }_{\theta }=
\mathcal{S}_{\pm 1,\theta }\cup \mathcal{S}_{\blacktriangle ,\theta }$.

\item $n_{\theta }:$ the finite number ($n_{0}=8,n_{arct\left( 0.5\right)
/2}=10,$and $n_{\pi /4}=9$) of the connected components $\Omega
_{1},...,\Omega _{n_{\theta }}$ of $\mathbb{C}\setminus \mathcal{\chi }_{\theta }$. 
\begin{figure}[ht]
    \centering
    \includegraphics[width=0.35\linewidth]{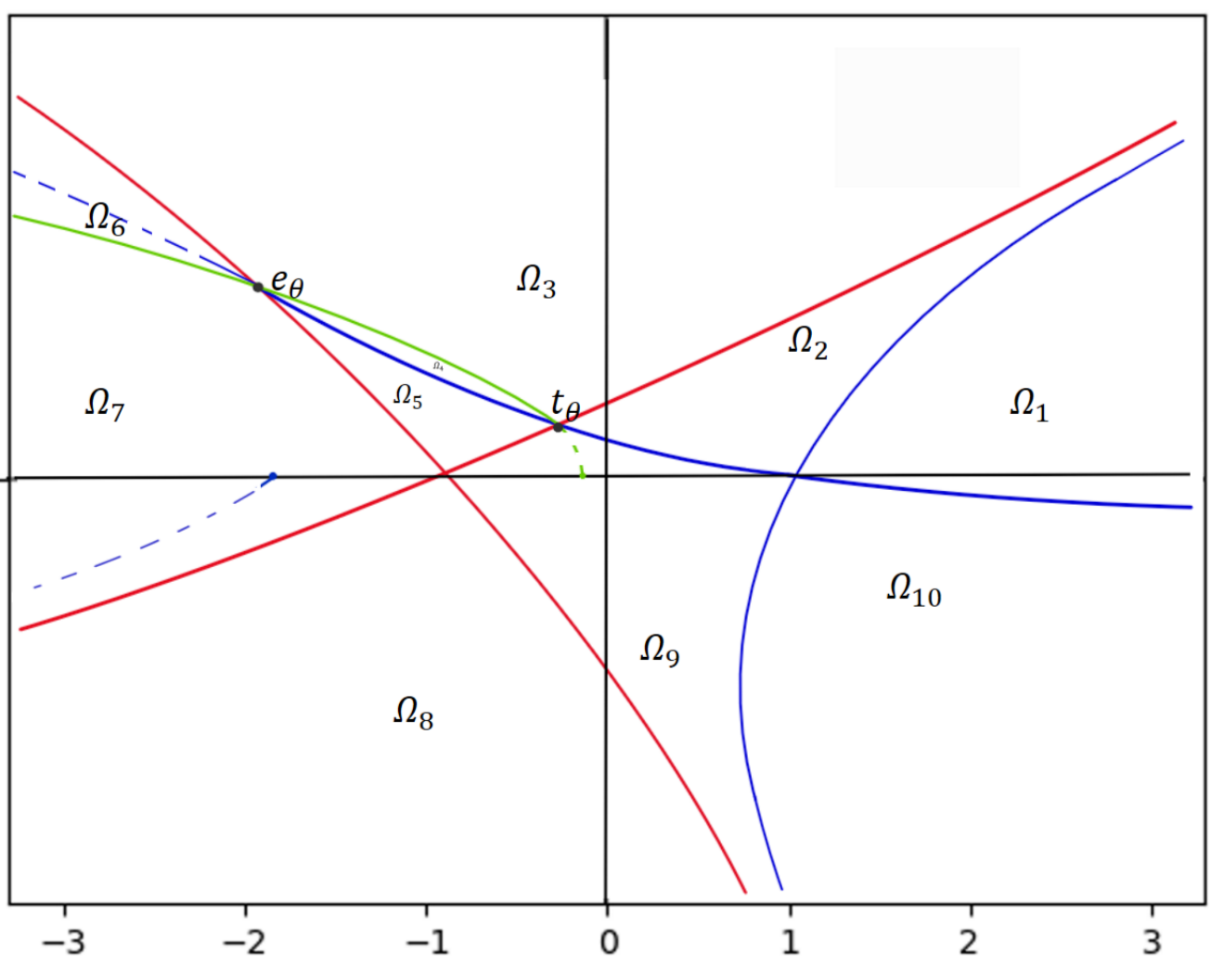}
    \caption{ approximate plot of $\mathcal{\chi }_{\theta }
\text{ for }\theta =\frac{arctan(0.5)}{2}$}
    \label{fig:enter-label}
\end{figure}
\end{itemize}
\end{notation}

\begin{definition}
We say that two Stokes graphs $\Gamma _{a,\theta }$ and $\Gamma _{b,\theta }$
have the same structure in $
\mathbb{C}
\setminus \mathcal{\chi }_{\theta }$ if for any critical trajectory of $
\varpi _{a,\theta }$ emerging from $z=-1$, (resp. $z=+1$, $z=a$), there exists a correspondent critical trajectory of $\varpi _{b,\theta }$ that
emerges from $z=-1$ (resp. $z=+1$, $z=b$), and diverges to $\infty $
following the same critical direction.
\end{definition}
We introduce the equivalence relation (see \cite[Proof of Theorem 9]
{Thabet+al}) in $\mathbb{C}\setminus \mathcal{\chi }_{\theta }$ 
\begin{equation*}
a\mathcal{R}^{\prime }b\text{ if and only if, critical graphs }\Gamma
_{a,\theta }\text{ and }\Gamma _{b,\theta }\text{ have the same structure}
\end{equation*}
The main result that we need from \cite{Thabet+al} is the following
theorem, which give a classification of Stokes graph in $\dsp\mathbb{C}\setminus \mathcal{\chi }_{\theta }$:

\begin{theorem}[{\protect\cite[Theorem 9]{Thabet+al}}]
\label{main theorem Thabet}In each of the domains $\Omega _{1},...,\Omega
_{n_{\theta }}$, the critical graph $\Gamma _{a,\theta }$ of the quadratic
differential $\varpi _{a,\theta }$ has the same structure; it splits the
Riemann sphere into five half-plane and two strip domains. Moreover, $\Gamma_{a,\theta }$:

\begin{enumerate}
\item has a short trajectory connecting $z=\pm 1$ if and only if $a\in 
\mathcal{S}_{\blacktriangle ,\theta };$ it has a short trajectory connecting 
$z=\pm 1$ to $z=a$ if and only if $a\in \mathcal{S}_{\pm 1,\theta }$. In all these cases, $\Gamma _{a,\theta }$ splits the Riemann sphere into five
half-plane domains and exactly one strip domain.

\item It has a tree (juxtaposition of two short trajectories) with summit $t_{\theta };$ it has a tree with summit $z=-1$ or $z=+1$ respectively for $\theta \in $ $\left[ 0,\pi /8\right[ $ and $\theta \in $ $\left[ 3\pi /8,\pi/2\right[ $. In these cases, $\Gamma _{a,\theta }$ splits the Riemann sphere into five half-plane domains.
\end{enumerate}

In particular, any change of a critical graph structure should pass by a
critical graph with at least one short trajectory.
\end{theorem}

As a result of this theorem, we obtain a full classification of the Stokes graphs for $\theta \in \lbrack 0,\frac{\pi }{2}[$ :

\begin{itemize}
\item Stokes graph of \textbf{type A}: it has no short trajectory and two
strip domains $B_{0}$ and $B_{1}$ (the number of strip domains is maximal). It depends on the location of $a$ in one of the domains $\Omega _{i}$. In fact, we have $\dsp\left(\mathbb{C}\setminus \mathcal{\chi }_{\theta }\right) \diagup \mathcal{R}^{\prime }=
\underset{i=1}{\overset{n_{\theta }}{\bigcup }}\Omega _{i}$ (see figure \ref{typeA}).

\item Stokes graph of \textbf{type B}: It has one short trajectory $\ell $
and one strip domain $B_{0}$ such that $\ell \subset \partial B_{0}$ (see figure \ref{typeB}).

\item Stokes graph of \textbf{type BB}: It has one short trajectory $\ell $ and one strip domain $B_{0}$ such that $\ell \cap \partial B_{0}$ is reduced to a turning point (see figure \ref{typeBB}).

\item A \textbf{Tree }(or \textbf{Boutroux) }Stokes graph: It has a broken
short trajectory with summit at a turning point, and no strip domain occurs. This curve can appear only for $a\in $\textbf{\ }$\left\{ t_{\theta
},e_{\theta }\right\} $  (see figure \ref{typeT}).
\end{itemize}
\begin{figure}[tbh]
\begin{minipage}[b]{0.3\linewidth}
		\centering\includegraphics[scale=0.3]{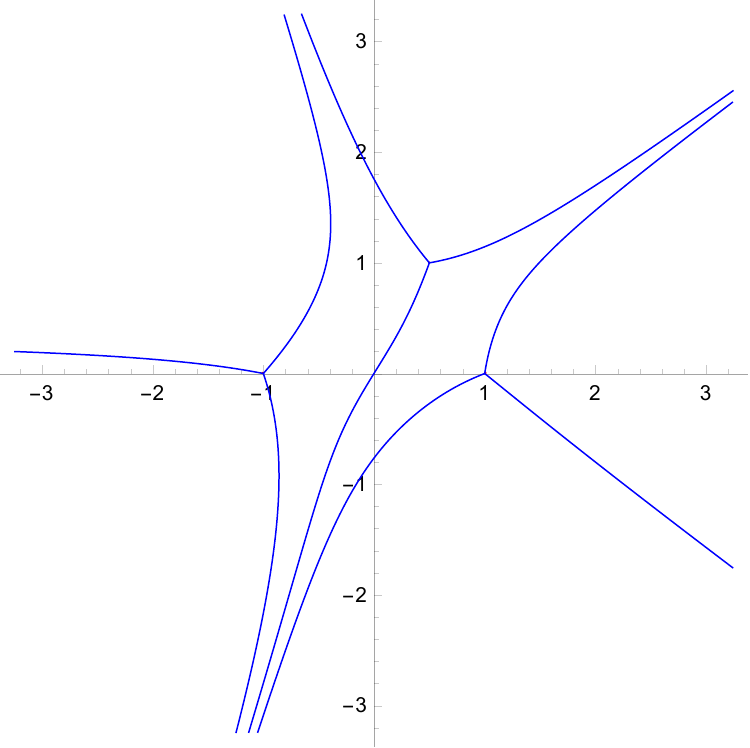}
		\caption{Stokes graph of \textbf{type A}; $(a\notin\chi_{\theta})$}
        \label{typeA}
\end{minipage}\hfill\begin{minipage}[b]{0.3\linewidth}
	\centering\includegraphics[scale=0.30]{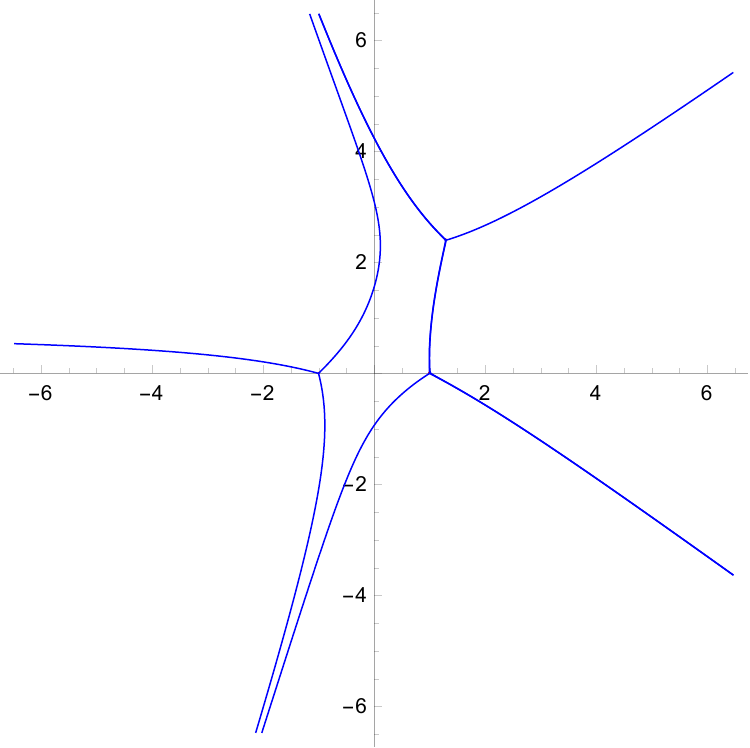}
	\caption{Stokes graph of \textbf{type B}; $(a\in\mathcal{S}_{1,\theta })$}
    \label{typeB}
\end{minipage}\hfill\begin{minipage}[b]{0.3\linewidth}
			\centering\includegraphics[scale=0.30]{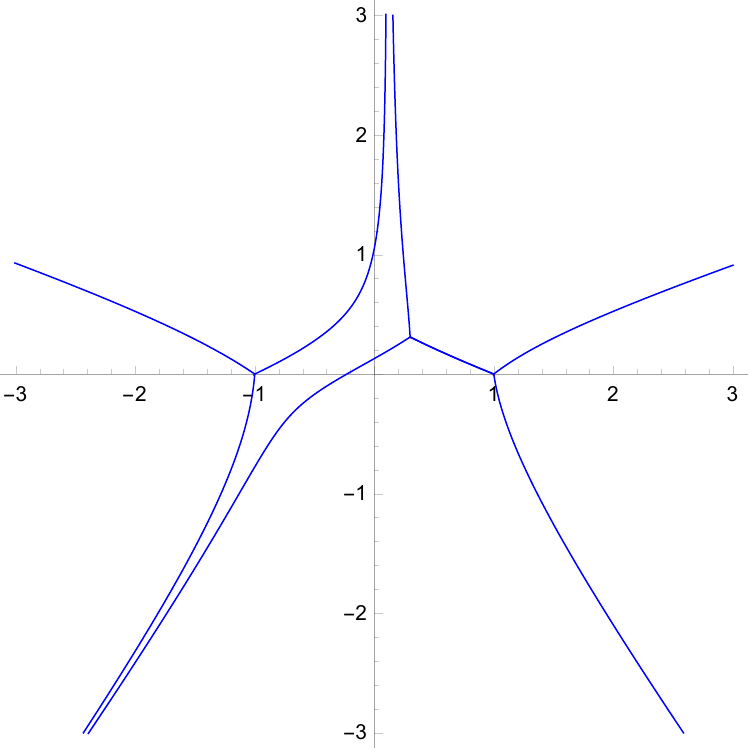}
			\caption{Stokes graph of \textbf{type BB}; $(a\in\mathcal{S}_{1,\theta })$}
            \label{typeBB}
\end{minipage}\hfill
\end{figure}\label{Figures7,8,9}
\begin{figure}[h]
	\centering\includegraphics[scale=0.30]{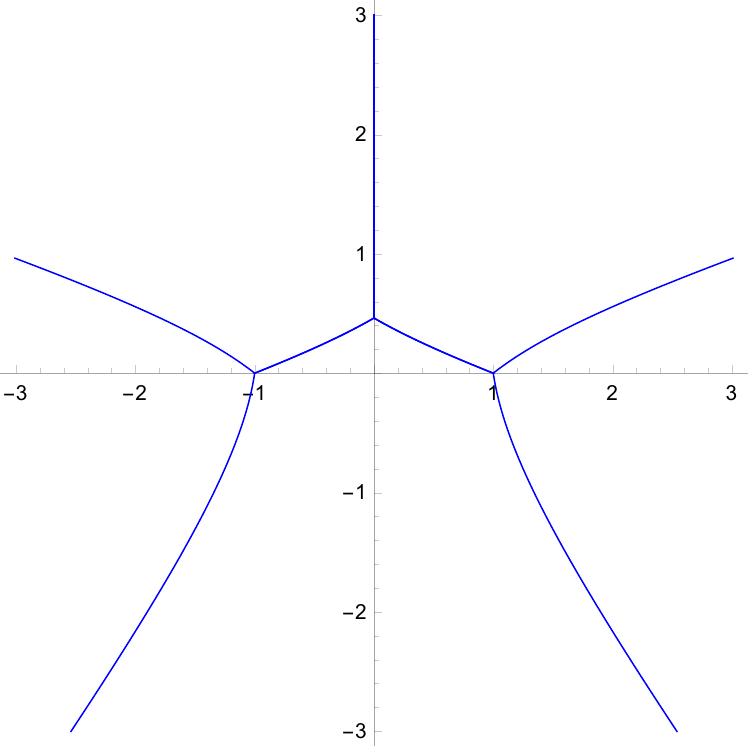}
	\caption{Stokes graph of type \textbf{Tree}; $(a=t_{\theta}, \theta\in[\frac{\pi}{8},\frac{3\pi}{8}[)$}
    \label{typeT}
\end{figure}\label{Figure 10}
\begin{corollary}
As $a$ varies in $\mathcal{\chi }_{\theta }$ any change in the  critical graph should pass as $a$ crosses one of the points $\left\{ -1,1,t_{\theta
},e_{\theta }\right\} $.
\end{corollary}

\begin{proof}
Suppose for simplicity that $a$ $\in \Sigma _{1,\theta }^{r}\subset \Sigma
_{1,\theta }^{\prime }\subset \mathcal{\chi }_{\theta }$, so there exists a short trajectory $\ell $ connecting $a$ to $1$ and one strip domain $B_{0}$. It is clear that $\Sigma _{1,\theta }^{r}$ $\in \partial \Omega _{i}\cap\partial \Omega _{j}$ for some $i\neq j$. Since the structure of the critical graph is unchangeable in $\Omega _{i}$ and $\Omega _{j}$ by Theorem  \ref{main theorem Thabet}, so $\partial B_{0}$ is invariant (just the length of $\ell $ change) and the Stokes graph is of \textbf{type BB}. As $a$ crosses $1$ staying in $\Sigma _{1,\theta }^{\prime }$, at least one of the domains, say $\Omega _{j}$, changes to a third domain $\Omega _{k}$ where $k\neq i,j$. The strip domain $B_{0}$ is now delimited by the short trajectory $\ell $ and so the Stokes graph changes the structure to \textbf{type B}.

\end{proof}

We can now extend the relation $\mathcal{R}^{\prime }$ to an equivalence
relation $\mathcal{R}$ defined over $\overset{\bullet }{
\mathbb{C}}$, and we obtain $\dsp\overset{\bullet }{\mathbb{C}}\diagup \mathcal{R=}\underset{i=1}{\overset{n_{\theta }}{\bigcup }}\Omega
_{i}\bigcup \left( \mathcal{\chi }_{\theta }\diagdown \left\{ -1,1\right\}
\right) $. Let 
\begin{align*}
\left[ \Gamma _{a,\theta }\right] & =\left\{ \Gamma _{b,\theta };\,\,a
\mathcal{R}b\,\,\right\} ; \\
\mathcal{SG}& \mathcal{=}\left\{ \left[ \Gamma _{a,\theta }\right] ;\,\,
\theta \in \lbrack 0,\frac{\pi }{2}[\right\} ;
\end{align*}
$\mathcal{SG}$ is called the \emph{Stokes Geometry} set. The $\mathcal{
D\diagup SG}$ correspondence is expressed via the map
\begin{equation*}
\begin{array}{cc}
\mathcal{J}:\mathcal{D\longrightarrow SG} & \,\,\left( \theta ,a\right)
\longmapsto \left[ \Gamma _{a,\theta }\right]
\end{array}
\end{equation*}

\begin{proposition}[\protect\cite{Thabet+al}]
The map $\mathcal{J}$ is surjective.
\end{proposition}

\subsection{Short review of complex WKB\label{sub2}}

In this section, we give a short review of some ingredients and notations in
complex WKB, for more details, see \cite[\S 3 ]{fed1}.

\subsubsection{Canonical domains and asymptotic expansions\label{topology Stokes}}

Let $\Upsilon =\left\{ -1,1,a\right\}$, the set of zeros of $p_{a}$ be called
the turning points associated with the ODE \eqref{first ODE}, and $\lambda =r\exp(i\theta )\in\mathbb{C}^{\ast }$ with $\theta \in \lbrack 0,\frac{\pi }{2}[$. The multi-valued
function 
\begin{equation*}
h(\lambda ,z_{0},z)=\lambda \int_{z_{0}}^{z}\sqrt{(t-a)(t^{2}-1)}dt,z_{0}\in
\Upsilon 
\end{equation*}
has holomorphic branches in each Stokes region. A domain $D$ in the complex
$z$-plane is called canonical if there is a holomorphic branch of $h(\lambda,z_{0},z)$ in $D$ that maps $D$ to the whole complex plane with a finite number of vertical cuts. The domain $D$ is simply connected and contains no turning points and $\partial D$ consists of Stokes lines (the preimages of the sides of the cuts). A canonical domain is the union of two domains of half-plane type and up to two strip-domains.\\
\newline
Let $\varepsilon >0$. We denote $D_{\varepsilon }$ for the preimage of $h(D)$
with $\varepsilon $-neighborhoods of the cuts and $\varepsilon $
-neighborhoods of the turning points removed. A canonical path in $
D_{\varepsilon }$ is a path such that $\Re(h(\lambda ,z_{0},z))$ is
monotone along the path. For example, the anti-Stokes
lines are canonical paths. For every point $z$ in $D_{\varepsilon }$, there are always canonical paths $\gamma _{\varepsilon }^{+}(z)$ and $\gamma_{\varepsilon }^{-}(z)$ from $z$ to $\infty $, such that $\Re
(h(\lambda ,z_{0},z))\rightarrow +\infty $ and $\Re(h(\lambda
,z_{0},z))\rightarrow -\infty $, respectively.\\
\newline
With $D_{\varepsilon }$, $\gamma _{\varepsilon }^{-}$ and $\gamma
_{\varepsilon }^{+}$ as above, there exist $r_{\varepsilon }>0$ such that for $r=\left\vert \lambda \right\vert >r_{\varepsilon }$, up to a constant
multiple, equation (\ref{first ODE}) has unique solutions $y_{1}(z,\lambda
)$ and $y_{2}(z,\lambda ):$ 
\begin{equation}
\begin{cases}
y_{1}(z,\lambda )=(p_a(z))^{-\frac{1}{4}}\exp (-h(\lambda,z_{0},z))[1+
 \phi_{1}(z,\lambda )]
\\{y_{2}(z,\lambda )=(p_a(z))^{-\frac{1}{
4}}\exp (h(\lambda,z_{0},z))[1+ \phi_{2}(z,\lambda )]}
    
\end{cases}\label{uniform asymp}
\end{equation}
for all $z\in D_{\varepsilon }$, where $\left\vert \phi _{l}(z,\lambda
)\right\vert \leq c_{l}\frac{r_{\varepsilon }}{r-r_{\varepsilon }}$ for $
l\in \left\{ 1,2\right\} $. Consequently, \\$y_{1,2}(z,\lambda )\sim
(p_{a}(z))^{-\frac{1}{4}}\exp (\pm h(\lambda ,z_{0},z))$, with dual
asymptotic behaviors. That is, it is true if:

\begin{itemize}
\item $\phi _{l}(z,\lambda )\rightarrow 0$, as $z\rightarrow \infty $, $z\in
\gamma _{\varepsilon }^{\pm }$ for $\lambda $ fixed such that $\left\vert
\lambda \right\vert >r_{\varepsilon }$.

\item $\phi _{l}(z,\lambda )=O(\frac{1}{\lambda })$, as $\left\vert \lambda
\right\vert \rightarrow +\infty $, uniformly for $z\in D_{\varepsilon }$.
\end{itemize}
\begin{remark}
\begin{enumerate}
\item If $y_{1,2}(z,\lambda )$ decays (or blows-up) exponentially
along a canonical path $\gamma $ in $D_{\varepsilon }$, so it decays
(blows up) exponentially along any path homotopic to $\gamma $ in $
D_{\varepsilon }$. We call so this solution subdominant (dominant) in $
D_{\varepsilon }$.\label{boundary-cond-remark}

\item If $y_{1,2}(z,\lambda )$ is subdominant in $H_{k}$ so it will be
dominant in $H_{j}$ for all $j\neq k$.

\item We will give a proof of this classical result based on Liouville
transformation in Section \ref{Appendix}. The reader can also see \cite{fed1, fed2, birkhoff, Hille, olver, heading}. We will
improve the validity of these asymptotic behaviors (see again Section \ref{Appendix}
). In fact, there exist $\delta _{a,\varepsilon }>0$ and a sector $\Lambda
_{a,\varepsilon }(\theta )=\left\{ \lambda \in 
\mathbb{C}^{\ast };\,\,\left\vert \arg \lambda -\theta \right\vert \leq \delta
_{a,\varepsilon };\,\,\left\vert \lambda \right\vert >r_{\varepsilon
}\right\} $ such that these formulas are still valid with dual asymptotic
behaviors.
\item $\left\{ y_{1},y_{2}\right\} $ constitutes a fondamental system of
solutions (F.S.S) to (\ref{first ODE}) in $D$.
\end{enumerate}
\end{remark}

\subsubsection{Elementary basis}\label{EFSS}

We note the Stokes data $(D,l,z_{0})$ for: $D$ a canonical domain, $l$ a
Stokes line in $D$ and $z_{0}$ a turning point. We select the branch of $
h(\lambda ,z_{0},z)$ in $D$ such that $\Im h(\lambda ,z_{0},z)>0$ for $
z\in l$. The elementary basis $\left\{ u(z),v(z)\right\} $ associated to $
(D,l,z_{0})$ is uniquely determined by
\begin{equation}
\begin{cases}
     {u(z)=cy_{1}(z,\lambda )\text{,   \ }v(z)=cy_{2}(z,
\lambda )}, \\{\left\vert c\right\vert =1}, {arg (c)=\underset{
z\rightarrow z_{0},z\in l}{\lim }\arg (p_{a}(z))^{\frac{1}{4}})}
\end{cases}\label{elementary basis}
\end{equation}
where $y_{1}(z,\lambda )$ and $y_{2}(z,\lambda )$ are given in \eqref{uniform
asymp}.

\subsubsection{Transition matrices}

Let $(D,l,z)_{k}$ and $(D,l,z)_{j}$ be two Stokes data and $\beta
_{k,j}=\left\{ u_{k,j}\text{, }v_{k,j}\right\} $ their corresponding
elementary basis. The matrix $A_{kj}$ that changes the basis $\beta _{k}$ to $
\beta _{j}$ is called the transition matrix from $\beta _{k}$ to $\beta _{j}$.
It is clear that
\begin{equation*}
A_{kj}=A_{rj}A_{kr}\text{ , }A_{kj}^{-1}=A_{jk}
\end{equation*}

The classification made by Fedoryuk \cite[page 98-100]{fed1} gives three
types of "elementary" transition matrix. Any transition matrix from one
elementary (F.S.S) to another is the product of a finite number of this
elementary transition matrix.

\begin{enumerate}
\item \textit{The transition matrix for }$(D,l,z_{0})\rightarrow (D,l,z_{1})$
. This transition matrix exists only for a finite Stokes line that remains in the same canonical domain. It is given by
\begin{equation*}
A(\lambda )=\exp (i\sigma )\left( 
\begin{array}{cc}
0 & \exp (-i\left\vert \lambda \right\vert \alpha ) \\ 
\exp (i\left\vert \lambda \right\vert \alpha ) & 0
\end{array}
\right) \text{, }\alpha =\left\vert \varphi (z_{0},z_{1})\right\vert \text{, 
}\exp (i\sigma )=\frac{c_{2}}{c_{1}}
\end{equation*}

\item \textit{The transition matrix for }$(D,l_{0},z_{0})\rightarrow
(D,l_{1},z_{1})$. Here, the rays $h(l_{0})$ and $h(l_{1})$ are directed to
one side and $l_{0}$ lies on the left to $l_{1}$. This is the transition
from one turning point to another along an anti-Stokes line, remaining in
the same domain $D$.
\begin{equation*}
A(\lambda )=\exp (i\sigma )\left( 
\begin{array}{cc}
\exp (-\lambda \xi ) & 0 \\ 
0 & \exp (\lambda \xi )
\end{array}
\right) \text{, }\xi =\varphi (z_{0},z_{1})\text{, }\Re(\lambda \xi )>0
\text{, }\exp (i\sigma )=\frac{c_{2}}{c_{1}}
\end{equation*}

\item \textit{The transition matrix }$(D_{0},l_{0},z_{0})\rightarrow
(D_{1},l_{1},z_{0})$. Let $\left\{ l_{0}\text{, }l_{1}\text{, }l_{2}\right\} 
$ be Stokes lines starting at $z_{0}$, and let $l_{j+1}$ lie to the left to $
l_{j}$ (the order is counter-clockwise and indexed thus $4=1,..$.). We
choose the canonical domain $D_{j}$ on the left of $l_{j}$ that coincides with
the part of $D_{j+1}$ on the right of $l_{j+1}$. Then
\begin{equation*}
\left\{ 
\begin{array}{c}
A_{j,j+1}(\lambda )=\exp (-\frac{i\pi }{6})\left( 
\begin{array}{cc}
0 & \alpha _{j,j+1}^{-1}(\lambda ) \\ 
1 & i\alpha _{j+1,j+2}(\lambda )
\end{array}
\right) \\ 
\alpha _{j,j+1}(\lambda )=1+O(\lambda ^{-1})\text{, \ \ }1\leq j\leq 3 \\ 
\alpha _{1,2}(\lambda )\alpha _{2,3}(\lambda )\alpha _{3,1}(\lambda )=1\text{
, and }\alpha _{j,j+1}(\lambda )\alpha _{j+1,j}(\lambda )=1\text{ \label
{relation matrix}}
\end{array}
\right.
\end{equation*}
\end{enumerate}

\begin{remark}
\begin{enumerate}
\item The total classification of transition matrices gives four types of
such matrices (see \cite{fed1}). We need to compute the transition matrices only within $O(\lambda ^{-1})$, so it is sufficient to work with these
three types.

\item The relations in \ref{relation matrix} are direct conclusions from
the fact that $A_{3,1}A_{2,3}A_{1,2}=I$ and $A_{j,j+1}^{-1}=A_{j+1,j}$ ,
respectively.
\end{enumerate}
\end{remark}

\section{ Eigenvalue problems related to the cubic oscillator\label{main parag}}

\subsection{Admissible half planes}

Let $\theta \in \lbrack 0,\frac{\pi }{2}[$. We introduce the admissibility of half plane domains in order to classify eigenvalue problems:
\begin{definition}
\label{masero dif}
We say that two half-planes $H_{i}$ and $H_{k}$ are admissible if for every $
x_{i}\in H_{i}$, $x_{k}\in H_{k}$ there exists a canonical path $\gamma :[0,1]\rightarrow \mathbb{C}$
such that $\gamma(0)=x_{i}$ and $\gamma(1)=x_{k}$. For example two adjacent half planes are admissible.
\end{definition}

\begin{remark}\label{admissible remark}
\begin{enumerate}
\item In \cite{masero}, the author gives a
global classification to the critical graph in the case of complex cubic oscillator (with possible multiple turning points), but he
did not study the dependence of critical graph on extra parameters
which is the key of the classification of eigenvalue problems studied in this work.

\item Recall that a canonical domain is swept by exactly two
half plane (left and right) and up to two band domains. We easily deduce that
this two half plane are admissible and reciprocally if two half plane are
admissible then there exist a canonical domain that it contains.
\item Notice that, from the $\mathcal{D\diagup SG}$ correspondence paragraph \ref{sub1}, two non admissible half plane exist if and only if there exist a finite Stokes line.
\end{enumerate}
\end{remark}

\begin{definition}
\begin{itemize}
\item A Stokes complex $\Delta _{a,\theta }$ captures the half-plane $H$ if
one of the domains into which $\Delta _{a,\theta }$ splits the $z-$plane
contains $H$.

\item A Stokes complex $\Delta _{a,\theta }$ joins $H^{\prime }$ to $
H^{\prime \prime }$ if it captures both $H^{\prime }$ and $H^{\prime \prime
} $.
\end{itemize}
\end{definition}
\begin{example}
In a Stokes graph of \textbf{type A,} there is no Stokes complex that join
any two distinct half plane; while in a \textbf{Tree }Stokes graph, any two
distinct half plane are joined by at least one Stokes complex.
\end{example}

\subsection{Classification of Eigenvalue problem\label{main result
paragraph}}

Let $\theta \in \lbrack 0,\frac{\pi }{2}[$ and $\lambda =r\exp (i\theta )\in\mathbb{C} ^{\ast }$.

\begin{definition}
\label{accumulation direction}We say that $\theta $ is an accumulation
direction for the ODE: 
\begin{equation}
-y^{\prime \prime }(z)+\lambda ^{2}(z-a)(z^{2}-1)y(z)=0
\label{cubic equation}
\end{equation}
if it exists a pair of disjoint half plane $(H^{\prime },H^{\prime \prime })$ such
that this equation admit a subdominant solution (non-trivial) in $H^{\prime
} $ and $H^{\prime \prime }$.
\end{definition}

The next two results give a necessary and sufficient conditions for $\theta $
to be an accumulation direction to \eqref{cubic equation}.

\begin{proposition}
\label{main result2}Let $K$ be a compact subset of $
\mathbb{C}
\backslash \chi _{\theta }$. Then there exist $\delta_{K}>0$ such that there is no accumulation direction  in the sector $$ \Lambda
^{\delta _{K}}(\theta )=\left\{ \alpha \in \lbrack 0,\frac{\pi }{2}[;\,\,
\left\vert \theta -\alpha \right\vert \leq \delta _{K}\right\} $$  for all $a\in K$. In particular, $\theta$ is not an accumulation direction for all $a\in \mathbb{C}\setminus\chi_\theta$. 
\end{proposition}

The main result of this subsection is the following:

\begin{theorem}
\label{main result1}Let  $a\in \mathcal{\chi }_{\theta }$. 
Then $\theta $ is
an accumulation direction to \eqref{cubic equation} if and only if one of the following conditions holds:

\begin{itemize}
\item  The Stokes graph is of \textbf{type B} and the boundary conditions  are given in two
half planes $H^{\prime }$ and $H^{\prime \prime }$ joined a finite
Stokes lines $\mathcal{\ell }$ , in this case there exists a Stokes complex $\Delta
_{a,\theta }$ that joins $H^{\prime }$ with $H^{\prime \prime }$.

\item $a\in \left\{ t_{\theta },e_{\theta }\right\} $and the boundary conditions are given in two non admissible half planes $H^{\prime }$ and $H^{\prime \prime }$
joined a finite Stokes line $\ell$ .
\\When  $\ell$  corresponds to the \textit{broken finite Stokes line}, we additionally require that 
$$\dsp\frac{\vert\oint_{C_{1}} e^{i\theta}\sqrt{p_a(t)}dt\vert}{\vert\oint_{C_{2}} e^{i\theta}\sqrt{p_a(t)}dt\vert}\in\mathbb{Q}$$
where $C_1$ and $C_2$ are two simple closed contours encircling, respectively, the two unbroken finite Stokes lines.
\end{itemize}

In addition, the spectrum is a discrete sequence $\left\{ \lambda
_{n}\right\} $ accumulating on the ray $L_{r_{\varepsilon}}(\theta )=\left\{
\rho \exp (i\theta );\,\,\rho \geq r _{\varepsilon}>0\right\} $ and we have
the asymptotic formula 
\begin{equation}
\left\vert \lambda _{n}\right\vert =(2n-1)\pi \left( \oint\limits_{C}\sqrt{
\left\vert p_{a}(z)\right\vert }dz\right) ^{-1}+O(n^{-1})
\label{spectrum-asymptotic}
\end{equation}
Here $C$ is a simple closed contour encircling $\ell $.

\end{theorem}

\begin{remark}\label{rmk}
\begin{enumerate}
\item As was shown in (\ref{sub1}) the topology of $
\mathbb{C}\diagdown \Sigma _{\theta }$ is invariant for all $\theta \in ]0,\frac{\pi }{
2}[$, so for simplicity we will deal with the proof for $\theta =\frac{\pi }{
4}$.

\item For $\theta \in \left\{ 0,\frac{\pi }{2}\right\} $ we have $t_{\theta
}\in \left\{ -1,1\right\} $ and in this case we may have a double turning
point at $\pm 1$.
\item The Table bellow in Figure \ref{tab2} gives all possible cases of non admissible half plane and the number of possible eigenvalue problem in each cases.
\begin{figure}[h]
  \centering
  \includegraphics[width=\textwidth, height=\textheight, keepaspectratio]{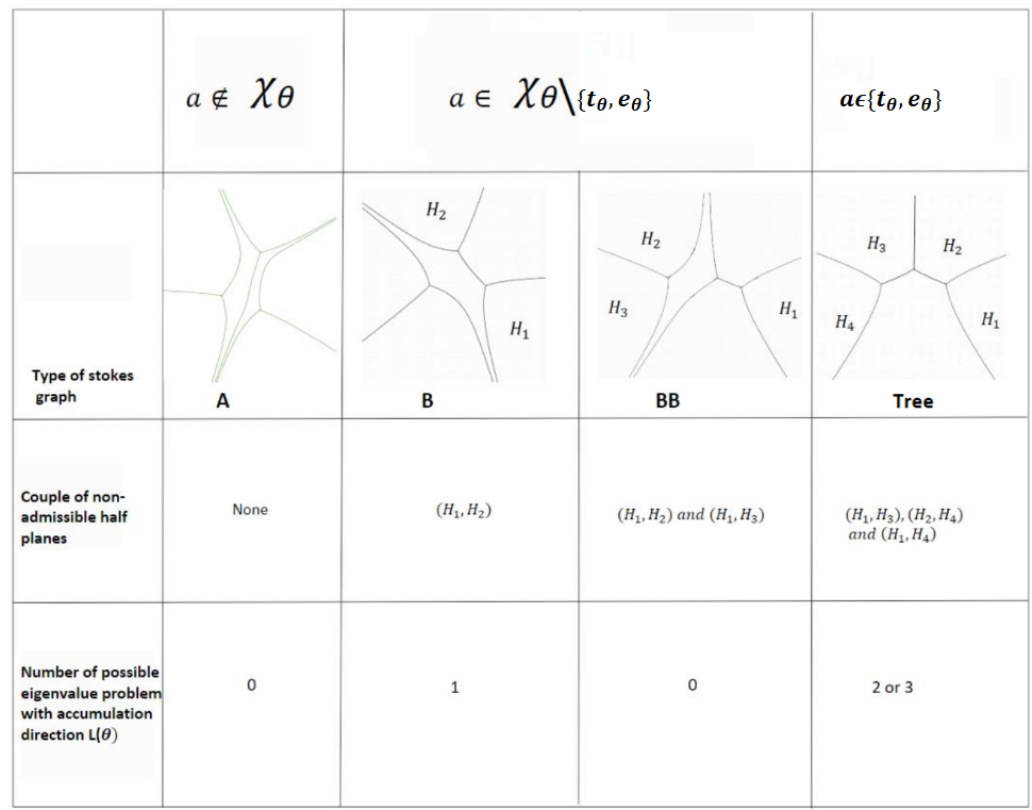}
  \caption{}
  \label{tab2}
\end{figure}  
\end{enumerate}
\end{remark}
 
\newpage

%\begin{remark}
%From main results, proposition \ref{main result2} and Theorem \ref{main result1}, we provide a
%geometric construction for a family of polynomial potentials $p_{a}$ that solves problem \ref{main problem}. Furthermore, by $\mathcal{D\diagup SG}$
%correspondence in subsection \ref{sub1}, we deduce that the set $\left\{ a\in \mathbb{C};\text{ problem\ref{main problem} has solution}\right\} $ has a strictly positive (Lebesgue) measure. Similar results were obtained in \cite{shin2}.
%The author constructs potentials which solve (\ref{main problem}) in the case of Sturm-Liouville problem. The methods used there are analytic. In \cite[
%theorem 1.6]{shin2}, the asymptotic distribution of the eigenvalue, which is
%analogues to (\ref{spectrum-asymptotic}), was investigated to reconstruct 
%\textit{"some"} coefficients of the polynomial potential. 5Our method is
%different in this case.
%\end{remark}
\section{Zeros of Eigenfunctions}\label{zeros section}

The distribution of zeros of eigenfunctions plays a crucial role in the classical Sturm-Liouville theory and interpolation theory. In
particular, the exact or asymptotic locations of zeros are crucial to analyze the completeness of the space of eigenfunctions in expansion problems, see \cite{Bender,trinh,titshmarch}. In this subsection, we try
to analyze the location and the asymptotic distribution of zeros of eigenfunctions
related to the cubic oscillator defined by \eqref{cubic equation}, as $\left(\lambda ,a\right) $ varies in $\mathcal{D}$.\\
Let $\theta \in \lbrack 0,\frac{\pi }{2}[$ and $\lambda =r \exp (i\theta )\in \mathbb{C}^{\ast }$. By Theorem \ref{main result1}, $\theta $ is an accumulation direction to $-y^{\prime \prime }(z)+\lambda
^{2}(z-a)(z^{2}-1)y(z)=0$ if and only if:

 \begin{itemize}
     \item $a\in \mathcal{\chi }_{\theta }\diagdown \left\{ t_{\theta},e_{\theta
}\right\} $\text{ and the Stokes graph is of }\textbf{type B}
\item $ a\in \left\{ t_{\theta },e_{\theta }\right\}$,

 \end{itemize}

Therefore, there exists a pair of non admissible half-planes $H^{\prime }$
and $H^{\prime \prime }$ joined by Stokes complex $\Delta _{a,\theta },$
such that the eigenvalue problem: 
\begin{equation}
\begin{cases}
 {-y^{\prime \prime }(z)+\lambda ^{2}(z-a)(z^{2}-1)y(z)=0}\\{ \quad y\text{
is subdominant in }H^{\prime }\cup H^{\prime \prime }}
\end{cases}\label{zeros-EP}
\end{equation}
has a non trivial solution $y_{a,\lambda }(z)$. The spectrum is a
discrete sequence $\left\{ \lambda _{n}\right\} $ accumulating on a ray $
L_{r _{\varepsilon}}(\theta )=\left\{ \rho \exp (i\theta );\,\,\rho \geq
r _{\varepsilon}>0\right\}$. The geometric
multiplicity of every eigenvalue is one, see \cite{titshmarch,sibuya}
. Let $f_{a,n}(z)$ be an eigenfunction associated to $\lambda _{n}$, for $
n\in \mathbb{N}$.\\
We introduce the set
\begin{equation*}
\mathcal{Z}_{a,n}=\left\{ z\in \mathbb{C}\mid f_{a,n}(z)=0\right\}
\end{equation*}
Our goal is the description of $\mathcal{Z}_{a,n}$ for $a$ $\in $ $\mathcal{
\chi }_{\theta }$. It is well known that $
f_{a,n}(z)$ is an entire function of order $\frac{\deg p_{a}+2}{2}=\frac{5}{2
}$ and of finite type, see \cite{Hille, sibuya}, hence it has
infinitely many zeros. Consequently, the set $\mathcal{Z}_{a,n}$ is not
empty when $a$ is described above, and is non reduced to a finite set of
points. It is trivial to see that the set $\mathcal{Z}_{a,n}$ is empty if $a$ 
$\notin $ $\mathcal{\chi }_{\theta }$  or
the Stokes graph is of \textbf{type BB}. 
\subsection{The Stokes Graph is of \textbf{Type B}}\label{caseB}
In this case $a\in \mathcal{\chi }_{\theta }\diagdown \left\{ t_{\theta
},e_{\theta }\text{ (if they exist)}\right\} $ and $H^{\prime }$ and $H^{\prime
\prime }$ joined by Stokes complex $\Delta _{a,\theta }$ with no broken
finite Stokes line $\ell$. We treat, for simplicity, the case $a\in \mathcal{S}_{1,\theta }$ (the other cases are similar). The Stokes complex 
$\Delta _{a,\theta }$ joins $H^{\prime }$ and $H^{\prime \prime }$ so that a
finite Stokes line $\ell $ connects the two turning points $z_{0}\in $ $
\partial H^{\prime }$ and $z_{1}\in \partial H^{\prime \prime }$. We denote
by $\left( l_{z_{j}}^{r}\right) _{j=0,1}^{r=0,1,2}$ the three trajectories that emanate from turning points. Let $l_{z_{0}}^{0}=l_{z_{1}}^{0}=\ell $ $\in \partial B_{0}$, where $B_{0}$ is the band domain, be the finite Stokes line and $\partial H^{\prime }=l_{z_{0}}^{1}\cup l_{z_{0}}^{2}$, $\partial
H^{\prime \prime }=l_{z_{1}}^{1}\cup l_{z_{1}}^{2}$ . By \eqref{boundary-cond-remark}, the boundary conditions in \eqref{zeros-EP} are equivalent to $\underset{r\rightarrow +\infty }{\lim }f_{a,n}(r\exp (i\alpha
^{\prime }))=\underset{r\rightarrow +\infty }{\lim }f_{a,n}(r\exp (i\alpha
^{^{\prime \prime }}))=0$, where $\alpha ^{\prime },\alpha ^{^{\prime \prime
}}$ are two admissible directions at $\infty $ (for example two anti-Stokes
directions \eqref{critical directions}) that correspond respectively to $H^{\prime }$ and $H^{\prime \prime }$.\\
Fix $\varepsilon >0$. Let $D_{\varepsilon }$ be a canonical domain that
contains $H^{\prime }$ with $\varepsilon $-neighborhoods of the cuts and $\varepsilon $-neighborhoods of $z_{0}$ removed. Fix a branch of the square root denoted by $\left( \sqrt{}\right) _{1}$ and $\left( \sqrt[4]{}\right) _{1}$ such that $\Re(h(\lambda _{n},z_{0},z))>0$ in $H^{\prime }$, where $h(\lambda _{n},z_{0},z)=\lambda _{n}\int_{z_{0}}^{z}\left( \sqrt{p_{a}(t)}\right) _{1}dt$. For every $z\in D_{\varepsilon }$, there exists an infinite
canonical path $\gamma _{\varepsilon }^{+}(z)$ to $\infty $. By \eqref{uniform asymp}, $f_{a,n}$ have the asymptotic formula
\begin{equation*}
\label{formula1}
f_{a,n}(z)=c_{1}(p_{a}(z))^{-\frac{1}{4}}\exp (-h(\lambda _{n},z_{0},z))[1+
\mathcal{O}_{1}(1)]
\end{equation*}
where $c_{1}$ is a constant and $\mathcal{O}_{1}(1)\rightarrow 0$, as $
z\rightarrow \infty $, $z\in \gamma _{1,\varepsilon }^{+}$ for $\lambda _{n}$
in $L_{r_{\varepsilon}}(\theta )$ . The maximal domain of applicability $
D_{H^{\prime },\varepsilon }^{1}$ of this formula is constructed by taking
the union of all the half plane admissible with $H^{\prime }$ minus a tubular neighborhood of width $2\varepsilon $ around the cuts ($\varepsilon $-neighborhood on the right and $\varepsilon $ -neighborhood on the left), the
resulting domain is shown in Figure \eqref{MDA'}. \\
In the same way, if we choose a branch of the square root $\left( \sqrt{}
\right) _{2}$ and $\left( \sqrt[4]{}\right) _{2}$ such that $\Re
(h(\lambda _{n},z_{1},z))>0$ in $H^{^{\prime \prime }}$, where $h(\lambda
_{n},z_{1},z)=\lambda _{n}\int_{z_{1}}^{z}\left( \sqrt{p_{a}(t)}\right)
_{2}dt$, and then  $f_{a,n}(z)$ admits an asymptotic formula 
\begin{equation*}\label{formula2}
f_{a,n}(z)=c_{2}(p_{a}(z))^{-\frac{1}{4}}\exp (-h(\lambda _{n},z_{1},z))[1+
\mathcal{O}_{2}(1)]
\end{equation*}
where $c_{2}$ is a constant and $\mathcal{O}_{2}(1)\rightarrow 0$, as $z\rightarrow \infty $, $z\in \gamma _{2,\varepsilon }^{+}$ for $\lambda _{n}$
fixed in $L_{r_{\varepsilon}}(\theta )$. The maximal domain of applicability $
D_{H^{^{\prime \prime }},\varepsilon }^{2}$ is constructed as in the
    previous case (see  Figure \ref{MDA''}).\\
Note that we should take care about the branch of the square root. For example, in $H^{\prime }$ we have $\left( \sqrt{}\right) _{2}=-\left( \sqrt{}
\right) _{1}$ and $\left( \sqrt[4]{}\right) _{2}=i\left( \sqrt[4]{}\right)
_{1}$. We denote by $\Pi _{\varepsilon }$ the resulting domain (see Figure \ref{pi_eps}).
\begin{figure}[h] 
\begin{minipage}[b]{0.3\linewidth}\centering\includegraphics[width=0.9\textwidth]{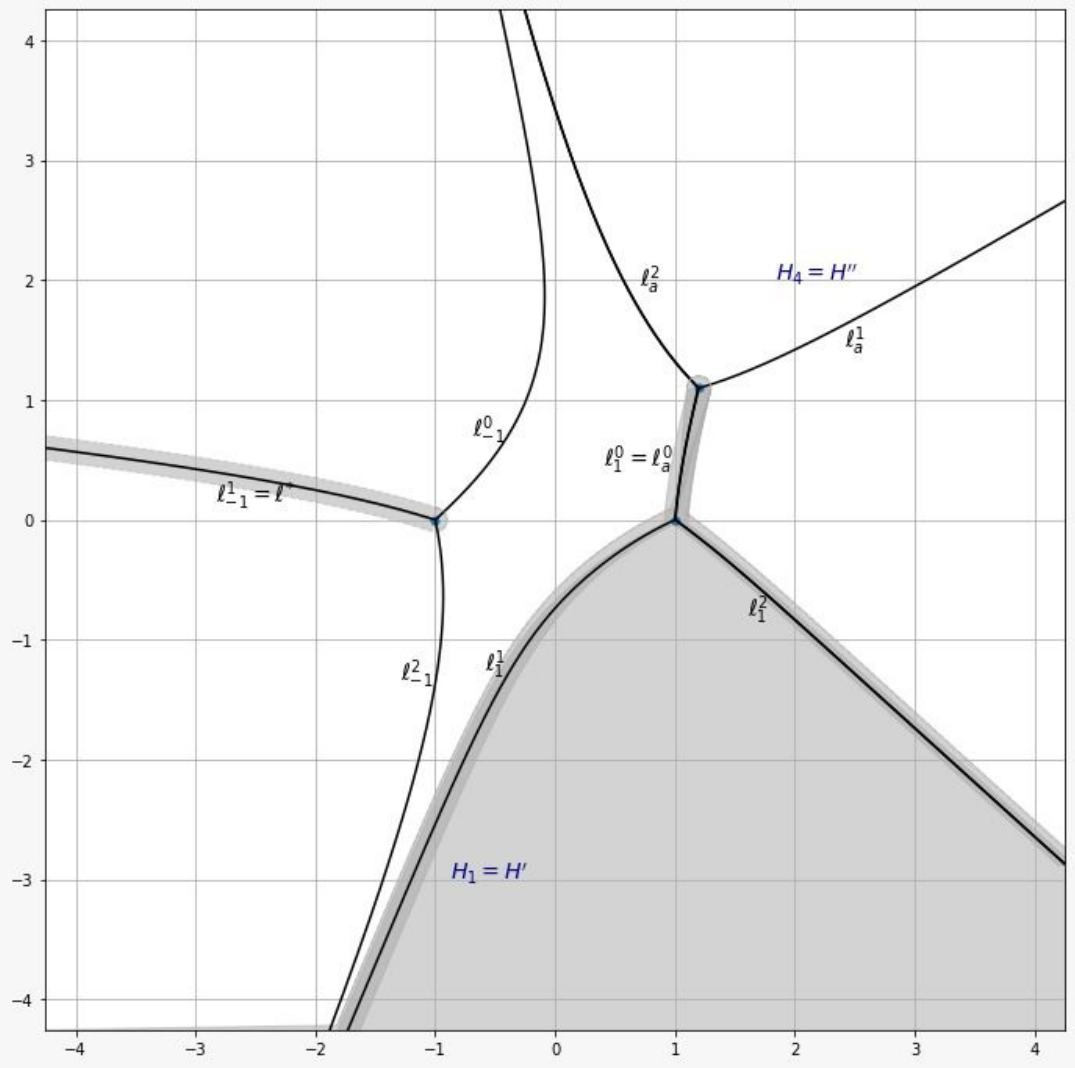}
  \caption{$D^1_{H',\varepsilon} $: The unshaded domain}\label{MDA'}
  \end{minipage} \hfill   \begin{minipage}[b]{0.3\linewidth}
  \centering
  \includegraphics[width=0.9\textwidth]{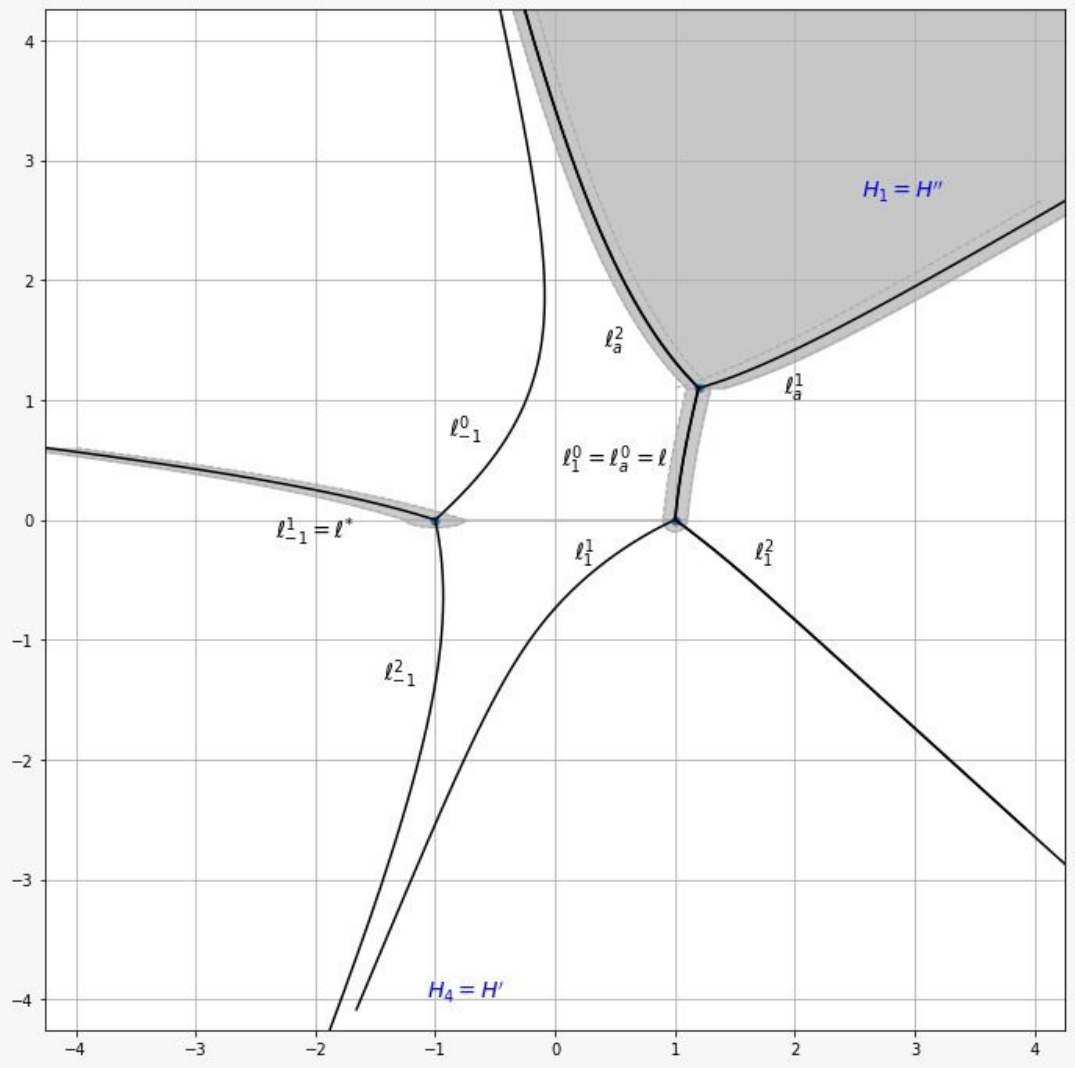}
  \caption{$D^2_{H'',\varepsilon}$: The unshaded domain}\label{MDA''}
  \end{minipage}
  \hfill   \begin{minipage}[b]{0.3\linewidth}
  \centering
  \includegraphics[width=0.9\textwidth]{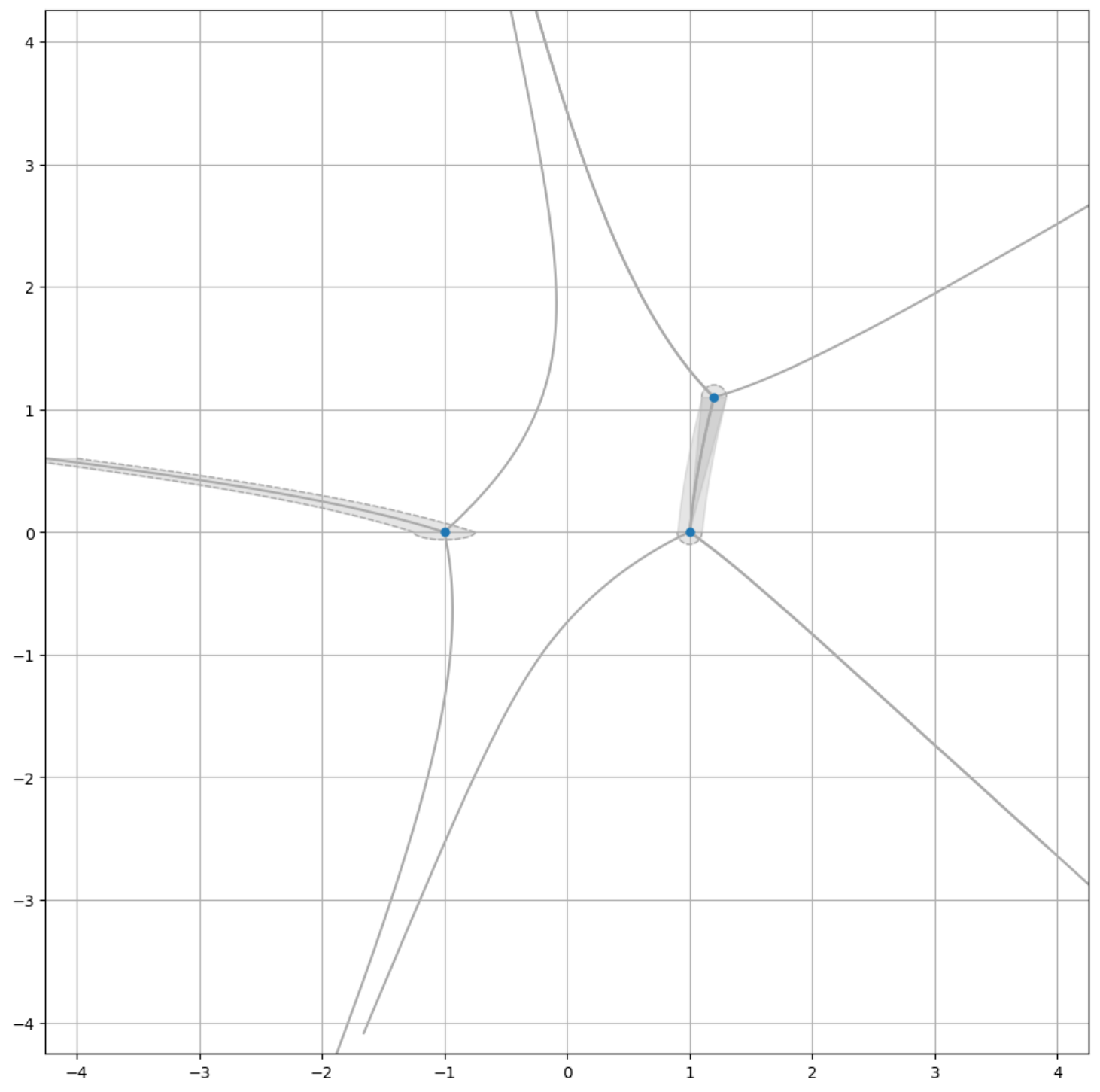}
    \caption{$\Pi _{\varepsilon }$ (unshaded domain)}
     \label{pi_eps}
  \end{minipage}
  
\end{figure}

The main result of this subsection is:

\begin{proposition}
\label{domain zeros}
$f_{a,n}(z)\neq 0$ for all $z\in \Pi _{\varepsilon }$. 
\end{proposition}
\begin{corollary}
\label{zeros corollary}
$\mathcal{Z}_{a,n}$ has two components: a bounded
component $\mathcal{Z}_{a,n}^{b}$ in $\ell $ and unbounded $\mathcal{Z}_{a,n}^{unb}$ contained in an infinite Stokes line $\ell ^{\ast }$.
\end{corollary}

\begin{proof}
From the previous proposition \ref{domain zeros}, we deduce that $\mathcal{Z}_{a,n}$ is contained in $\mathbb{C}\setminus\Pi _{\varepsilon }$. By the isolated zeros principle for entire functions \cite{AHLFORS}, we deduce that $\mathcal{Z}_{a,n}$ has two components: a bounded
component $\mathcal{Z}_{a,n}^{b}$ near $\ell $ and unbounded $\mathcal{Z}_{a,n}^{unb}$ near an infinite Stokes line $\ell^{\ast }$.\\
Suppose now that there exists $t\in\mathcal{Z}_{a,n}$ such that $dist(t,\ell \cup \ell ^{\ast}\mathcal{)>}c\mathcal{>}0$, where $dist$ is the usual Hausdorff metric, then by choosing $0<\varepsilon
<c$, we obtain a contradiction with \ref{domain zeros}. This achieves the proof.
\end{proof}

\begin{proposition}\label{finite zeros}
There exists $n_{0}(a)\in \mathbb{N}$ such that for $n\geq n_{0}(a)$, $\mathcal{Z}_{a,n}^{b}$ contain exactly $n$ zeros in $\ell$.
\end{proposition}

\begin{proof}
First, let $N_{n}(a)$ be the number of zeros in $\mathcal{Z}
_{a,n}^{b}$, and let $C$ be a simple closed contour encircling $\ell $ in $\Pi
_{\varepsilon }$. It is obvious that $N_{n}(a)=\frac{1}{2i\pi }
\oint\limits_{C}\left( \frac{f_{a,n}^{\prime }(z)}{f_{a,n}(z)}\right) dz$.
For $\varepsilon >0$, let $D$ a canonical domain that contains $\ell $ and $
D_{\varepsilon }$ as described in (\ref{caseB}). Fix a branch of the square
root in $D_{\varepsilon }$. By WKB formulas (\ref{uniform asymp},\ref{WKB
derivative}) we have:
\begin{eqnarray*}
\dsp f_{a,n}(z) &=&cy_{l}(z)=c(p_{a}(z))^{\frac{-1}{4}}\exp (\pm h(\lambda
_{n},z_{0},z))[1+\phi _{l}(z,\lambda _{n})] \\
f_{a,n}^{\prime }(z) &=&cy_{l}^{\prime }(z)=c\frac{-1}{4}\frac{p_{a}^{\prime
}(z)}{(p_{a}(z))^{\frac{5}{4}}}\exp (\pm h(\lambda _{n},z_{0},z))[1+\phi
_{l}(z,\lambda _{n})]\pm \\
c\lambda _{n}&(p_{a}(z))^{\frac{1}{4}} &\exp(\pm h(\lambda
_{n},z_{0},z))[1+\phi _{l}(z,\lambda _{n})]+
c(p_{a}(z))^{\frac{-1}{4}}\exp
(\pm h(\lambda ,z_{0},z))\phi _{l}^{\prime }(z,\lambda _{n})
\end{eqnarray*}
where $\phi _{l}(z,\lambda _{n})\rightarrow 0$ as $n\rightarrow +\infty ,$
uniformly for $z\in D_{\epsilon }$ and $l\in \left\{ 1,2\right\}$.\\
So we obtain:
\begin{equation*}
\frac{f_{a,n}^{\prime }(z)}{f_{a,n}(z)}=\frac{-1}{4}\frac{p_{a}^{\prime }(z)
}{(p_{a}(z))}\pm \lambda _{n}\sqrt{p_{a}(z)}+\frac{\phi _{l}^{\prime
}(z,\lambda _{n})}{[1+\phi _{l}(z,\lambda _{n})]}.
\end{equation*}
By (\ref{asymp derive}), $\phi _{l}^{\prime }(z,\lambda _{n})\rightarrow 0$
as $n\rightarrow +\infty $, uniformly for $z\in D_{\epsilon }$ and $l\in
\left\{ 1,2\right\} $. This formulas still valid, as described in Section (\ref
{caseB}), as $n\rightarrow +\infty$, uniformly for $z\in D_{H^{\prime},\varepsilon }^{1}$ (or uniformly for $z\in D_{H^{^{\prime \prime
}},\varepsilon }^{2}$). Consequently, 
\begin{equation*}
\frac{f_{a,n}^{\prime }(z)}{f_{a,n}(z)}=\frac{-1}{4}\frac{p_{a}^{\prime }(z)
}{(p_{a}(z))}\pm \lambda _{n}\sqrt{p_{a}(z)}+o(n^{-1}),
\end{equation*}
as $n\rightarrow +\infty $, uniformly for $z\in \Pi _{\varepsilon }$ (for any choice of square root).\\
We deduce that $N_{n}(a)=\frac{-1}{4}\frac{1}{2i\pi }\oint\limits_{C}( 
\frac{p_{a}^{\prime }(z)}{(p_{a}(z))} dz)\pm \frac{\lambda _{n}}{2i\pi 
}\oint\limits_{C}( \sqrt{p_{a}(z)} dz)+o(n^{-1})$. From $\lambda_{n}=\dsp\left\vert \lambda _{n}\right\vert \exp (i\theta )$, we have $N_{n}(a)= \dsp \frac{-1}{4}\frac{1}{2i\pi }\oint\limits_{C}( \frac{p_{a}^{\prime }(z)}{
(p_{a}(z))}dz)\pm \frac{\left\vert \lambda _{n}\right\vert }{2i\pi }
\oint\limits_{C}\left( \exp (i\theta )\sqrt{p_{a}(z)}\right) dz+o(n^{-1})$,
as $n\rightarrow +\infty$.\\
From Lemma (\ref{lm1}), we have 
\begin{eqnarray*}
\oint\limits_{C}\left( \exp (i\theta )\sqrt{p_{a}(z)}\right) dz &=&\pm
2\int_{\ell }\left( \exp (i\theta )\sqrt{p_{a}(z)}\right) dz \\
&=&\pm 2i\int_{\ell }\Im\left( \exp (i\theta )\sqrt{p_{a}(z)}dz\right)
\\
&=&\pm i\oint\limits_{C}\sqrt{\left\vert p_{a}(z)\right\vert }\left\vert
dz\right\vert .
\end{eqnarray*}
By (\ref{spectrum-asymptotic}) and since $\dsp \frac{1}{2i\pi }\oint\limits_{C}
( \frac{p_{a}^{\prime }(z)}{(p_{a}(z))} dz)=2$, 
\begin{equation*}
N_{n}(a)=n+o(n^{-1}).
\end{equation*}
There exists $n_{0}(a)\in \mathbb{N}$ such that for $n\geq n_{0}(a),o(n^{-1})<1$, and $N_{n}(a)=n$. This achieves
the proof.
\end{proof}

\begin{remark}

\begin{enumerate}
\item The results (\ref{domain zeros},\ref{zeros corollary}, \ref{finite zeros}) show that zeros of eigenfunctions have a limit distribution in the plane which depends only on the Stokes geometry. In \cite{eremenko3} the symmetries of the $\mathcal{SG}$ were investigated to establish similar results for an appropriate rescaled eigenfunctions, resulting from a Sturm-Liouville problem, which correspond to a particular case of our work (namely $\theta =\frac{\pi }{4}$ and $a\in \mathcal{\chi }_{\frac{\pi }{4}})$. Similar results were obtained in \cite{hezari} for real quartic potential. 

\item The result (\ref{finite zeros}) gives a proof of the conjecture
proposed by Trinh \cite{trinh} for the exact location of a finite number of
zeros of PT-symmetric cubic oscillator eigenfunctions derived from a
Sturm-Liouville problem ( $\theta =\frac{\pi }{4}$ in our work).
\end{enumerate}
\end{remark}

\subsection{The Boutroux Graph case}
Let $a\in t_{\theta},e_{\theta}\text{(if it exists)}\}$ , If the boundary conditions $H^{\prime }$ and $H^{\prime\prime }$ are joined by the unbroken short trajectory $l_1$ or $l_2$, then the eigenvalue problems  are similar to the case when the stokes graph is of \textbf{type B}. Consequently, the set $\mathcal{Z}_{a,n}$ consists of two components located along \textbf{the marked Stokes lines}, as illustrated in Figures \ref{FIG28} and \ref{FIG.29}.
\begin{figure}[tbh]
\begin{minipage}[b]{0.4\linewidth}
		\centering\includegraphics[scale=0.35]{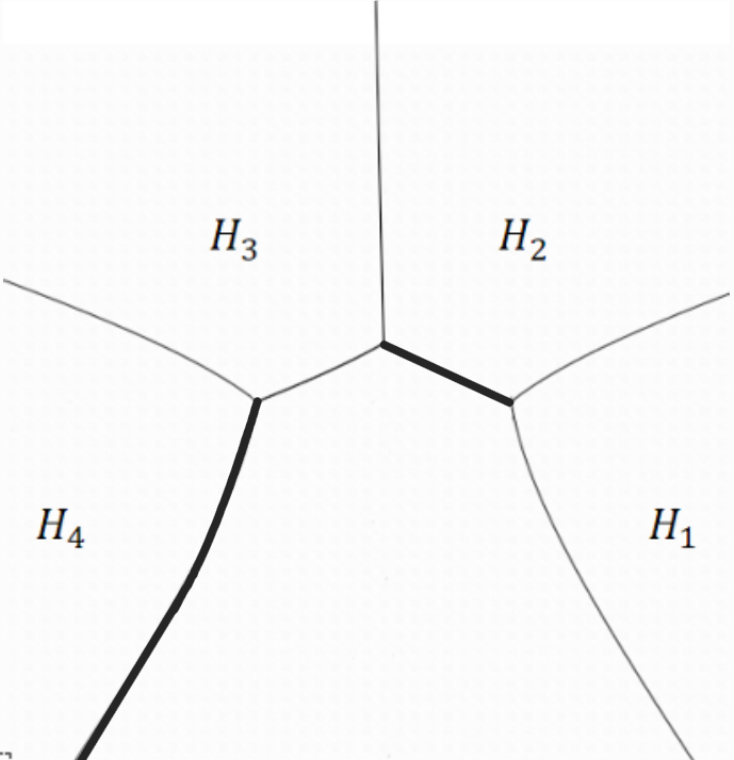}
		\caption{$(H',H'')=(H_1,H_3)$}\label{FIG28}
\end{minipage}\hfill\begin{minipage}[b]{0.4\linewidth}
	\centering\includegraphics[scale=0.35]{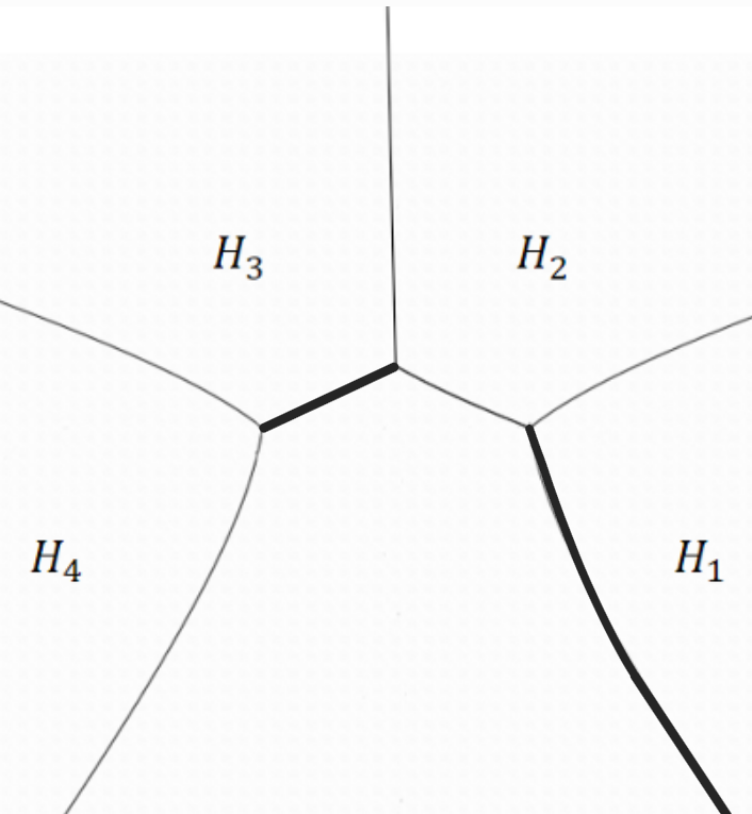}
	\caption{$(H',H'')=(H_2,H_4)$}\label{FIG.29}
\end{minipage}\hfill\end{figure}

Furthermore, Theorem \ref{main result1} implies another possible case: $H^{\prime }$ and $H^{\prime\prime }$ are joined by the broken short trajectory $l_1\cup l_2$ when 
$$
\alpha=\dsp\frac{\vert\oint_{C_{1}} e^{i\theta}\sqrt{p_a(t)}dt\vert}{\vert\oint_{C_{2}} e^{i\theta}\sqrt{p_a(t)}dt\vert}\in\mathbb{Q}
$$ 
where $C_1$ and $C_2$ are two simple closed contours encircling  $l_1$ and $l_2$, respectively. In this case $\mathcal{Z}_{a,n}$ is entirely supported on \textbf{the marked Stokes lines}, as shown in Figure \ref{fig12}.
\begin{figure}[h]
    \centering
    \includegraphics[width=0.33\linewidth]{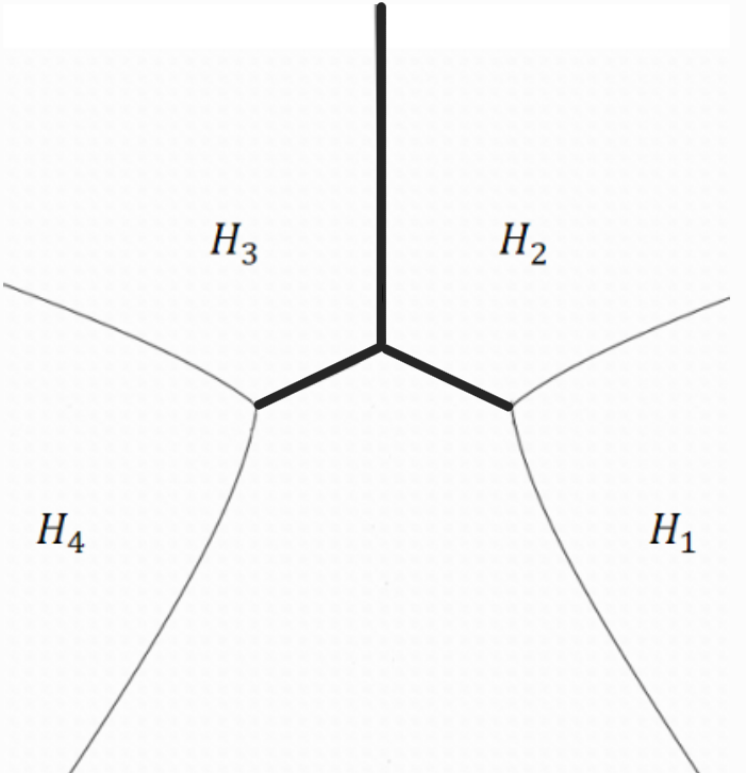}
    \caption{$(H',H'')=(H_1,H_4)$}
    \label{fig12}
\end{figure}

\begin{remark}
    As shown in Proposition \ref{finite zeros}, for sufficiently large $n$, the number of zeros supported on a short trajectory $l$ is proportional to $\left|\oint_{C} e^{i\theta}\sqrt{p_a(t)}\,dt\right|$, where $C$ is a simple closed contour encircling $l$. Consequently, if $\alpha>1$, the density of zeros in $l_1$ is greater than in $l_2$; if $\alpha<1$, the reverse is true.
\end{remark}
\section{Applications}\label{applications}

\subsection{Cubic oscillator with infinite real zeros solution}
As a first application to the results obtained in Section \ref{zeros section}, we try to construct polynomial potential $P_{x}$ such that $\deg(P_{x})=3$ and the ODE
\begin{equation*}
y^{\prime \prime }(z)+P_{x}(z)y(z)=0
\end{equation*}
admits solutions with infinity real zeros. This problem was introduced in \cite[Problem $2.71$]{brann} and has been investigated
later in \cite{eremenko1, eremenko2, gunderson, shin1, titshmarch}.
\begin{proposition}
There exists a family of polynomials $\left( P_{x}\right) _{x\in \Theta },$
where $\Theta $ is an unbounded curve in the plane and $\deg \left(
P_{x}\right) =3$, such that the ODE 
\begin{equation*}\label{ODE-question}
y^{\prime \prime }(z)+P_{x}(z)y(z)=0
\end{equation*}
has a non trivial solution with infinitely real zeros and finitely non real
zeros. In addition, for some value $x_{crit\,\,}$, the number of non
real zeros is even and there are symmetric with the real axis.
\end{proposition}

\begin{proof}
Starting with the eigenvalue problem:
\begin{equation*}
\left\{ 
\begin{array}{c}
-y^{\prime \prime }(z)+\left\vert \lambda \right\vert
^{2}i(z-a)(z^{2}-1)y(z)=0 \\ 
y(\pm \infty )=0
\end{array}
\right. 
\end{equation*}
where  $a\in \mathcal{S}_{\blacktriangle ,\frac{\pi }{4}}$, see Notation \ref
{Notations sigma}. By Theorem \ref{main result1}, the eigenvalue problem
had a non trivial solution with boundary conditions in two half planes $
H^{\prime }$ and $H^{\prime \prime }$ joined by a Stokes complex that
contains $\pm \infty $ as anti-Stokes directions (see \ref{critical
directions}). Let $f_{a}$ be the eigenfunction that corresponds to $\lambda $.

By Section \ref{zeros section}, all zeros of $f_{a}$ are non real and
there are infinitely many pure imaginary zeros. If we take $
g_{a}(z)=f_{a}(iz)$, then 
\begin{equation*}
\left\{ 
\begin{array}{c}
g_{a}^{\prime \prime }(z)+\left\vert \lambda \right\vert
^{2}(z+ia)(z^{2}+1)g_{a}(z)=0 \\ 
g_{a}(z)(\pm i\infty )=0
\end{array}
\right. 
\end{equation*}

By taking the changes $-1\leftrightarrow -i,1\leftrightarrow i$ in the levels sets (
\ref{level sets}), we obtain:
\begin{align*}
\Sigma _{i,0}& =\left\{ x\in 
\mathbb{C}
\setminus \left] -i\infty ,-i\right] :\Im \left( \int_{\left[ i,x\right] }
\sqrt{(z-x)(z^{2}+1)}dz\right) =0\right\} ; \\
\Sigma _{-i,0}& =\left\{ x\in 
\mathbb{C}
\setminus \left[ i,+i\infty \right[ :\Im \left( \int_{\left[ -i,x\right] }
\sqrt{(z-x)(z^{2}+1)}dz\right) =0\right\} ; \\
\Sigma _{i,\blacktriangle ,0}& =\left\{ x\in 
\mathbb{C}
\setminus \left[ -i,i\right] :\Im \left( \int_{\left[ -i,i\right] }\sqrt{
(z-x)(z^{2}+1)}dz\right) =0\right\} ;
\end{align*}
The topology of these sets is similar to subsection\ref{level sets}. Let $\Theta =
\mathcal{S}_{i,\blacktriangle ,0}:$ the part of $\Sigma _{i,\blacktriangle
,\theta }$ starting at $t_{i,0}=ix_{crit}$ and diverging to $\infty $ (see 
 Notation \ref{Notations sigma}). Let so $P_{x}(z)=\left\vert \lambda \right\vert
^{2}(z-x)(z^{2}+1)$ for $x\in \Theta $. The function $g_{ix}$ had an
infinite number of real zeros and a finite number of non real zeros. This
answers the first part of the question. 

For the second part of the answer, we
take $g_{ix_{crit}}$. In this case, the Stokes graph is a Boutroux curve with
broken finite Stokes line. The two parts of the finite Stokes line are
symmetric with respect to the real line, and the number of non real zeros is even. This finish the proof.
\end{proof}
\begin{remark}
    As it shown in figures (\ref{fig1}, \ref{fig2} and \ref{fig3}), we can construct solutions to (\ref{ODE-question}) with infinite positive or negative real zeros.  
\end{remark}

\begin{figure}[ht]
\begin{minipage}[b]{0.3\linewidth}
		\centering\includegraphics[scale=0.35]{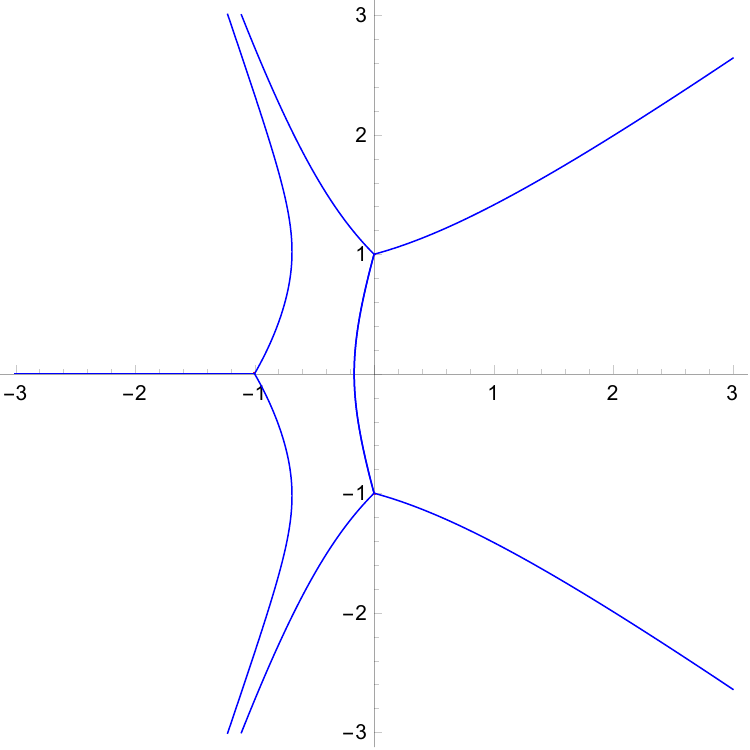}\caption{Negative real zero+finite non real zeros}\label{fig1}
\end{minipage}\hfill\begin{minipage}[b]{0.3\linewidth}
	\centering\includegraphics[scale=0.35]{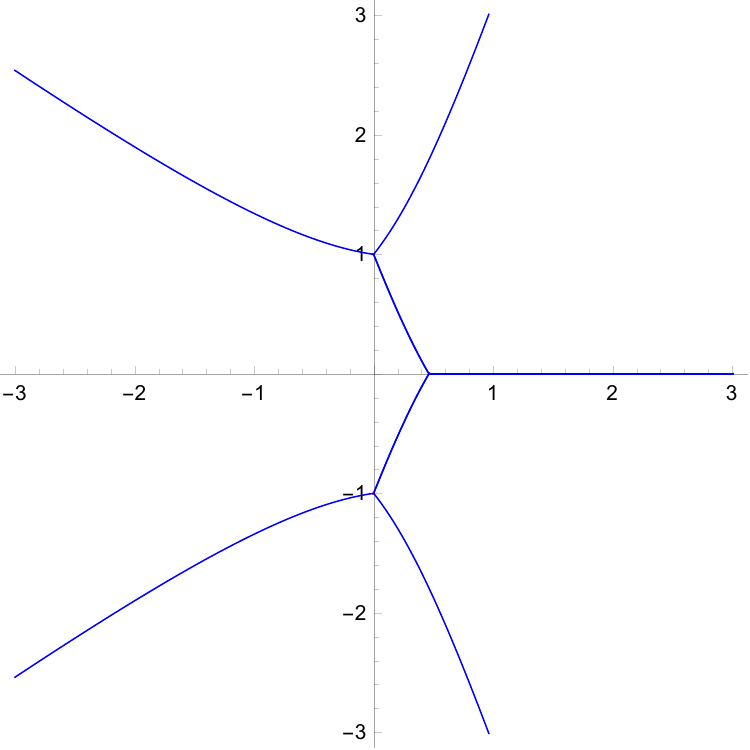}
	\caption{Positive real zero and $2n$ non real zeros}\label{fig2}
\end{minipage}\hfill\begin{minipage}[b]{0.3\linewidth}
			\centering\includegraphics[scale=0.35]{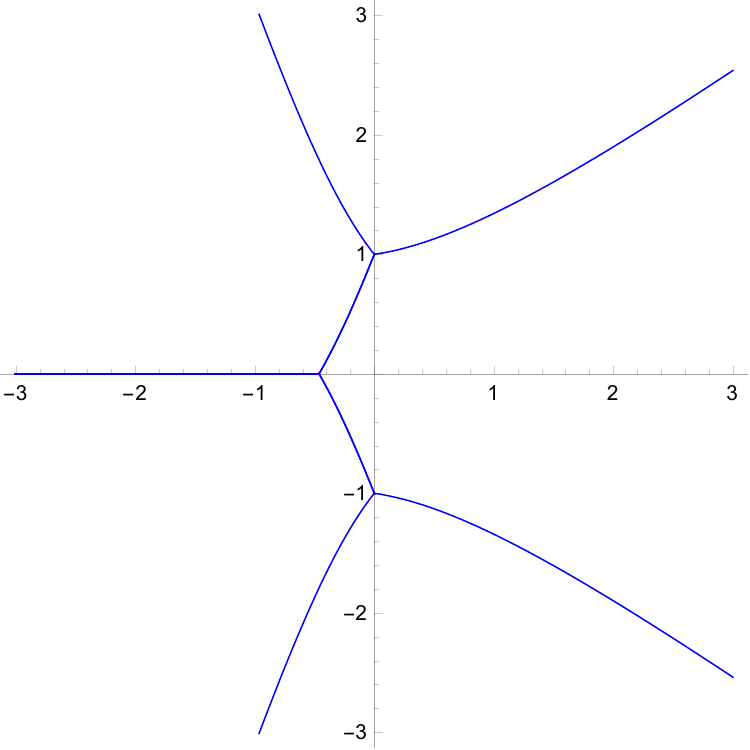}\caption{Negative real zero and 2n non real zeros}\label{fig3}
\end{minipage}\hfill\end{figure}
\begin{figure}[h]
    \centering
    \includegraphics[width=0.35\linewidth]{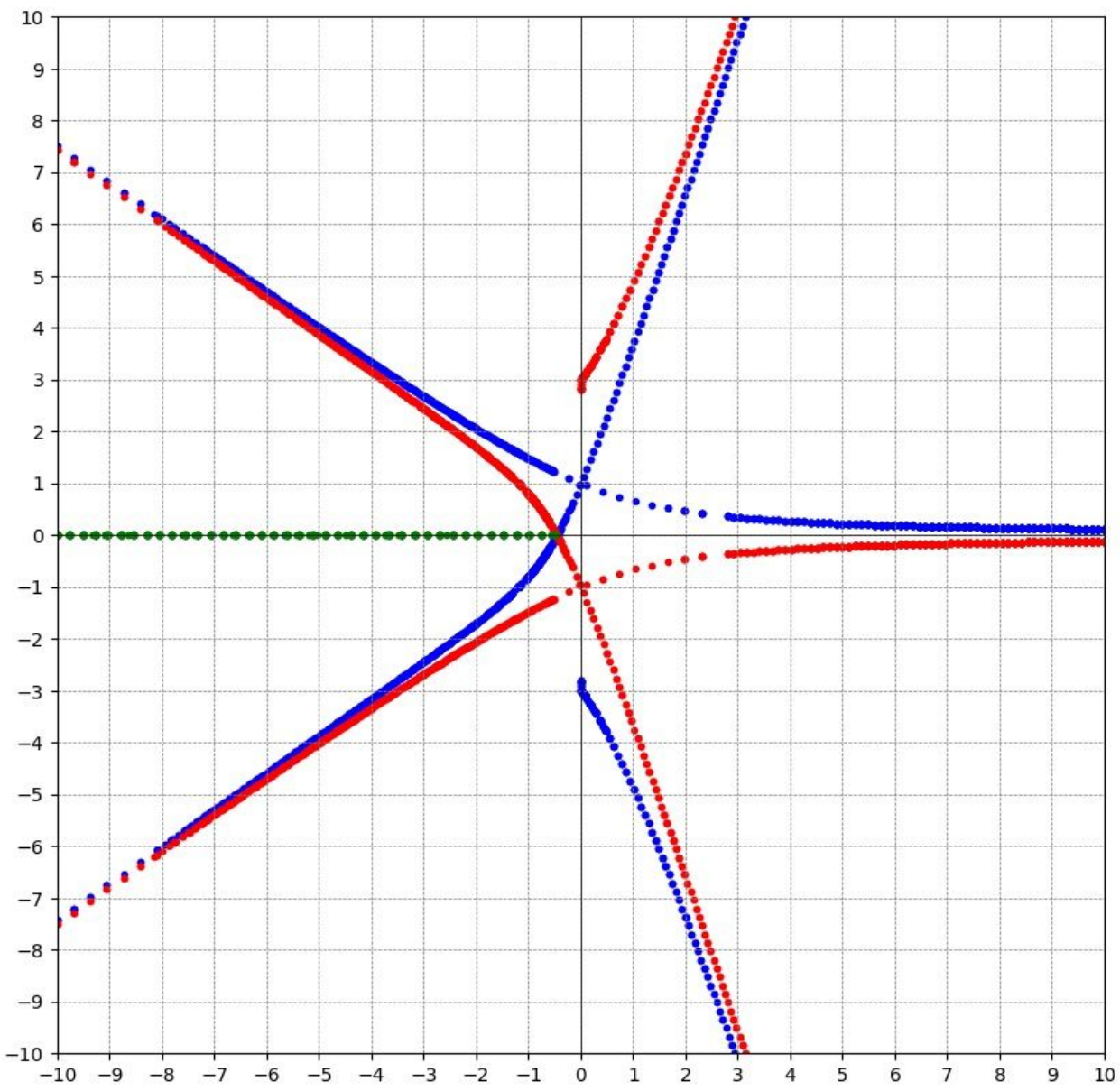}
    \caption{Exact plot of modified Sigma curves for the case $\theta=0$}
    \label{theta=0, i,-i}
\end{figure}
\newpage
\subsection{Sturm-Liouville Problem}

The second application concerns the eigenvalue problem for the Sturm-Liouville equation.
\begin{equation}
\label{eig99}
    \begin{cases}
        y''(z)+i(z^3+\alpha z)y(z)=Ey(z)\\y(-\infty)=y(+\infty)=0.
    \end{cases}
\end{equation}
It is well known (Y. Sibuya \cite{sibuya}, Shin K.C \cite{shin2}) that the eigenvalues $E_n$ of \eqref{eig99} are discrete, simple, can be arranged $\left\vert
E_{1}\right\vert <\left\vert E_{2}\right\vert <\left\vert
E_{3}\right\vert <...<\left\vert
E_{n}\right\vert...\rightarrow \infty $, real for sufficiently large $n$ and have the following asymptotic expression.
\begin{equation}\label{asymeq}
E_n=\left(\dsp\frac{2\Gamma(\frac{2}{3})\sqrt{\pi}(2n-1)}{\sqrt{2}\Gamma(\frac{1}{3})}\right)^\frac{6}{5}[1+o(1)], \text{ as }n\text{ tends to infinity, }n\in \mathbb{N}.
\end{equation}
In this subsection we present a new proof that the spectrum is real for sufficiently large $n$, and we also establish the asymptotic distribution of the zeros of the corresponding eigenfunctions in the complex plane. 
 \begin{proposition}
Consider the eigenvalue problem \begin{equation}
\label{eig(99)}
    \begin{cases}
        y''(z)+i(z^3+\alpha z)y(z)=Ey(z)\\y(-\infty)=y(+\infty)=0.
    \end{cases}
\end{equation}
where $\alpha\in \mathbb{R}$. Let {$E_n$} be the set of eigenvalues of \eqref{eig(99)}  arranged in increasing modulus: $$\left\vert
E_{1}\right\vert <\left\vert E_{2}\right\vert <\left\vert
E_{3}\right\vert <...<\left\vert
E_{n}\right\vert...\rightarrow \infty $$ Then, for sufficiently large n:
\begin{enumerate}
\item $E_n$ is real.
\item The zeros of the eigenfunction associated with $E_n$  are supported on two lines obtained by applying an affine transformation to the Stokes lines $l_0$ and $l_1$ of the quadratic differential 
\[
i(z^2 - 1)(z - i\sqrt{3})dz^2
\]
where $l_0$ is the short trajectory and $l_1$ is the one contained in the imaginary axis (see Figure \ref{FIG31}).
\end{enumerate} 
\end{proposition}

\begin{proof}

Let $y=y_n$ be an eigenfunction of \eqref{eig99} associated with the eigenvalue $E=E_n$. We make the change of variable:
$$
z=rx,~~\text{where}~~r=|E|^{\frac{1}{3}},
$$
where we have used the notation $\beta=arg(E)$ and $C=ir^{-3}E=ie^{i\beta}$. We see that the new function $\varphi(x)=y(rx)$ is a solution of: 
$$
-\varphi''(x)+ir^5(x^3+\alpha r^{-2}x+C)\varphi(x)=0.
$$
Moreover, since $y$ is subdominant at both $\pm \infty$, and the resealing $r$ is real, the function $\varphi$ is also exponentially small at $\pm\infty$.As $n\to \infty$, $\varphi$ satisfy the eigenvalue problem of the form:
$$
\begin{cases}
    -\varphi''(x)+ir^5(x^3+C)\varphi(x)=0\\
    \varphi(\pm \infty)=0.
\end{cases}
$$
Let $P(x)=x^3+ie^{i\beta}$. Using the change of variable $z=bx+c$ where $b=\frac{2}{\sqrt{3}}ie^{-i\frac{\beta-\frac{\pi}{2}}{3}}$ and $c=1-\frac{2}{\sqrt{3}}ie^{-i\frac{2\pi}{3}}$, we define $Q(z)=b^{3}P(x)=(z^2-1)(z-i\sqrt{3})$ and $\psi(z)=\varphi(x)$.\\
Then $\psi$ satisfies the eigenvalue problem of the form:
\begin{equation}\label{eig100}
\begin{cases}
    -\psi''(z)+\lambda^2 (z^2-1)(z-i\sqrt{3})\psi(z)=0,~~\text{where} ~\lambda^2=ir^5b^{-1}\\ \underset{\underset{x\in\mathbb{R}}{|x|\to+\infty}}{\lim}\psi(bx)=0
    \end{cases}    
\end{equation}

Let $\theta=arg(\lambda)\in[0,\pi]$. We observe from some Stokes graphs of the quadratic differential $\varpi_\theta=\lambda^2 Q(z) dz^2$ (see Figure  \ref{FIG31},  \ref{FIG.32} and  \ref{FIG33}) we have 
 $i\sqrt{3}\in S_{\frac{\pi}{4}}\cap S_{1,\frac{7\pi}{12}} \cap S_{-1,\frac{11\pi}{12}}$. As seen in Lemma  \ref{relation sigma} the existence of a short trajectory  connecting two turning points of $\varpi_\theta$ is only possible for at most three values of $\theta\in[0,\pi[$ , then $\theta\in \{\frac{\pi}{4},\frac{7\pi}{12},\frac{11\pi}{12}\}$.
\begin{figure}[ht]
\begin{minipage}[b]{0.3\linewidth}
		\centering\includegraphics[scale=0.3]{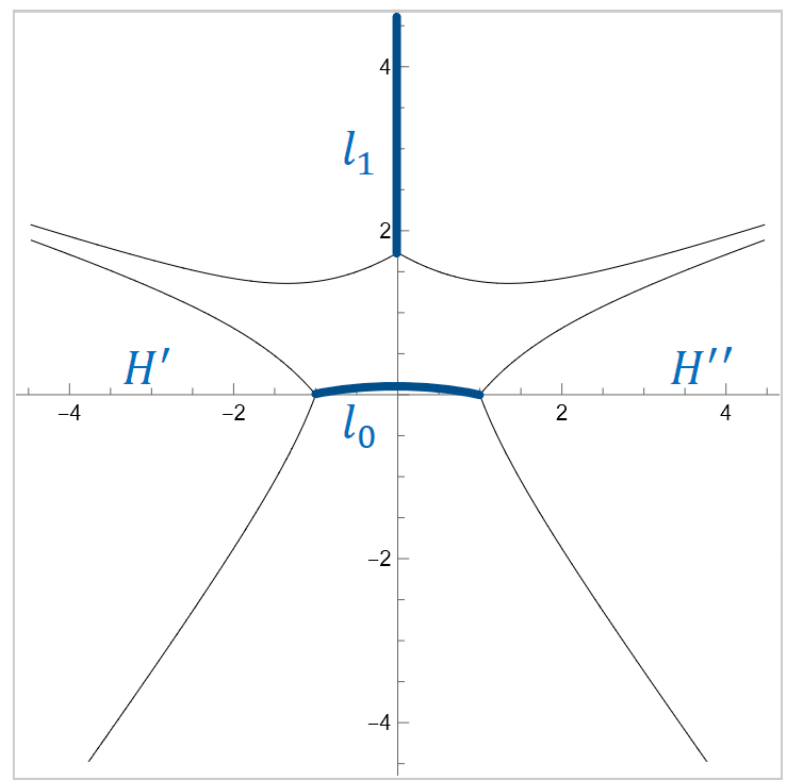}
		\caption{$\theta=\frac{\pi}{4}$}\label{FIG31}
\end{minipage}\hfill\begin{minipage}[b]{0.3\linewidth}
	\centering\includegraphics[scale=0.38]{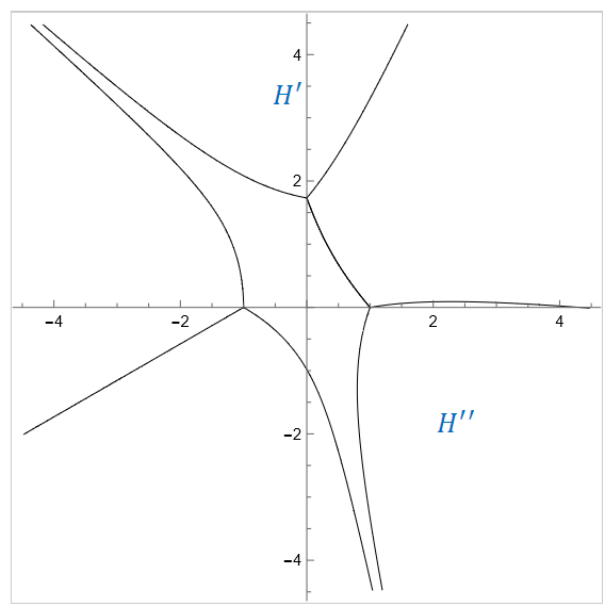}
	\caption{$\theta=\frac{7\pi}{12}$}\label{FIG.32}
\end{minipage}\hfill\begin{minipage}[b]{0.3\linewidth}
			\centering\includegraphics[scale=0.45]{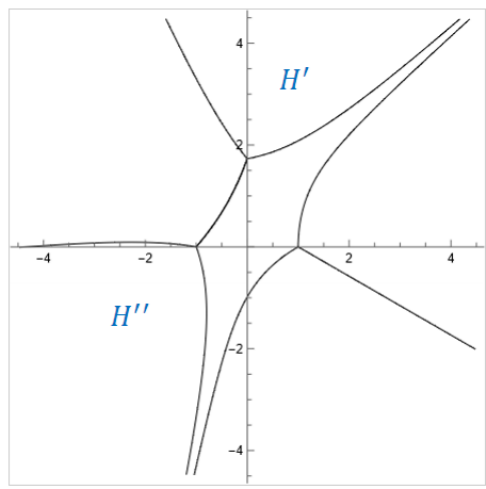}
			\caption{$\theta=\frac{11\pi}{12}$}\label{FIG33}
\end{minipage}\hfill\end{figure}

Now let $\theta\in \{\frac{\pi}{4},\frac{7\pi}{12},\frac{11\pi}{12}\}$, $H'$ and $H''$ represent the two non admissible half-plan connected by the short trajectory. Since $\{bx,~~x\in\mathbb{R},|x|>>1\}\subset H'\cup H''$ only for $\theta=\frac{\pi}{4}$, from Theorem  \ref{main result1} we deduce that the spectrum is discrete, accumulates in the direction $\frac{\pi}{4}$ and has the following asymptotic expression.
\begin{equation}\label{asymeq2}
\left\vert \lambda \right\vert \underset{\underset{n\in \mathbb{N}}{n\to \infty}}{=}(2n-1)\pi \left( \oint\limits_{C}\sqrt{
\left\vert Q(z)\right\vert }dz\right) ^{-1}[1+o(1)], \text{ where }  C \text{ is a simple contour encircling }l_0.
\end{equation}
We also note that $arg(\lambda^2) =\frac{\pi}{2}$ implies $a \in \mathbb{R}$, which leads to $\beta - \frac{\pi}{2} = -\frac{\pi}{2}$  and thus $\beta = 0$. Consequently, $E$ is real. \\
Moreover, as seen in Section  \ref{zeros section}, the zeros of the eigenfunctions $\psi$ are accumulated on the Stokes lines $l_0$ and $l_1$ (see Figure \ref{FIG31}). Consequently, the zeros of the eigenfunctions $y$ are obtained by applying the affine transformation $\frac{r}{b}(x - c)$ to the zeros of $\psi$.
    
\end{proof}
\begin{remark}
    Substituting  $|\lambda|=|E|^{\frac{5}{6}}.|b^{-1}|$ into \eqref{asymeq} and \eqref{asymeq2}, we obtain: $$\oint\limits_{C}\sqrt{
\left\vert Q(z)\right\vert }dz=\frac{\sqrt{2\pi}\Gamma(\frac{1}{3})}{\sqrt3 \Gamma(\frac{1}{2})}$$
This gives the value of the Abelian integral:
$$\oint\limits_{C}\sqrt{
 (z^2-1)(z-i\sqrt{3})}dz=\pm ie^{i\frac{\pi}{4}}\dsp\frac{\sqrt{2\pi}\Gamma(\frac{1}{3})}{\sqrt3 \Gamma(\frac{1}{2})}=\dsp\pm ie^{i\frac{\pi}{4}}\sqrt\frac{{2}}{3}\Gamma(\frac{1}{3}) $$
\end{remark}

\section{Proofs}\label{proofs}

The next Lemma describes the topology of $\left( \Sigma _{\theta }\right)
_{\theta \in \lbrack 0,\frac{\pi }{2}[}$ defined in the subsection  \ref{level sets}:

\begin{lemma}
\label{relation sigma}

\begin{enumerate}
\item If $\alpha \neq \theta $ with $\alpha \in \lbrack 0,\frac{\pi }{2}[$ then: 
\begin{equation*}
\begin{array}{c}
\dsp\Sigma _{-1,\theta }\cap \Sigma _{-1,\alpha }=\left\{ -1\right\} \\ 
\dsp\Sigma _{1,\theta }\cap \Sigma _{1,\alpha }=\left\{ 1\right\} \\ 
\dsp\Sigma _{\blacktriangle ,\theta }\cap \Sigma _{\blacktriangle ,\alpha
}=\varnothing
\end{array}
\end{equation*}

\item Let $\left( \theta _{n}\right) $ be a sequence in $[0,\frac{\pi }{2}[$
and $\left( a_{n}\right) $ a sequence of complexes numbers. Suppose that $
\theta _{n}\rightarrow \theta $ and $a_{n}\rightarrow a$. If it exist $n_{0}$ such that $a_{n}\in \Sigma _{\theta _{n}}$ for all $n\geq n_{0}$, so $a\in
\Sigma _{\theta }$.

\item Let $a\in \overset{\bullet }{
\mathbb{C}}$. If $a\notin \Sigma _{\theta }$, so there exists $\delta >0$, such that $
a\notin \Sigma _{\alpha }$ for all $\alpha \in \Lambda ^{\delta }(\theta
)=\left\{ \alpha \in \lbrack 0,\frac{\pi }{2}[;\,\,\left\vert \theta
-\alpha \right\vert \leq \delta \right\} $
\end{enumerate}
\end{lemma}

\begin{proof}
To prove the first part, it is sufficient to note from the
definition of $\Sigma _{1,\theta }$ such that there exists a unique
value of $\theta $ in $[0,\frac{\pi }{2}[$ such that the real part of the product of the two complex number $\exp (i\theta )$ and $Z$ is zero, where $Z=\dsp \int_{1}^{a}\sqrt{p_{a}\left( z\right) }dz$ and hence $\dsp \Sigma _{1,\theta }\cap \Sigma_{1,\alpha }=\left\{ 1\right\} $. Similarly, $\Sigma _{-1,\theta }\cap\Sigma _{-1,\alpha }=\left\{ -1\right\} $. Finally, $\Sigma _{\blacktriangle,\theta }$ originates from the point $s_{\theta }$ on the real axis and passes through $\Sigma _{1,\theta }\cap \Sigma _{-1,\theta }\subseteq \left\{t_{\theta },e_{\theta }\right\} $. By the previous argument, it is easy to
deduce that $\Sigma _{\blacktriangle ,\alpha }$ and $\Sigma _{\blacktriangle,\theta }$ have no intersection.\\
For the second part, it is sufficient to note that $\underset{
n\rightarrow \infty }{\lim }\Re \left( \int_{1}^{a_{n}}e^{i\theta _{n}}\sqrt{
p_{a_{n}}\left( z\right) }dz\right) =\Re \left( \int_{1}^{a}e^{i\theta }
\sqrt{p_{a}\left( z\right) }dz\right) $.\\
Finally, $\Sigma _{\theta }$ divides the complex plane $\mathbb{C}$ into $n_{0}$ simply connected domains $\left( \Omega _{i}\right) _{1\leq
i\leq n_{0}}$. If $K$ is a compact subset of $\mathbb{C}
\backslash \Sigma _{\theta }$, then $K\subset \Omega _{i_{0}}$ for some $i_{0}\in \left\{ 1,...,n_{0}\right\}$. The function 
\begin{eqnarray*}
\Omega _{i_{0}} &\longrightarrow &\mathbb{R}\\
a &\longmapsto &\Re \left( \int_{1}^{a}e^{i\theta }\sqrt{p_{a}\left(
z\right) }dz\right) 
\end{eqnarray*}
is continuous according to the previous part of this lemma and has a constant sign. Suppose that for $a\in K$, we have $\Re \left( \int_{1}^{a}e^{i\theta }\sqrt{
p_{a}\left( z\right) }dz\right) >0$. For a suitable branch of the square
root, we deduce that there exist $\delta _{K},A_{K}>0$ such that $\Re \left(
\int_{1}^{a}e^{i\theta }\sqrt{p_{a}\left( z\right) }dz\right) \geq A_{K}>0$
and so $\left\vert \arg (\int_{1}^{a}e^{i\theta }\sqrt{p_{a}\left( z\right) }dz)\right\vert \leq \frac{\pi }{2}-2\delta _{K}$ for all $a\in K$ ($K$ compact).\\
Let $\Lambda ^{\delta _{K}}(\theta )=\left\{ \alpha \in \lbrack 0,
\frac{\pi }{2}[\text{; }\left\vert \theta -\alpha \right\vert \leq \delta
_{K}\right\}$, for $\alpha \in \Lambda ^{\delta _{K}}(\theta )$ we have $
\frac{-\pi }{2}<\frac{-\pi }{2}+\delta _{K}\leq \alpha +\arg (\int_{1}^{a}\sqrt{p_{a}\left( z\right) }dz)\leq \frac{\pi }{2}-\delta _{K}<\frac{\pi }{2}$ which proves that $a\notin \Sigma _{\alpha }$. By the some way we prove the case $\Re \left( \int_{1}^{a}e^{i\theta }\sqrt{p_{a}\left( z\right) }dz\right) <0$. This completes the proof of the lemma.
\end{proof}

\begin{proof}[Proof of (\protect \ref{main result2})]
Fix $\varepsilon >0$. Let $K$ a compact subset of $\mathbb{C}\backslash \Sigma _{\theta }$. For $a\in K$, all the five half planes are
admissible. By Lemma (\ref{relation sigma}), there exists $\delta _{K}>0$ such that $a\notin $ $\Sigma _{\alpha }$ for all $\alpha \in \Lambda
^{\delta _{K}}(\theta )=\left\{ \alpha \in \lbrack 0,\frac{\pi }{2}[;\,\,
\left\vert \theta -\alpha \right\vert \leq \delta _{K}\right\}$.\\
We denote $\Lambda _{a,\varepsilon }(\theta )=\left\{ \lambda \in \mathbb{C}^{\ast };\,\,\left\vert \arg \lambda -\theta \right\vert \leq \delta_{K};\,\,\left\vert \lambda \right\vert \geq r_{\varepsilon }\right\}$. It is obvious so that all half planes still admissible for all $\lambda \in \Lambda _{a,\varepsilon }(\theta )$.\\
Let $H_{-1}$ be a half-plane domain such that $\partial H_{-1}$ contains the turning point $-1$.\\
From $-1$ emanate three Stokes lines $l_{-1}^{0},l_{-1}^{1}$ and $l_{-1}^{2}$. Let us make a cut along one of these three lines, says $l_{-1}^{2}$, and remove an $\varepsilon $ neighborhood of it, such that $\Re(h(\lambda ,-1,z))>0$. Formula \eqref{uniform asymp} is then applicable for fixed $z$ in $H_{-1}^{\varepsilon }$ and $\lambda \rightarrow \infty $ in $\Lambda _{a,\varepsilon }(\theta )$ (or for $\lambda$ fixed and $z\rightarrow \infty$ in $H_{-1}^{\varepsilon}$). $H_{-1}$ borders a band domain of type  $B_{0}$ such that $\partial B_{0}$ contains $l_{-1}^{0}$, $l_{-1}^{2}$ and another turning point, says $a$. Any point in $B_{0}$ can be joined to any point of $H_{-1}^{\varepsilon }$ by an admissible curve and so \eqref{uniform asymp} is applicable in $H_{-1}^{\varepsilon }\cup B_{0}$.\\
From $1$ emanate three Stokes lines $l_{1}^{0}$, $l_{1}^{1}$ and $l_{1}^{2}$ such that $l_{1}^{0}$, $l_{1}^{1}\in \partial B_{0}$. Let us make a cut along $l_{1}^{2}$ and remove an $\varepsilon $ neighborhood of it. $B_{0}$ borders a band domain of type $B_{1} $ such that $\partial B_{1}$ contains $l_{0}^{0}$, $l_{1}^{2}$ and the third turning point $a$. Any point in $B_{1}$ can be joined with any point in $H_{-1}^{\varepsilon }\cup B_{0}^{\varepsilon }$ by an admissible curve and so \eqref{uniform asymp} is applicable in $H_{-1}^{\varepsilon }\cup B_{0}^{\varepsilon }\cup B_{1}$.\\
From $a$ emanate three Stokes lines $l_{a}^{0}$, $l_{a}^{1}$ and $l_{a}^{2}$ such that $l_{a}^{0}$, $l_{a}^{1}\in \partial B_{1}$. The half plane $H_{a}$, delimited for example by $l_{a}^{1}$ and $l_{a}^{2}$, is admissible to $H_{-1}$ and so \eqref{uniform asymp} is
then applicable in $H_{-1}^{\varepsilon }\cup B_{0}^{\varepsilon }\cup
B_{1}^{\varepsilon }\cup H_{a}$, where $B_{1}^{\varepsilon }$ is obtained
from $B_{1}$ by removing an $\varepsilon $ neighborhood of $l_{1}^{2}$ (
see Figure  \ref{fig:enter-label}).\\
In the same way, we extend the applicability of \eqref{uniform asymp} to the whole complex plane from which neighborhoods of some Stokes lines have been removed. This achieve the proof.
\begin{figure}[th]
    \centering
    \includegraphics[width=0.35\linewidth]{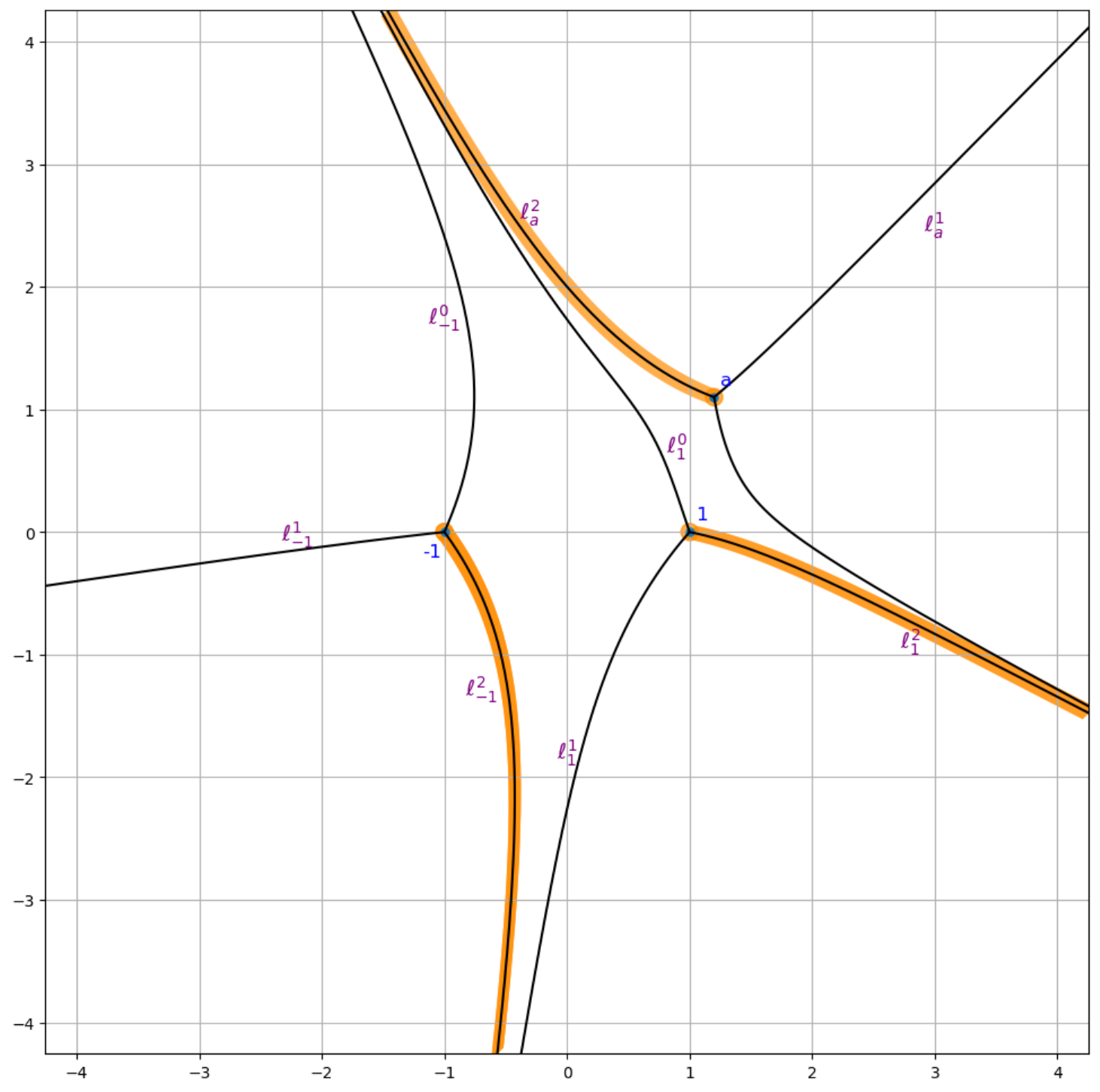}
    \caption{Stokes graph with cuts for $a\in\mathbb{C}\setminus{\mathcal{X}_\theta}$}
\end{figure}
\end{proof}
\begin{lemma}\label{lm1}
    Let $a,b\in \mathbb{C}$, $a\neq b$, and let $C$ be a simple closed curve encircling the line segment  $l_{a,b}$ and oriented as in Figure  \ref{contour} then $$\oint_C {\sqrt{(z-a)(z-b)}dz}=2\int_a^b{\sqrt{(z-a)(z-b)}dz} $$
    \begin{figure}[ht] 
  \centering
  \includegraphics[width=0.5\textwidth]{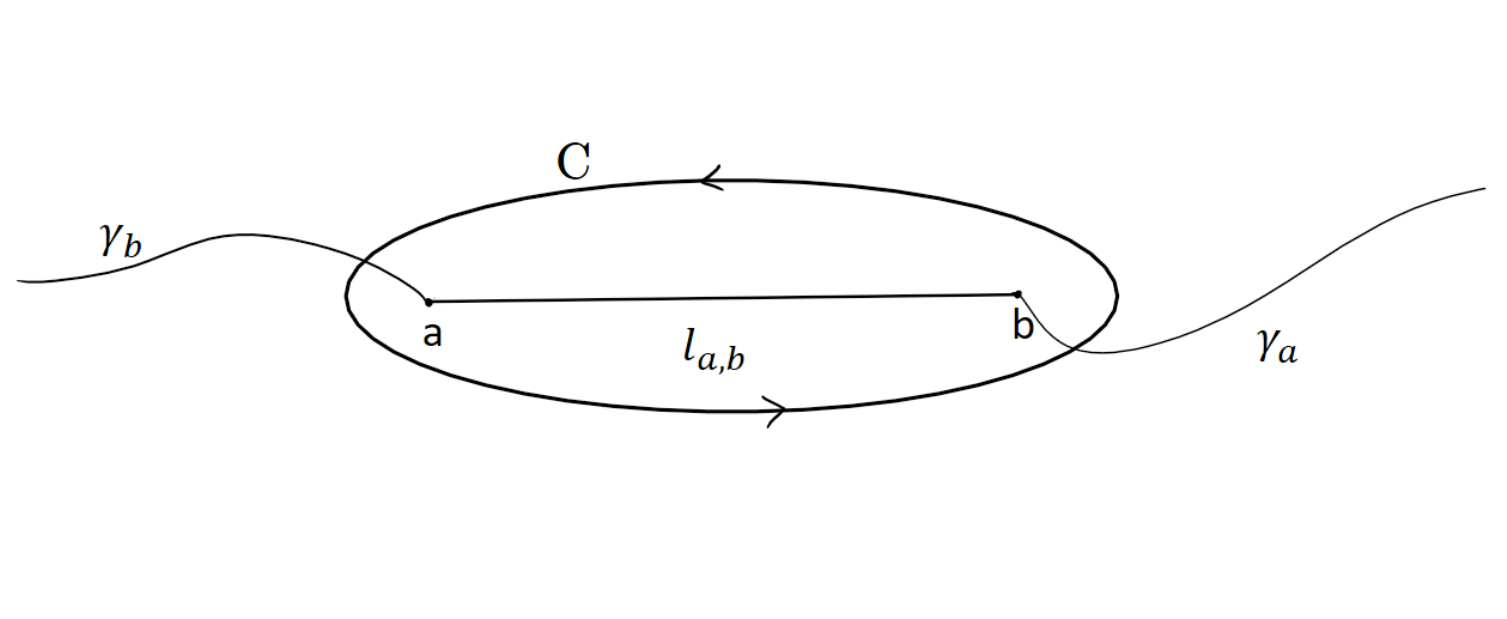}
  \caption{} \label{contour}
\end{figure}
\end{lemma}
\begin{proof}
Let $\gamma_a$ and $\gamma_b$ be two Jordan curves going, respectively, from $a$ to infinity and from $b$ to infinity as shown in Figure \ref{contour}. The function  $\sqrt{(z-a)(z-b)}$ has two holomorphic branches in $\mathbb{C}\setminus (\gamma_a\cup \gamma_b)$, denoted by $f_+$ and $f_-$, where $f_+=-f_-$.\\
Now suppose $f$ be a holomorphic branch of $\sqrt{(z-a)(z-b)}$ in $\mathbb{C}\setminus l_{a,b}$. If $f$ coincides with $f_+$ in one connected component $\mathbb{C}\setminus (\gamma_a\cup \gamma_b\cup l_{a,b})$, then $f$ must coincide with $f_-$ in the other component. Let $x_a$ and $x_b$ represent the intersection points of the contour $C$ with the curves $\gamma_a$ and $\gamma_b$ respectively. Then $$\oint_C {fdz}=\int_{x_a}^{x_b}f_+dz+\int_{x_b}^{x_a}{f_-}dz=2\int_{x_a}^{x_b}f_+dz=2\int_{a}^{b}{f}dz$$

\end{proof}

\begin{proof}[Proof of (\protect \ref{main result1})]
As it was shown in Theorem \ref{main theorem Thabet}, the topology of $\mathbb{C}\diagdown \Sigma _{\theta }$ is invariant for all $\theta \in ]0,\frac{\pi }{2}[$, so for simplicity we will focus on the proof for $\theta =\frac{\pi }{4}$. By the subsection  \ref{sub2}, the equation \eqref{cubic equation} has unique (to within constant multiple) solutions $y_{1}(z,\lambda )$ and  $y_2(z,\lambda )$ which are subdominant in $H'$ and $H''$, respectively. If $\lambda$ is an eigenvalue, then $y_{1}(z,\lambda )=c~y_2(z,\lambda )$.
\begin{enumerate}
\item\label{11}  Let $a\in \mathcal{\chi }_{\frac{\pi }{4} }\setminus \left\{t_{\frac{\pi }{4}}\right\}$ and the Stokes graph be of \textbf{type B}. Without loss of generality, we can suppose that the Stokes graph has a short trajectory connecting $a$ and $1$. In this case, there exists a unique pair of non admissible domains  $(H',H'')$ where $1\in\partial H'$ and $a\in \partial H''$ (see the Stokes graph in Figure  \ref{MDA}) and the eigenvalue problem is given by: 
\begin{equation}
    \begin{cases}\label{eig0}
        -y^{\prime \prime }(z)+\lambda ^{2}(z-a)(z^{2}-1)y(z)=0\\ \lim_{|z|\to +\infty}y(z,\lambda)=0,~~z\in H'\cup H''.
    \end{cases}
\end{equation}  
\begin{itemize}
\item \label{r1}
Let $\varepsilon>0$. Let the branch of $(\sqrt{p_a(z)})_1$ in ${H'}$ be chosen such that $\Re{h(\lambda,a,z)}>0$, and let $\gamma^+$ be an anti-Stokes line in $H'_{\varepsilon}$. Then the solution $y_{1}$ has the asymptotic expansion:
\begin{equation}\label{1}
y_{1}(z,\lambda)=(p_a(z))_1^{-\frac{1}{4}}(e^{-h(\lambda,a,z)}(1+O({\lambda}^{-1})) 
\end{equation}
as  $|\lambda|\to +\infty $, $z$ uniformly in $\gamma^+$. The  domain $D^1_{H',\varepsilon}$ of applicability of \eqref{1}  is obtained by removing an $\varepsilon-$neighborhood of the Stokes lines $l_a^0$, $l_a^1$ and $l_a^2$ and the domain $H''$ lying to the left of $l_a^1$ and $l_a^2$ (see Figure \ref{MDA}). Indeed, for all $z\in D^1_{H',\varepsilon}$ there exists an admissible curve $\gamma$ such that $\gamma(0)=z$ and $\gamma^+\subset\gamma$.
\item The solution $y_{2}$ is defined with the asymptotic expansion:
\begin{equation}\label{2}
y_{2}(z,\lambda)=(p_a(z))_{2}^{-\frac{1}{4}}(e^{h(\lambda,1,z)}(1+O({\lambda}^{-1})) 
\end{equation}
$|\lambda|\to +\infty $, $z$ uniformly in $H_{2,\varepsilon}$, the branch of square root is chosen such that $\Re{h(\lambda,1,z)}<0$, with the same reason in  \ref{r1} the maximal domain $D^2_{H'',\varepsilon}$ applicability of the asymptotic expansion \eqref{2} is given by removing $\varepsilon-$ neighborhood of the cuts $l^0_1,l^1_1,l_1^2$ and the half plane domain $H'$  (see Figure  \ref{MDA-}). 

\begin{figure}[h] 
\begin{minipage}[b]{0.5\linewidth}\centering\includegraphics[width=0.8\textwidth]{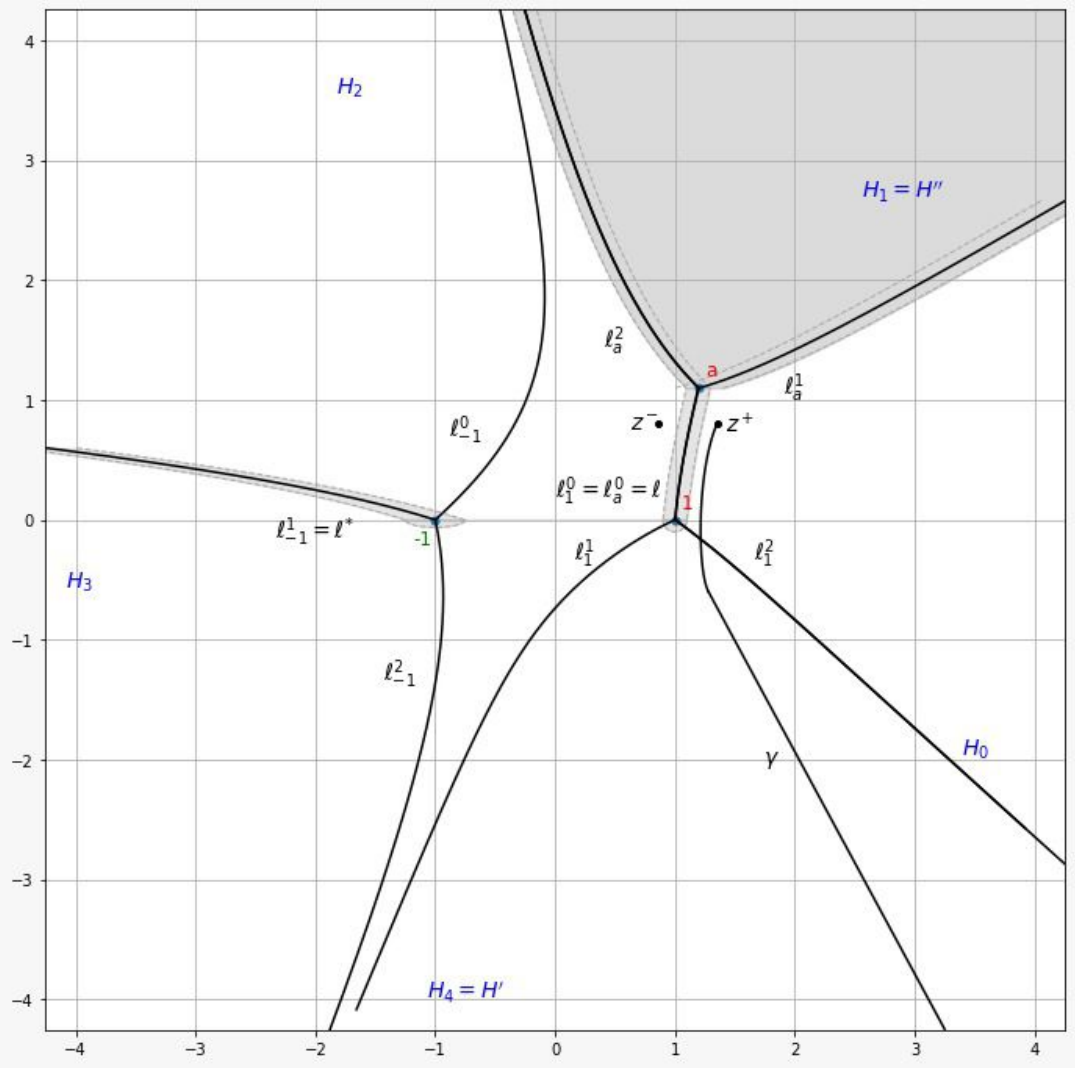}
  \caption{$D^1_{H',\varepsilon} $: The unshaded domain}\label{MDA}
  \end{minipage}\hfill   \begin{minipage}[b]{0.5\linewidth}
  \centering
  \includegraphics[width=0.8\textwidth]{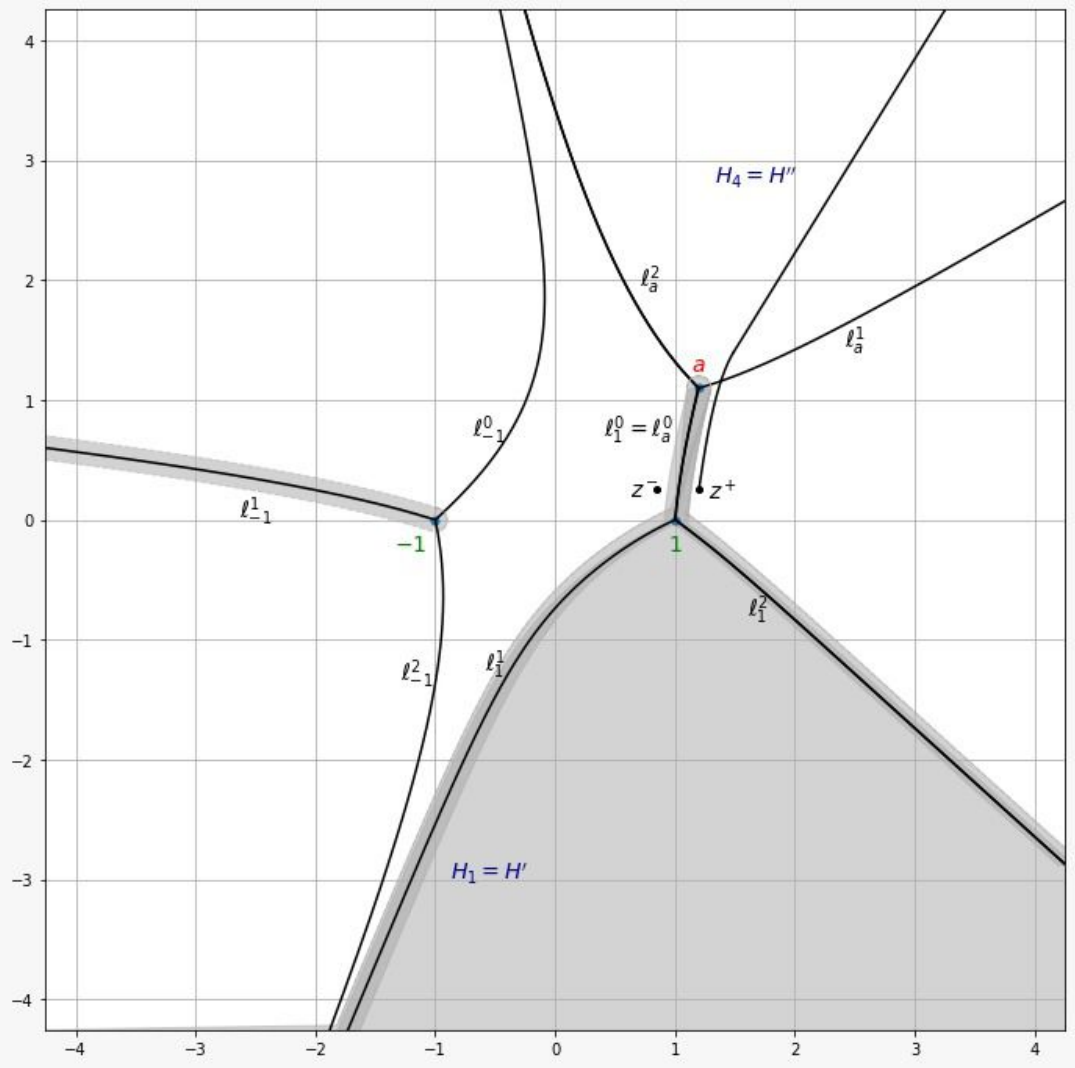}
  \caption{$D^2_{H'',\varepsilon}$: The unshaded domain}\label{MDA-}
  \end{minipage}
\end{figure}

Let $D=D^1_{H',\varepsilon}\cap D^2_{H'',\varepsilon}\cap V$ and $z^+,z^-\in D$, where $V$ is a $2\varepsilon-$neighborhood of $l_a^0$, such that $z^+,z^-$ does not belong to the same connected component of $D$. Then \eqref{1} and \eqref{2} are applicable in $z^+$ and $z^-$, we obtain the equation for the eigenvalue 
\begin{equation}
\label{eve}
\frac{y_{1}(z^+,a,\lambda)y_{2}(z^-,1,\lambda)}{y_{1}(z^-,a,\lambda)y_{2}(z^+,1,\lambda)}=1.
\end{equation} 
While the branches $(\sqrt{p_a(z)})_1\text{ and }(\sqrt{p_a(z)})_2$ are chosen such that $$(\Re{h(\lambda,a,z)})_1(\Re{h(\lambda,a,z)})_2<0 \text{  in }D^1_{H',\varepsilon}\cap D^2_{H'',\varepsilon}$$ then$  (\sqrt{p_a(z)})_1=-(\sqrt{p_a(z)})_2$. Additionally, the branches $(p_a(z)^\frac{1}{4})_1$ and $(p_a(z)^\frac{1}{4})_2$ are chosen such that  $(p_a(z^{+})^\frac{1}{4})_1= i(p_a(z^+)^\frac{1}{4})_2$, then it follows that $(p_a(z^{-})^\frac{1}{4})_1=-i(p_a(z^-)^\frac{1}{4})_2$, from \eqref{eve} we obtain an equation of eigenvalue :
$$
exp(i\pi+2|\lambda|(h(\lambda,1,a))_2)(1+O(\frac{1}{\lambda}))=1 .
$$
From Lemma  \ref{lm1}, $h(\lambda,1,a))$ is purely imaginary equals $\frac{1}{2}\oint_C e^{i\frac{\pi}{4}}\sqrt{p_a(t)}dt$, where $C$ is simple closed curve go around the short trajectory $l_0$.\\
Consequently, the set of eigenvalues $\{\lambda_n,~n\in \mathbb{N}\}$ is discrete, and verifies the asymptotic formula:
$$
|\lambda_n|\underset{n \to \infty}{\sim} (2n-1)\pi (\vert\oint_C e^{i\frac{\pi}{4}}\sqrt{p_a(t)}dt\vert)^{-1}
$$
then this set  can be arranged as  $|\lambda_{n_0}|<|\lambda_{n_0+1}|<...<|\lambda_n|....\rightarrow\infty$.
\end{itemize}

\item Let $a\in \mathcal{\chi }_{\frac{\pi }{4} }\setminus \left\{t_{\frac{\pi }{4}}\right\}$ and the Stokes graph is of \textbf{type BB}, there are precisely two pairs of non admissible half-plane domains, as depicted in Figure \ref{MDA2+}. It is sufficient to concentrate on the eigenvalue defined within one of these pairs :
\begin{equation}
      \begin{cases}
        -y^{\prime \prime }(z)+\lambda ^{2}(z-a)(z^{2}-1)y(z)=0\\ \lim_{|z|\to +\infty}y(z,\lambda)=0,~~z\in H_1\cup H_2
    \end{cases}
\end{equation}
By following the same steps as in the part  \ref{11} of the proof we obtain: 
\begin{itemize}
\item The subdominant solutions  $y_{1}(z,\lambda)$ and $y_{2}(z,\lambda)$ in $H_1$ and $H_2$ respectively are given by:  
\begin{equation}\label{sol3}
\begin{cases}
y_{1}(z,\lambda)=(p_a(z))_{2}^{-\frac{1}{4}}(e^{h(\lambda,a,z)}(1+O({\lambda}^{-1})) \\
y_{2}(z,\lambda)=(p_a(z))_{2}^{-\frac{1}{4}}(e^{-h(\lambda,-1,z)}(1+O({\lambda}^{-1}))
\end{cases}
\end{equation}
Here, the maximal domains of applicability of these expansions are denoted by $H_1^+$ and $H_2^-$, which are the non-shaded domains illustrated in Figures \ref{MDA2+} and \ref{MDA2-}.
\item Substituting \eqref{sol3} in \eqref{eve}, we obtain the equation for the eigenvalues: 
$$
exp(i\pi+2|\lambda|(h\left(\lambda,-1,a)\right)_2)\left[1+O(\frac{1}{\lambda})\right]=1 .
$$
We have the real part of the term $h(\lambda,-1,a)$ equal $\Re(h(\lambda,-1,1))$ which is not vanishing (there is no trajectory  joins $1$ and $-1$), then $\frac{\pi}{4}$ can not be an accumulation direction of \eqref{cubic equation} in this case.
\begin{figure}[h] 
  \begin{minipage}[b]{0.5\linewidth}\centering\includegraphics[width=0.8\textwidth]{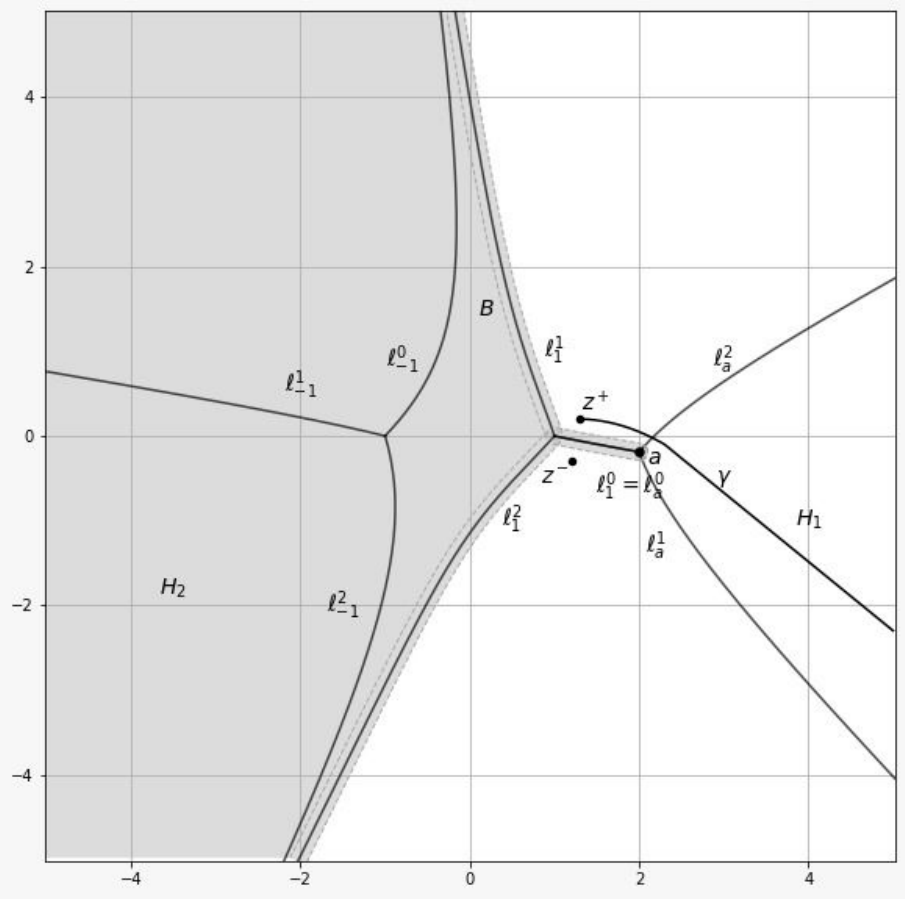}
  \caption{$H_1^+ $: the unshaded domain}\label{MDA2+}
  \end{minipage}\hfill   \begin{minipage}[b]{0.5\linewidth}
  \centering
  \includegraphics[width=0.8
  \textwidth]{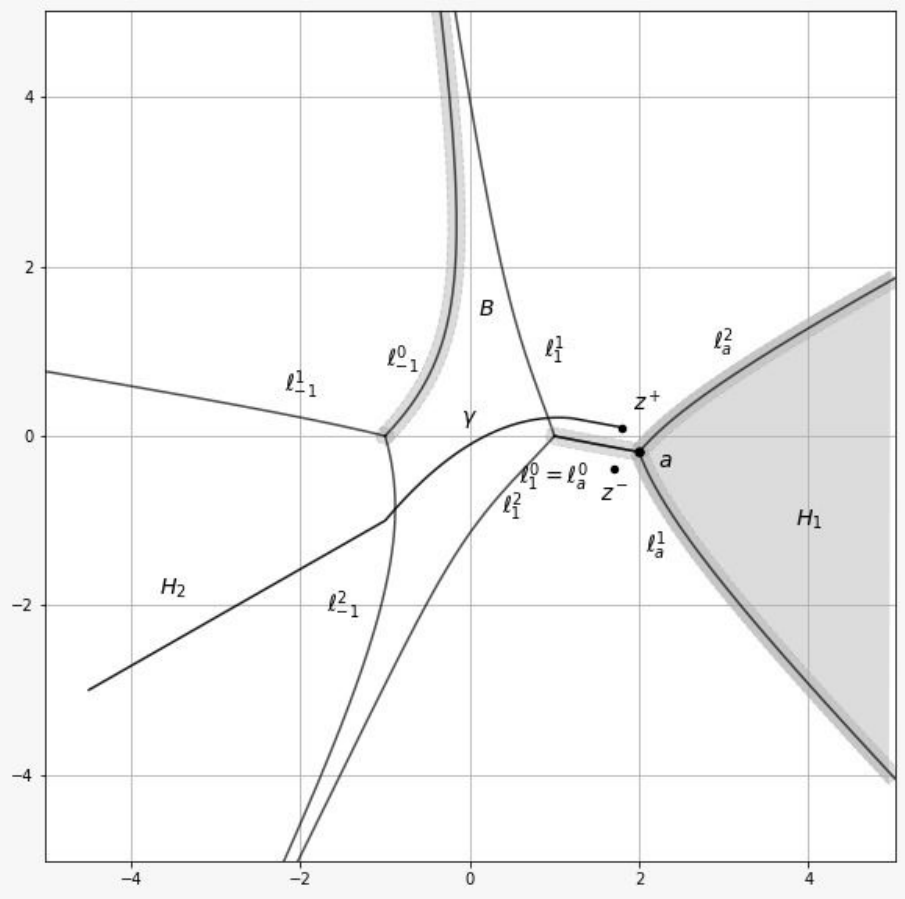}
  \caption{$H_2^-$: the unshaded domain}\label{MDA2-}
  \end{minipage}
\end{figure}

\end{itemize}

\item  Let $a=t_\frac{\pi}{4}$. As seen in Figure  \ref{fig tri} there exists
a unique Stokes complex that splits the complex plane into five half plane domains $H_0$ $H_1$, $H_2$, $H_3$ and $H_4$. 
\begin{figure}[th]
    \centering
    \includegraphics[width=0.4\linewidth]{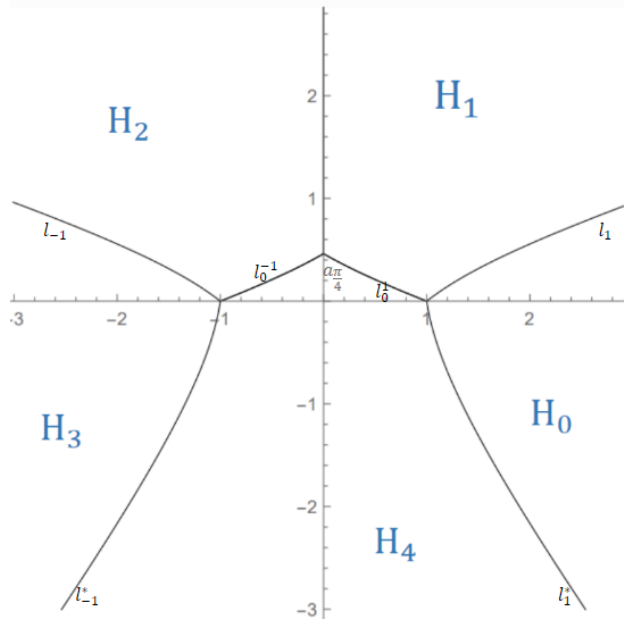}
    \caption{}
    \label{fig tri}
\end{figure}
There are three pairs of non admissible domains, so we have three possible eigenvalue problem:

\begin{equation}
    \begin{cases}\label{eig10}
      \dsp  -y^{\prime \prime }(z)+\lambda ^{2}(z-a)(z^{2}-1)y(z)=0\\ \dsp\lim_{|z|\to +\infty}y(z,\lambda)=0,~~z\in H_0\cup H_2
    \end{cases}
\end{equation}
\begin{equation}
    \begin{cases}\label{eig20}
        -y^{\prime \prime }(z)+\lambda ^{2}(z-a)(z^{2}-1)y(z)=0\\ \lim_{|z|\to +\infty}y(z,\lambda)=0,~~z\in H_1\cup H_3
    \end{cases}\end{equation}    
\begin{equation}
    \begin{cases}\label{eig30}
        -y^{\prime \prime }(z)+\lambda ^{2}(z-a)(z^{2}-1)y(z)=0\\ \lim_{|z|\to +\infty}y(z,\lambda)=0,~~z\in H_0\cup H_3
    \end{cases}\end{equation}    
\eqref{eig10} and \eqref{eig20} are  given with boundary conditions respectively in  $(H_0,H_2)$ and $(H_1,H_3)$, where are connected with unbroken short trajectory. This is similar to the case when $a\in S_{1,\theta_0}$ in  part \ref{11}. We deduce that the spectrum is discrete, accumulating near $L(\frac{\pi}{4})$  and verify,  respectively for $k=1$ and $k=-1$ the asymptotic formula 
$$
|\lambda_n|\underset{n \to \infty}{\sim} (2n-1)\pi (\vert\oint_{C_k} e^{i\frac{\pi}{4}}\sqrt{p_a(t)}dt\vert)^{-1}
$$
where $C_k$ is a simple closed contour go around the short trajectory   $l_0^k$.\\
For the remaining case \eqref{eig30}, where $H_0$ and $H_3$ are jointed with a broken short trajectory. We use transition matrices to obtain asymptotic formula for the eigenvalue. We define the canonical domains $D_{1},D_{2},D_{3}$ and $D_{4}$ by
\[
\begin{array}{lll}
     l_1\subset D_{1} & l_1^{*}\cup l_0^1\cup l_a\subset \partial D_1 & D_{1}=H_0\cup  H_1\\
     l_0^1\subset D_{1} & l_1^{*}\cup l_1\cup l_a\subset \partial D_1 & D_{2}=H_1\cup H_4 \\
     l_0^{-1}\subset D_{1} & l_1^{*}\cup l_a\cup l_{-1}\subset \partial D_1 & D_{3}=H_2\cup H_4\\
     l_{-1}\subset D_{1}& & D_{4}=H_3\cup H_4
\end{array}
\]
Then following this sequence of triples : 
\[
\left( D_{1},1,l_{1}\right)_1 \rightarrow \left( D_{2},1,l_{0}^{1}\right)_2
\rightarrow \left( D_{2},a,l_{0}^{1}\right)_3 \rightarrow \left(
D_{3},a,l_{0}^{-1}\right)_4 \rightarrow \left( D_{3},-1,l_{0}^{-1}\right)_5\rightarrow (D_4,-1,l_1^*)_4.
\]
From the paragraph  \ref{EFSS}, let $\{(u_i,v_i)\}_{1\leq i\leq 5}$ be elementary bases corresponding respectively to the triples above. Suppose that $y(z,\lambda)$ is a solution of \eqref{cubic equation} with asymptotic expansion \eqref{1} in $H_{1,\varepsilon}$. \\
Then $y(z,\lambda)=cv_1$, we will express $y$ in terms of $(u_4,v_4)$:  $y=a u_4+b v_4$
$$\dsp\left( 
\begin{array}{c}
a\left( \lambda \right) \\ 
b\left( \lambda \right)
\end{array}
\right) =K
\begin{pmatrix}
0 & \alpha _{0,1}^{-1}(\lambda) \\ 
1 & i\alpha _{2,1}(\lambda)
\end{pmatrix}
\begin{pmatrix}
0 & e^{-i|h(a,-1,\lambda)| } \\ 
e^{i|h(a,-1,\lambda)|} & 0 
\end{pmatrix}
\begin{pmatrix}
0 & \beta _{0,1}^{-1}(\lambda) \\ 
1 & i\beta _{1,2}(\lambda)
\end{pmatrix}$$
$$\dsp
\begin{pmatrix}
0 & e^{-i|h(1,a,\lambda)|} \\ 
e^{i|h(1,a,\lambda)|} & 0
\end{pmatrix}
\begin{pmatrix}
0 & \gamma _{0,1}^{-1}(\lambda) \\ 
1 & i\gamma _{1,2}(\lambda)
\end{pmatrix}
\begin{pmatrix}
0 \\ 
1
\end{pmatrix}
= \frac{1}{\beta_{1, 2} \gamma_{1, 2}} \allowbreak \begin{pmatrix}
 \Delta_1\\ \Delta_2
\end{pmatrix}
\allowbreak, $$
where 
\begin{equation}
    \begin{cases}
     \dsp \Delta_1=\alpha_{1, 2}^{-1}\dsp 
     e^{i |h(-1, a,\lambda)| + i |h(1, a\lambda)|} \\
  \dsp  \Delta_2=i e^{-i |h(-1, a,\lambda) - i|h(1, a,\lambda)|} (\alpha_{2, 3}\dsp e^{2i (|ih(-1, a,\lambda)| + |h(1, a,\lambda)|)} + \beta^{-1}_{1,3} e^{2 i |h(1, a,\lambda)} + \beta_{1, 2} \gamma^{-1}_{1,3}) 
    \end{cases}
\end{equation}

\[
\begin{array}{llll}
     \lambda \text{ is an eigenvalue of } \eqref{eig30} \iff b(\lambda)=0 \\
     \dsp\iff \alpha_{2, 3} e^{2i (|h(-1, a,\lambda)| + |h(1, a,\lambda)|)} + \beta_{1, 2} \beta_{2, 3} e^{2 i |h(1, a,\lambda)|}=-1+O(\frac{1}{\lambda}) \\
     \dsp\iff e^{i(\vert\lambda\oint_{C_1} e^{i\frac{\pi}{4}}\sqrt{P(t)}dt\vert+\frac{1}{2} \vert\lambda\oint_{C_{-1}} e^{i\frac{\pi}{4}}\sqrt{p_a(t)}dt\vert)}2\cos(\frac{1}{2}~ \vert\lambda\oint_{C_{-1}} e^{i\frac{\pi}{4}}\sqrt{p_a(t)}dt\vert)=\\~~~~~~~~~~-1+O(\frac{1}{\lambda})\\
     \iff\begin{cases}
        \dsp|\lambda_n|\underset{|\lambda_n| \to +\infty}{\sim}(2n+1)\pi (\vert\oint_{C_1}e^{i\frac{\pi}{4}}\sqrt{p_a(t)}dt\vert+\frac{1}{2} \vert\oint_{C_{-1}} e^{i\frac{\pi}{4}}\sqrt{p_a(t)}dt\vert)^{-1}\\ \dsp|\lambda_n|\underset{|\lambda_n| \to +\infty}{\sim}(2\varphi(n)+\frac{2\pi}{3})\pi (\vert\oint_{C_{-1}} e^{i\frac{\pi}{4}}\sqrt{p_a(t)}dt\vert)^{-1}    
        \end{cases}
\end{array}
\]
Then $\lambda_n$ is an eigenvalue of \eqref{eig30} if and only if $\dsp\frac{\vert\oint_{C_{-1}} e^{i\frac{\pi}{4}}\sqrt{p_a(t)}dt\vert}{\vert\oint_{C_{1}} e^{i\frac{\pi}{4}}\sqrt{p_a(t)}dt\vert}=\frac{n}{\varphi(n)}-\frac{1}{2}\in \mathbb{Q}$.\\
Here the symmetry observed in the Stokes graph of $e^\frac{\pi}{2}(z^2-1)(z-a_\frac{\pi}{4})$ implies that    $ \vert\oint_{C_1} e^{i\frac{\pi}{4}}\sqrt{p_a(t)}dt\vert=\vert\oint_{C_{-1}} e^{i\frac{\pi}{4}}\sqrt{p_a(t)}dt\vert$, so the condition of the existence of a solution to the eigenvalue problem \eqref{eig30} is verified. However, it is important to note that this condition does not hold universally  for $a\in\{a_\theta,e_\theta\}$ with $\theta\neq \frac{\pi}{4}$.

\end{enumerate}
\end{proof}

\begin{proof}[Proof of (\protect \ref{domain zeros})]
Suppose that $f_{a,n}(X)=0$ for some $X\in \Pi _{\varepsilon }$, for example 
$X\in D_{H^{\prime },\varepsilon }^{1}$. The formula in subsection \ref{formula1} is
valid for all canonical paths $\gamma _{1,\varepsilon }^{+}$, starting from $X$ to $\infty \exp (i\alpha )$, such that $\Re\left[ \lambda
_{n}\int_{X}^{z}\left( \sqrt{p_{a}(t)}\right) _{1}dt\right] \rightarrow
+\infty $, as $z\in \gamma _{1,\varepsilon }^{+}$. From $-f_{a,n}^{\prime \prime }(z)+\lambda
_{n}^{2}p_{a}(z)f_{a,n}(z)=0$, we have $f_{a,n}^{\prime \prime }(z)\overline{f_{a,n}(z)}=\lambda _{n}^{2}p_{a}(z)\left\vert f_{a,n}(z)\right\vert ^{2}$, and by integrating by parts along $\gamma _{1,\varepsilon }^{+}$ we obtain
\begin{equation}
0=\left[ f_{a,n}^{\prime }(z)\overline{f_{a,n}(z)}\right] _{\gamma
_{1,\varepsilon }^{+}}=\int_{\gamma _{1,\varepsilon }^{+}}\left\vert
f_{a,n}^{\prime }(z)\right\vert ^{2}\overline{dz}+\int_{\gamma
_{1,\varepsilon }^{+}}\left( \lambda _{n}^{2}p_{a}(z)\left\vert
f_{a,n}(z)\right\vert ^{2}\right) dz.
\label{equation-integr}
\end{equation}
The construction of $\gamma _{1,\varepsilon }^{+}$ to obtain a contradiction in (\ref{equation-integr}) is crucial in our proof. Let $D_{1}$ and $D_{1,\varepsilon }$ be canonical domains that contain $H^{\prime }$ as described in Section (\ref{caseB}). For definiteness, we suppose that the cuts are all directed downwards. Let $\phi (z)=\lambda _{n}\int^{z}\left( \sqrt{p_{a}(t)}\right) _{1}dt$ a one-to-one conformal transformation from $D_{1}$ to $\phi(D_{1})$, which are both simply connected domains (see Lemma \ref{first lemma} for details). By differentiation, $\phi^{\prime }(z)=\lambda _{n}\left( \sqrt{p_{a}(z)}\right) _{1}$ for all $z\in D_{1,\varepsilon }$. Now, let $\varrho _{1,\varepsilon }=\phi (\gamma_{1,\varepsilon }^{+})$ a path in $\phi (D_{1,\varepsilon })$ starting from $\phi (X)=\xi _{X}+i\eta _{X}$. By the change of variable $\zeta =\xi +i\eta =\phi (z)$ for all $z\in D_{1,\varepsilon }$, we obtain in (\ref{equation-integr}):
\begin{equation}
0=\int_{\varrho _{1,\varepsilon }}\left\vert f_{a,n}^{\prime }(\phi
^{-1}(\zeta ))\right\vert ^{2}\frac{\overline{d\zeta }}{\overline{\phi
^{\prime }(\phi ^{-1}(\zeta ))}}+\int_{\varrho _{1,\varepsilon }}\left\vert
f_{a,n}(\phi ^{-1}(\zeta ))\right\vert ^{2}\phi ^{\prime }(\phi ^{-1}(\zeta
))d\zeta   
\label{equation-integr2}
\end{equation}
There are two possibilities:
\begin{enumerate}
\item The horizontal path $\left\{ \zeta \in \phi (D_{1,\varepsilon });\Re\zeta
\geq \xi _{X}\text{ and }\Im\zeta =\Im\eta _{X}\right\} $ does not
intersect the cuts. In this case, we take $\varrho _{1,\varepsilon
}:=\left\{ \zeta \in \phi (D_{1,\varepsilon });\Re\zeta \geq \xi _{X}
\text{ and }\Im\zeta =\Im\eta _{X}\right\} $. On $\varrho
_{1,\varepsilon }$, we have $d\zeta =d\xi $ and $\phi ^{\prime }(\phi
^{-1}(\zeta ))=\Re\left[ \phi ^{\prime }(\phi ^{-1}(\zeta ))\right] >0.
$ In fact, the function $\zeta \longmapsto \Re\zeta $ is strictly
increasing on $\varrho _{1,\varepsilon }$. Equation (\ref{equation-integr2}) becomes:
\begin{equation*}
0=\int_{\xi _{X}}^{+\infty }\left\vert f_{a,n}^{\prime }(\phi ^{-1}(\zeta
))\right\vert ^{2}\frac{d\xi }{\Re\left[ \phi ^{\prime }(\phi^{-1}(\zeta ))\right] }+\int_{\xi _{X}}^{+\infty }\left( \left\vert
f_{a,n}(\phi ^{-1}(\zeta ))\right\vert ^{2}\Re\left[ \phi ^{\prime
}(\phi ^{-1}(\zeta ))\right] \right) d\xi 
\end{equation*}
It is obvious that the right-hand side is strictly positive, which lead to a contradiction.
\item If the horizontal path $\left\{ \zeta \in \phi (D_{1,\varepsilon });\Re \zeta \geq \xi _{X}\text{ and }\Im\zeta =\Im\eta _{X}\right\} $
intersects the cuts. In this case, we can choose the path $\varrho
_{1,\varepsilon }$ in $\phi (D_{1,\varepsilon })$ starting from $\phi
(X)=\xi _{X}+i\eta _{X}$ passing through $\Im\phi (X^{\prime })=\xi
_{X^{\prime }}+i\eta _{X^{\prime }}$, where $\eta _{X^{\prime }}>0$, with
two components:
\begin{equation*}
\begin{array}{ll}
\varrho _{1,\varepsilon }^{\bullet }:=\left\{ \zeta \in \phi
(D_{1,\varepsilon });\Re\zeta =\xi _{X}<0\text{ and }\eta _{X^{\prime
}}\geq \Im\zeta \geq \eta _{X}\right\}  \\ 
\varrho _{1,\varepsilon }^{\bullet \bullet }:=\left\{ \zeta \in \phi
(D_{1,\varepsilon });\Re\zeta \geq \xi _{X}\text{ and }\Im\zeta =
\Im\eta _{X^{\prime }}\right\} ;  
\end{array}
\end{equation*}
see the Figure \ref{ph-plane}.
\begin{figure}[h]
    \centering
    \includegraphics[width=0.4\linewidth]{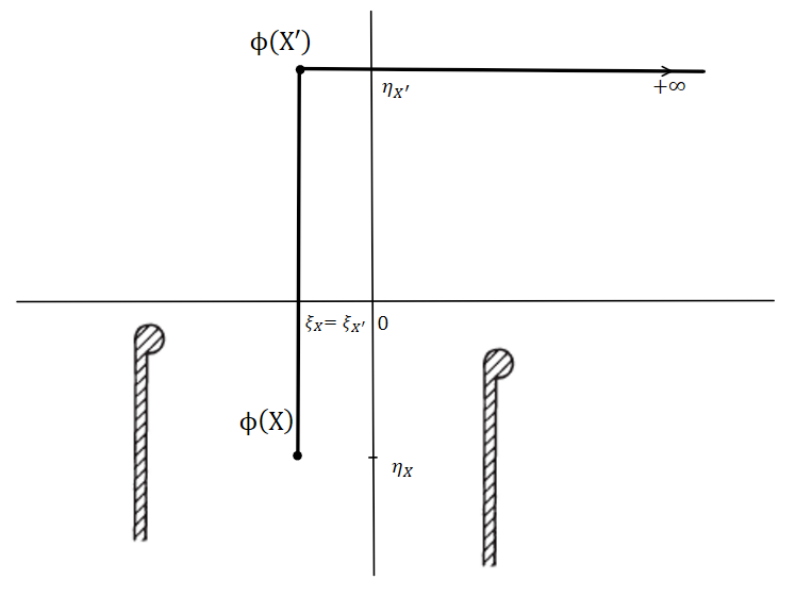}
    \caption{$\zeta=\xi+i\eta=\phi(z)$-plane}
    \label{ph-plane}
\end{figure}
\begin{itemize}
\item In the first component $\varrho _{1,\varepsilon }^{\bullet }$, $d\zeta =id\eta $ and $\phi ^{\prime }(\phi ^{-1}(\zeta ))=i\Im\left[ \phi
^{\prime }(\phi ^{-1}(\zeta ))\right] $, where $\Im\left[ \phi
^{\prime }(\phi ^{-1}(\zeta ))\right] >0$. In fact, the function $\zeta
\longmapsto \Im\zeta $ is strictly increasing from $\eta _{X}$ to $
\eta _{X^{\prime }}$. We have:
\begin{eqnarray*}
\int_{\varrho _{1,\varepsilon }^{\bullet }}\left\vert f_{a,n}^{\prime }(\phi
^{-1}(\zeta ))\right\vert ^{2}\frac{\overline{d\zeta }}{\overline{\phi
^{\prime }(\phi ^{-1}(\zeta ))}} &=&\int_{\eta _{X}}^{\eta _{X^{\prime
}}}\left\vert f_{a,n}^{\prime }(\phi ^{-1}(\zeta ))\right\vert ^{2}\frac{
d\eta }{\Im\left[ \phi ^{\prime }(\phi ^{-1}(\zeta ))\right] } \\
\int_{\varrho _{1,\varepsilon }^{\bullet }}\left\vert f_{a,n}(\phi
^{-1}(\zeta ))\right\vert ^{2}\phi ^{\prime }(\phi ^{-1}(\zeta ))d\zeta 
&=&-\int_{\eta _{X}}^{\eta _{X^{\prime }}}\left( \left\vert f_{a,n}(\phi
^{-1}(\zeta ))\right\vert ^{2}\Im\left[ \phi ^{\prime }(\phi
^{-1}(\zeta ))\right] \right) d\eta 
\end{eqnarray*}
so we have 
\begin{eqnarray*}
&&\int_{\varrho _{1,\varepsilon }^{\bullet }}\left\vert f_{a,n}^{\prime
}(\phi ^{-1}(\zeta ))\right\vert ^{2}\frac{\overline{d\zeta }}{\overline{
\phi ^{\prime }(\phi ^{-1}(\zeta ))}}+\int_{\varrho _{1,\varepsilon
}^{\bullet }}\left\vert f_{a,n}(\phi ^{-1}(\zeta ))\right\vert ^{2}\phi
^{\prime }(\phi ^{-1}(\zeta ))d\zeta  \\
&=&\int_{\eta _{X}}^{\eta _{X^{\prime }}}\left\vert f_{a,n}(\phi ^{-1}(\zeta
))\right\vert ^{2}\Im\left[ \phi ^{\prime }(\phi ^{-1}(\zeta ))\right]
\left( \frac{\left\vert f_{a,n}^{\prime }(\phi ^{-1}(\zeta ))\right\vert ^{2}
}{\left\vert f_{a,n}(\phi ^{-1}(\zeta ))\right\vert ^{2}}\frac{1}{\left( 
\Im\left[ \phi ^{\prime }(\phi ^{-1}(\zeta ))\right] \right) ^{2}}
-1\right) d\eta \\
&=&\int_{\eta _{X}}^{\eta _{X^{\prime }}}\psi^{\bullet}(\zeta=\xi_{X}+i\eta)d\eta
\end{eqnarray*}
where\\
$\psi^{\bullet}(\zeta=\xi_{X}+i\eta)=\left\vert f_{a,n}(\phi ^{-1}(\zeta
))\right\vert ^{2}\Im\left[ \phi ^{\prime }(\phi ^{-1}(\zeta ))\right]
\left( \frac{\left\vert f_{a,n}^{\prime }(\phi ^{-1}(\zeta ))\right\vert ^{2}
}{\left\vert f_{a,n}(\phi ^{-1}(\zeta ))\right\vert ^{2}}\frac{1}{\left( 
\Im\left[ \phi ^{\prime }(\phi ^{-1}(\zeta ))\right] \right) ^{2}}
-1\right)$
\item In the second component $\varrho _{1,\varepsilon }^{\bullet \bullet }$, as in the first case we have:
\begin{eqnarray}
&&\int_{\varrho _{1,\varepsilon }^{\bullet \bullet }}\left\vert
f_{a,n}^{\prime }(\phi ^{-1}(\zeta ))\right\vert ^{2}\frac{\overline{d\zeta }
}{\overline{\phi ^{\prime }(\phi ^{-1}(\zeta ))}}+\int_{\varrho
_{1,\varepsilon }^{\bullet \bullet }}\left\vert f_{a,n}(\phi ^{-1}(\zeta
))\right\vert ^{2}\phi ^{\prime }(\phi ^{-1}(\zeta ))d\zeta \\
&=&\int_{\xi _{X}}^{+\infty }\psi^{\bullet\bullet}(\zeta=\xi+i\eta_{X'})d\xi.
\end{eqnarray}
where
\[
\psi^{\bullet\bullet}(\zeta=\xi+i\eta_{X'})=\left\vert f_{a,n}(\phi ^{-1}(\zeta
))\right\vert ^{2}\Re\left[ \phi ^{\prime }(\phi ^{-1}(\zeta ))\right]
\left( \frac{\left\vert f_{a,n}^{\prime }(\phi ^{-1}(\zeta ))\right\vert ^{2}
}{\left\vert f_{a,n}(\phi ^{-1}(\zeta ))\right\vert ^{2}}\frac{1}{\left( 
\Re\left[ \phi ^{\prime }(\phi ^{-1}(\zeta ))\right] \right) ^{2}}
+1\right)
\]
\end{itemize}
By \eqref{WKB theorem} and \eqref{WKB derivative},
\begin{eqnarray}
\frac{\left\vert
f_{a,n}^{\prime }(z)\right\vert^{2} }{\left\vert f_{a,n}(z)\right\vert^{2} }
&=&\left\vert \phi ^{\prime }(z)\right\vert ^{2}+\epsilon (z)\\
\end{eqnarray}
 where $\epsilon (z)\in\mathbb{R} \rightarrow 0$ as $z\rightarrow
\infty $, $z\in \gamma _{1,\varepsilon }^{+}$, which implies that:
\begin{equation}
\begin{cases}
\psi^{\bullet}(\zeta=\xi_{X}+i\eta)=\left\vert f_{a,n}(\phi ^{-1}(\zeta
))\right\vert ^{2}\Im\left[ \phi ^{\prime }(\phi ^{-1}(\zeta ))\right]
(\frac{\epsilon(\phi^{-1}(\zeta))}{\left\vert\Im\phi^{\prime}(\phi^{-1}(\zeta))\right\vert^{2}})\\
 \psi^{\bullet\bullet}(\zeta=\xi+i\eta_{X'})=\left\vert f_{a,n}(\phi ^{-1}(\zeta
))\right\vert ^{2}\Re\left[ \phi ^{\prime }(\phi ^{-1}(\zeta ))\right]
(2+\frac{\epsilon(\phi^{-1}(\zeta))}{\left\vert\Re\phi^{\prime}(\phi^{-1}(\zeta))\right\vert^{2}}).
\end{cases}  
\end{equation}
From:
\begin{equation}
    \begin{cases}
      (\dsp\frac{\epsilon(\phi^{-1}(\zeta))}{\dsp\left\vert\Im\phi^{\prime}(\phi^{-1}(\zeta))\right\vert^{2}})\rightarrow 0, \Im\zeta=\eta\rightarrow +\infty\\ 
      (\dsp \frac{\epsilon(\phi^{-1}(\zeta))}{\left\vert\Re\phi^{\prime}(\phi^{-1}(\zeta))\right\vert^{2}})\rightarrow 0, \xi=\Re\zeta\rightarrow +\infty,\eta=\Im\zeta\rightarrow+\infty\\
      \Re\left[ \phi ^{\prime }(\phi ^{-1}(\zeta ))\right]\rightarrow+\infty, \xi=\Re\zeta\rightarrow +\infty,\eta=\Im\zeta\rightarrow+\infty\\
     \left\vert \Im\left[ \phi ^{\prime }(\phi ^{-1}(\zeta ))\right]\right\vert \leq M, \eta_{X}\leq\eta=\Im\zeta\leq\eta_{X'}
    \end{cases}
\end{equation}
we can choose $\eta_{X^{\prime }}$ large enough such that the right hand of (\ref{equation-integr2}) is again strictly positive, which leads to a
contradiction.
\end{enumerate}
By the same techniques, if $X\in \Pi _{\varepsilon }\backslash D_{H^{\prime },\varepsilon }^{1}$, we construct $\gamma _{2,\varepsilon }^{+}$ with an unbounded component from anti-Stokes line in $H^{^{\prime \prime }}$, which completes the proof.
\end{proof}

\section{Conclusions}\label{summary}
\begin{enumerate}
    \item In our work, we study all possible geometric configurations of cubic oscillator with three simple turning points. A natural question is about the cubic oscillator with double turning points which correspond to the cases $a\in \{-1,1\}$ in our work. Do we have a "continuity version" of the spectrum and zero locations as $a\to\{-1,1\},a\in\Sigma _{\theta }$?

    \item Theorem \ref{main result1} gives a necessary and sufficient condition guaranteeing that there exist an eigenvalue problem \ref{zeros-EP} with infinitely many eigenvalues which belong to some accumulation ray. This give an answer to \cite[Problem 2]{shapiroev} in the case of the degree of the polynomial potential is $3$. A similair results was obtained in \cite{shin3} in the case of Sturm-Liouville problem with PT-symmetric potential (which correspond to $\theta=\frac{\pi}{4}$ in our work). 

    \item In Subsection \ref{topology Stokes}, we had mentioned only the first parts of the asymptotics. In fact, we can extand the asymptotics WKB formulas (\ref{uniform asymp}) to:
    \begin{equation}
y_{l}(z,\lambda )=(p_a(z))^{-\frac{1}{4}}\exp (\pm h(\lambda,z_{0},z))[1+
 \underset{k=1}{\overset{+\infty }{\sum }}(\frac{b_{k}(z)}{\lambda^{k}})]
\end{equation}
where $l\in\{1,2\}$ and ${b_{k}(z)}$ depend not in $\lambda$, and are determined from the construction of $\phi _{l}(z,\lambda
)$ (see \ref{volterra}). 

A naturel question in the excat WKB method (WKB method based on Borel resummation technics) (see\cite{kawai}), is to establish a necessary and sufficient conditions guaranteenig the Borel summability of WKB solutions. For second order ODE with polynomial potentials, it was assumed to have no finite Stokes line (see\cite[Proposition 2.12]{kawai}) as sufficient condition. In our case, the location of $a\in \mathcal{\chi }_{\theta }$ \eqref{chi theta} is crucial to have this condition. But it still the problem whether the WKB solutions are Borel summable as $a\in\Sigma _{\theta }\setminus\mathcal{\chi }_{\theta }$ (which means that $\mathcal{\Re }\int_{l}e^{i\theta }\sqrt{p_{a}(z)}dt=0$ for all Jordan curve $l$ connecting two turning points, and there exist not a finite Stokes line connecting this two turning points)?

\item From the works of Sibuya and Hille \cite{sibuya,Hille,Hille2}, the infinite zeros of non trivial solutions to second order ODE with polynomials coefficients, tends to accumulate, as $\lambda$ fixed and large $z$, in some Stokes rays defined by:
\begin{equation}
  R_{j}=\{z\in\mathbb{C}; \arg(z)=\alpha_{j}\}
\end{equation}
where $\alpha_{j}$ is a critical directions (see \ref{critical directions}). Later this proposition was corrected by Bank \cite{Bank1}. In fact, It was proven that the infinite zeros tend to accumulate near a translate ray:
\begin{equation}
  R^{c}_{j}=\{z\in\mathbb{C}; \arg(z+c)=\alpha_{j}\} 
\end{equation}
for some $c\in\mathbb{C}$. For a detailed exposition on the topic the reader can see \cite{zemirni}.

Ou results in sections \ref{zeros section}), gives another justification to result in \cite{Bank1} in the case of cubic oscillator. In fact, as $\lambda$ fixed and $z$ large, infinite zeros of subdominant solutions (the set $\mathcal{Z}_{a,n}^{unb}$ ( see Corollary\ref{zeros corollary})) will accumulate near an infinite Stokes line asymptotic to some $R_{j}(\varepsilon)$ and so near a translated ray. 

Regarding the finite number of zeros (the set $\mathcal{Z}_{a,n}^{b}$ (see corollary\ref{zeros corollary}), a conjecture on the interlacing of zeros was introduced in \cite{Bender}. We pretend that the geometry of the finite Stokes line will be crucial in the proof of this conjecture even for higher degree polynomials potentials.
\end{enumerate}

\begin{acknowledgement}
This work is partially supported by the research laboratory LR17ES11 from the Faculty of Sciences of Gabes, Tunisia. Mondher Chouikhi would like to thank professor Alexandre Eremenko for indication to reference \cite{trinh}.
\end{acknowledgement}

\section*{Declarations}
\begin{itemize}
    \item \textbf{Conflict of interest}:
The authors have no Conflict of interest to declare in this work.
\item \textbf{Data availability statement}: No datasets were generated or analyzed during the current study.
\end{itemize}
\section{Appendix\label{Appendix}}
In this section, we will give a proof of (\ref{uniform asymp}). The proof is based on classical Liouville transformation, following \cite[chapter6]{olver} (see also \cite{heading}). In \cite{fed1} and \cite{fed2}, the authors' construction of asymptotic solutions is different.\\
Fix $\varepsilon >0$. Recall some notations from  Subsection \ref{topology Stokes}. Let $H$ be a half plane such that $\partial H$ contains a turning point $z_{0}\in \left\{ \pm 1,a\right\} $ and $D$ be a canonical domain that contains $H$.  For $\theta \in \lbrack 0,\frac{\pi }{2}[$ and $\lambda =r\exp (i\theta )\in \mathbb{C}^{\ast }$, the function $h(\theta ,z_{0},z)=\exp (i\theta )\int_{z_{0}}^{z}\sqrt{(t-a)(t^{2}-1)}dt$ is multi-valued. Fix a branch $h$ of $h(\theta,z_{0},z)$ that maps $D$ conformally onto the plane with vertical cuts, and such that $h(H)=\left\{ \zeta :\Re\zeta >0\right\} $, and let $D_{\varepsilon }$ the preimage of $h(D)$ with $\varepsilon $-neighborhoods of the cuts and $\varepsilon $-neighborhoods of the turning points removed. For every point $z$ in $D_{\varepsilon }$, there exists an infinite canonical
path $\gamma _{\varepsilon }^{+}$ from $z$ to $\infty $, such that $\Re
(h(t))\uparrow +\infty $ dor all $t\in \gamma _{\varepsilon }^{+}$ as $
t\longrightarrow \infty $. 
\begin{theorem}
\label{WKB theorem} 
With $H,D_{\varepsilon }$ and $h$ as above, the ODE (\ref{first ODE}) admits (up to constant multiple) a solution $y_{2}(z,\lambda )$, which satisfies
\begin{equation}
y_{2}(z,\lambda )=(p_{a}(z))^{-\frac{1}{4}}\exp (-h(\lambda
,z_{0},z))[1+\phi _{2}(z,\lambda )],  
\label{F2}
\end{equation}
where $\left\vert \phi _{2}(z,\lambda )\right\vert \leq c_{2}\dsp\frac{
r_{\varepsilon }}{r-r_{\varepsilon }}$ for all $z\in \gamma _{\varepsilon
}^{+}$. The constant $c_{2}$ independent of $\lambda $ and 
\begin{equation*}
r_{\varepsilon }=\sup \left( \int_{\gamma _{\varepsilon }^{+}}\left\vert 
\frac{5}{16}\frac{\left( p_{a}^{\prime }(t)\right) ^{2}}{\left(
p_{a}(t\right) )^{3}}-\frac{\left( p_{a}^{\prime \prime }(t)\right) }{
4\left( p_{a}(t)\right) ^{2}}\right\vert \sqrt{\left\vert
p_{a}(t)\right\vert }\left\vert dt\right\vert \right) ,
\end{equation*}
where the supremum is taken over all infinite canonical paths $\gamma
_{\varepsilon }^{+}$ in  $D_{\varepsilon }$.
\end{theorem}
Before proving this theorem, we need these two lemmas.
\begin{lemma}
\label{first lemma}$h$ is a one-to-one conformal map from $D$ to $h(D)$.
\end{lemma}

\begin{proof}
$h(z_{1})=h(z_{2})\Leftrightarrow \Re(h(z_{1}))=\Re(h(z_{2}))$
and $\Im(h(z_{1}))=\Im(h(z_{2}))$. This implies that $z_{1}$ and 
$z_{2}$ are two points from the same horizontal and vertical trajectories of the quadratic differential $\exp (2i\theta )(z^{2}-1)(z-a)dz^{2}$. From the local and global behavior of trajectories of a polynomial quadratic differential (the intersection of horizontal and vertical trajectories is a single point), we deduce that $z_{1}=z_{2}$. To prove that $h$ is conformal, it is sufficient to note that $h^{\prime }(z)=\exp (i\theta )\sqrt{(z-a)(z^{2}-1)}\neq 0$ for all $z\in D$.
\end{proof}

\begin{lemma}
\label{integral convergence}
The integral $\dsp\int_{\gamma _{\varepsilon
}^{+}}\left\vert \frac{5}{16}\frac{\left( p_{a}^{\prime }(t)\right) ^{2}}{
\left( p_{a}(t\right) )^{3}}-\frac{\left( p_{a}^{\prime \prime }(t)\right) }{
4\left( p_{a}(t)\right) ^{2}}\right\vert \sqrt{\left\vert
p_{a}(t)\right\vert }\left\vert dt\right\vert $ is convergent.
\end{lemma}
\begin{proof}
\begin{eqnarray*}
\left\vert \frac{5}{16}\frac{\left( p_{a}^{\prime }(t)\right) ^{2}}{\left(
p_{a}(t\right) )^{3}}-\frac{\left( p_{a}^{\prime \prime }(t)\right) }{
4\left( p_{a}(t)\right) ^{2}}\right\vert \sqrt{\left\vert
p_{a}(t)\right\vert } &=&\left\vert \frac{5}{16}\left( \frac{p_{a}^{\prime
}(t)}{p_{a}(t)}\right) ^{2}-\frac{p_{a}^{\prime \prime }(t)}{4p_{a}(t)}
\right\vert \left( \left\vert p_{a}(t)\right\vert \right) ^{\frac{-1}{2}} \\
&\leq &\left( \left\vert \frac{5}{16}\left( \frac{p_{a}^{\prime }(t)}{
p_{a}(t)}\right) ^{2}\right\vert +\left\vert \frac{p_{a}^{\prime \prime }(t)
}{4p_{a}(t)}\right\vert \right) \left( \left\vert p_{a}(t)\right\vert
\right) ^{\frac{-1}{2}}
\end{eqnarray*}
for all $t\in \gamma _{\varepsilon }^{+}$. The infinite canonical path $
\gamma _{\varepsilon }^{+}$ start at finite point in $D_{\varepsilon }$ and
end to infinity, so we should estimate this integral near $\infty $. Let $R>1
$ large enough such that $\left\vert a\right\vert \leq $ $R$. Writing
\begin{eqnarray*}
\frac{p_{a}^{\prime }(t)}{p_{a}(t)} &=&\frac{1}{t-1}+\frac{1}{t+1}+\frac{1}{
t-a}=\tau (t) \\
\frac{p_{a}^{\prime \prime }(t)}{p_{a}(t)} &=&\left( \frac{p_{a}^{\prime }(t)
}{p_{a}(t)}\right) ^{\prime }+\left( \frac{p_{a}^{\prime }(t)}{p_{a}(t)}
\right) ^{2}=\tau ^{\prime }(t)+\tau ^{2}(t)
\end{eqnarray*}
we have for  $\left\vert t\right\vert >R$, $\left\vert \tau (t)\right\vert
\leq \frac{3}{\left\vert t\right\vert -R}$ and $\left\vert \tau ^{\prime
}(t)\right\vert \leq \frac{3}{\left( \left\vert t\right\vert -R\right) ^{2}}.
$ Consequently, 
\begin{equation*}
\left\vert \frac{5}{16}\frac{\left( p_{a}^{\prime }(t)\right) ^{2}}{\left(
p_{a}(t\right) )^{3}}-\frac{\left( p_{a}^{\prime \prime }(t)\right) }{
4\left( p_{a}(t)\right) ^{2}}\right\vert \sqrt{\left\vert
p_{a}(t)\right\vert }\leq \frac{M}{\left( \left\vert t\right\vert -R\right)
^{2}\left\vert t\right\vert ^{\frac{3}{2}}}
\end{equation*}
which guaranteed the convergence of the integral. This proves the lemma. 
\end{proof}

\begin{proof}[Proof of (\protect\ref{WKB theorem})]
By (\ref{first lemma}), we can use the Liouville transform 
\begin{equation*}
W(\zeta )=p_{a}^{\frac{1}{4}}(h^{-1}(\zeta ))y(h^{-1}(\zeta ))
\end{equation*}
for $h(z)=\exp (i\theta )\int_{z_{0}}^{z}\sqrt{(t-a)(t^{2}-1)}dt=\zeta
\in h(D_{\varepsilon })$. Then the differential equation (\ref{first ODE})
becomes
\begin{equation}
\overset{\bullet \bullet }{W}(\zeta )=r^{2}W(\zeta )+\psi (\zeta )W(\zeta )
\label{WKB1}
\end{equation}
where 
\begin{eqnarray*}
\psi (\zeta ) &=&\alpha \circ h^{-1}(\zeta )=\exp (-i\theta )\left( \frac{5}{
16}\frac{\left( p_{a}^{\prime }\right) ^{2}}{\left( p_{a}\right) ^{3}}-\frac{
\left( p_{a}^{\prime \prime }\right) }{4\left( p_{a}\right) ^{2}}\right)
\circ \left( h^{-1}(\zeta )\right)  \\
\alpha (z) &=&\exp (-i\theta )\left( \frac{5}{16}\frac{\left( p_{a}^{\prime
}\right) ^{2}}{\left( p_{a}\right) ^{3}}-\frac{\left( p_{a}^{\prime \prime
}\right) }{4\left( p_{a}\right) ^{2}}\right) 
\end{eqnarray*}
where the prime stands for the differentiation with respect to $z$ and the $
\bullet $ for the differentiation with respect to $\zeta $. 

In (\ref{WKB1}) we substitute 
\begin{equation*}
W(\zeta )=\exp (-r\zeta )(1+w(\zeta ))
\end{equation*}
and obtain
\begin{equation}
\overset{\bullet \bullet }{w}(\zeta )-2\overset{\bullet }{rw}(\zeta )=\psi
(\zeta )(w(\zeta )+1)  \label{WKB2}
\end{equation}
To solve this inhomogeneous differential equation, the term $\psi (\zeta
)(w(\zeta )+1)$ is regarded a given function. Applying the method of
variation of constants, we obtain 
\begin{equation}
w(\zeta )=\frac{1}{2\lambda }\int_{\zeta }^{\infty }\left[ \left( 1-\exp
(2r(\zeta -t))\right) \psi (t)(w(t)+1)\right] dt.  \label{WKB3}
\end{equation}
Conversely, it is easy to verify directly that every analytic solution of
this Volterra integral equation satisfies (\ref{WKB2}), for details see \cite
[chapter 6]{olver}. We solve this integral equation by the method of successive approximations:
set $w_{0}=0$ and 
\begin{equation}
w_{n}(\zeta )=\frac{1}{2\lambda }\int_{\zeta }^{\infty }\left[ \left( 1-\exp
(2r(\zeta -t))\right) \psi (t)(w_{n-1}(t)+1)\right] dt  \label{volterra}
\end{equation}
where $n=1,2,3..$. Let $\left\Vert w_{n}\right\Vert _{\gamma _{\varepsilon
}^{+}}=\underset{\zeta \in h(\gamma _{\varepsilon }^{+})}{\sup }\left\vert
w_{n}(\zeta )\right\vert $. Note that for $\zeta \in h(\gamma _{\varepsilon
}^{+})$, we have $\xi _{0}\leq \Re\zeta \rightarrow +\infty $ for some
real $\xi _{0}$ and $\left\vert \Im\zeta \right\vert \leq c$. By
definition of infinite canonical path (\ref{admissible remark}), $t\mapsto 
\Re(t)$ is a non-decreasing function on $\gamma _{\varepsilon }^{+},$
so we can assume that $\Re(\zeta -t)\leq 0$. Then we have
\begin{equation*}
\left\Vert w_{n+1}-w_{n}\right\Vert _{\gamma _{\varepsilon }^{+}}\leq \frac{1
}{r}\left\Vert w_{n}-w_{n-1}\right\Vert _{\gamma _{\varepsilon
}^{+}}\int_{\gamma _{\varepsilon }^{+}}\left\vert \alpha (t)\right\vert 
\sqrt{\left\vert p_{a}(t)\right\vert }\left\vert dt\right\vert 
\end{equation*}
Let $\digamma $ denote the set of all infinite canonical path $\gamma
_{\varepsilon }^{+}$ in $D_{\varepsilon }$ and $0<s<\inf \left\{ \left\vert
z_{1}-z_{0}\right\vert ;z_{1}\in \digamma ,z_{0}\in \Upsilon \right\} ,$
where $\Upsilon $ is the set of turning points. By Lemma \ref{integral
convergence}, the integrals  $\left\{ \int_{\gamma _{\varepsilon
}^{+}}\left\vert \alpha (t)\right\vert \sqrt{\left\vert p_{a}(t)\right\vert }
\left\vert dt\right\vert ,\gamma _{\varepsilon }^{+}\in \digamma \right\} $
are  bounded by a constant which depends only on $s$, therefore $
r_{\varepsilon }=\underset{\gamma _{\varepsilon }^{+}\in \digamma }{\sup }
\left( \int_{\gamma _{\varepsilon }^{+}}\left\vert \alpha (t)\right\vert 
\sqrt{\left\vert p_{a}(t)\right\vert }\left\vert dt\right\vert \right) $ is
well defined and finite. By induction, we deduce that 
\begin{equation}
\left\Vert w_{n+1}-w_{n}\right\Vert _{\gamma _{\varepsilon }^{+}}\leq \frac{
\left( r_{\varepsilon }\right) ^{n}}{r^{n}}\left\Vert w_{1}\right\Vert
_{\gamma _{\varepsilon }^{+}}  \label{uniform-convergence}
\end{equation}
and so for $\left\vert \lambda \right\vert =r>r_{\varepsilon }$, the series $
\underset{n\geq 1}{\sum }\left( w_{n+1}-w_{n}\right) $ converge uniformly
(on the set $h(K^{\ast })$) to a solution $\phi _{2}$ to the integral
equation (\ref{WKB3}) which satisfies  $\left\vert \phi _{2}(z,\lambda
)\right\vert \leq \underset{n=1}{\overset{+\infty }{\sum }}\left\Vert
w_{n+1}-w_{n}\right\Vert _{\gamma _{\varepsilon }^{+}}\leq c_{2}\frac{
r_{\varepsilon }}{r-r_{\varepsilon }}$ for all $z\in \gamma _{\varepsilon
}^{+}$. This achieves the proof.
\end{proof}
\begin{corollary}\label{asymp derive}
$\phi _{2}^{\prime }(z,\lambda )\rightarrow 0$, as $
z\rightarrow \infty ,z\in \gamma _{\varepsilon }^{+}$ for $\lambda \in
\Lambda _{a,\varepsilon }(\theta )$ or as $\left\vert \lambda \right\vert
\rightarrow +\infty $, uniformly for $z\in D_{\varepsilon }$.
\end{corollary}

\begin{proof}
By the same notations of the previous proof, we have $\phi _{2}(z,\lambda )= 
$ $\phi _{2}(h^{-1}(\zeta ),\lambda )=\underset{n=1}{\overset{+\infty }{\sum 
}}\left( w_{n+1}-w_{n}\right) $. As the serie $\underset{n\geq 1}{\sum }
\left( w_{n+1}-w_{n}\right) $ converge uniformly we deduce: 
\begin{equation*}
\phi _{2}^{\prime }(z,\lambda )=\overset{\bullet }{\phi }_{2}(h^{-1}(\zeta
),\lambda )=\underset{n=1}{\overset{+\infty }{\sum }}\left( \overset{\bullet 
}{w}_{n+1}-\overset{\bullet }{w}_{n}\right)
\end{equation*}
By differentiating (\ref{WKB3},\ref{volterra}), we obtain:
\begin{eqnarray*}
\overset{\bullet }{w}(\zeta ) &=&-\int_{\zeta }^{\infty }\exp (2r(\zeta
-t))\psi (t)(w(t)+1)dt \\
\overset{\bullet }{w}_{n}(\zeta ) &=&-\int_{\zeta }^{\infty }\exp (2r(\zeta
-t))\psi (t)(w_{n-1}(t)+1)dt
\end{eqnarray*}
and so 
\begin{equation*}
\left\Vert \overset{\bullet }{w}_{n+1}-\overset{\bullet }{w}_{n}\right\Vert
_{\gamma _{\varepsilon }^{+}}\leq \left\Vert w_{n}-w_{n-1}\right\Vert
_{\gamma _{\varepsilon }^{+}}\int_{\gamma _{\varepsilon }^{+}}\left\vert
\alpha (t)\right\vert \sqrt{\left\vert p_{a}(t)\right\vert }\left\vert
dt\right\vert \leq r_{\varepsilon }\left\Vert w_{n}-w_{n-1}\right\Vert
_{\gamma _{\varepsilon }^{+}}.
\end{equation*}
which prove that the series $\underset{n\geq 1}{\sum }\left( \overset{\bullet 
}{w}_{n+1}-\overset{\bullet }{w}_{n}\right) $ converges uniformly (on the set 
$h(K^{\ast })$) and 
\begin{equation*}
\left\vert \phi _{2}^{\prime }(z,\lambda )\right\vert \leq \underset{n=1}{
\overset{+\infty }{\sum }}\left\Vert \overset{\bullet }{w}_{n+1}-\overset{
\bullet }{w}_{n}\right\Vert _{\gamma _{\varepsilon }^{+}}\leq c_{2}\frac{
r_{\varepsilon }^{2}}{r-r_{\varepsilon }}
\end{equation*}
for all $z\in \gamma _{\varepsilon }^{+}$. This achieves the proof.
\end{proof}

\begin{remark}

\begin{enumerate}
\item With the same notations as above, and $\gamma _{\varepsilon }^{-}$ is
an infinite canonical path in $D_{\varepsilon }$ such that $h(\theta
,z_{0},z)\downarrow -\infty $, as $z\rightarrow \infty ,z\in \gamma ^{-},$
the ODE (\ref{first ODE}) admit a solution $y_{1}(z,\lambda )$, which
satisfies
\begin{equation*}
y_{1}(z,\lambda )=(p_{a}(z))^{-\frac{1}{4}}\exp (h(\lambda ,z_{0},z))[1+\phi
_{1}(z,\lambda )],
\end{equation*}
where $\left\vert \phi _{1}(z,\lambda )\right\vert \leq c_{1}\frac{
r_{\varepsilon }}{r-r_{\varepsilon }}$ for all $z\in \gamma _{\varepsilon
}^{-}$. The constant $c_{1}$ independent of $\lambda $ and $r_{\varepsilon }$
is as in (\ref{WKB theorem}).
\item Let $\gamma ^{\pm }(\theta )$ infinite canonical paths in the
canonical domain $D_{\varepsilon }$ of the quadratic differential $\exp
(2i\theta )(z^{2}-1)(z-a)dz^{2}$ and $K$ a compact subset of  $\mathbb{C}
\backslash \Sigma _{\theta }$. By (\ref{relation sigma}), it exists $\delta
_{\varepsilon ,K}>0$ such that for all $a\in K$ and $\lambda \in \mathbb{C}^{\ast }$ where $\arg \lambda \in \left\{ \alpha \in \lbrack 0,\frac{\pi }{2}
[;\,\,\left\vert \theta -\alpha \right\vert \leq \delta _{\varepsilon
,K}\right\} $, the quadratic differential $\lambda ^{2}(z^{2}-1)(z-a)dz^{2}$
has a canonical domain in which $\gamma ^{\pm }$ still infinite canonical.
We conclude that WKB asymptotic formulas with dual behaviors still valid
for $\lambda $ in a sector $\Lambda _{a,\varepsilon }(\theta
)=\left\{ \lambda \in \mathbb{C}^{\ast };\,\,\left\vert \arg \lambda -\theta \right\vert \leq \delta
_{\varepsilon ,K};\,\,\left\vert \lambda \right\vert >r_{\varepsilon
}\right\} $ and all $a\in K$. This result improves the validity of WKB method
for cubic oscillator as one turning point moves in a compact set keeping the
remaining turning points fixed. 
\item If $a\in \mathcal{\chi }_{\theta },$the Stokes geometry is invariant
as $a$ varies in a connected component of $\mathcal{\chi }_{\theta }\diagdown
\left\{ t_{\theta },e_{\theta }\text{(if they exist)}\right\} $. We deduce that
WKB asymptotic formulas still valid as $a$ describe this connected
component, for $\left\vert \lambda \right\vert >r_{\varepsilon }$ and
uniformly for $z\in D_{\varepsilon }$.
\end{enumerate}
\end{remark}

\begin{corollary}\label{WKB derivative}
$y_{2}^{\prime }(z,\lambda )\sim -\lambda (p_{a}(z))^{
\frac{1}{4}}\exp (-h(\lambda ,z_{0},z))$ as $z\rightarrow \infty ,z\in
\gamma _{\varepsilon }^{+}$ for $\lambda \in \Lambda _{a,\varepsilon
}(\theta )$ (or as $\left\vert \lambda \right\vert \rightarrow +\infty ,$
uniformly for $z\in D_{\varepsilon }$).
\end{corollary}

\begin{proof}
By differentiation of the WKB formula (\ref{uniform asymp}) we have: 
\begin{eqnarray*}
y_{2}^{\prime }(z,\lambda ) &=&\frac{-1}{4}\frac{p_{a}^{\prime }(z)}{
(p_{a}(z))^{\frac{5}{4}}}\exp (-h(\lambda ,z_{0},z))[1+\phi _{2}(z,\lambda
)]- \\
&&\lambda (p_{a}(z))^{\frac{1}{4}}\exp (-h(\lambda ,z_{0},z))[1+\phi
_{2}(z,\lambda )]+(p_{a}(z))^{\frac{-1}{4}}\exp (-h(\lambda ,z_{0},z))\phi
_{2}^{\prime }(z,\lambda )
\end{eqnarray*}
as $z\in \gamma _{\varepsilon }^{+}$ for uniformly for $\lambda \in \Lambda
_{a,\varepsilon }(\theta )$ or as $\lambda \in \Lambda _{a,\varepsilon
}(\theta )$ uniformly for $z\in D_{\varepsilon }$. Hence we obtain:
\begin{equation*}
\frac{y_{2}^{\prime }(z,\lambda )}{-\lambda (p_{a}(z))^{\frac{1}{4}}\exp
(-h(\lambda ,z_{0},z))}=\frac{1}{4\lambda }\frac{p_{a}^{\prime }(z)}{
(p_{a}(z))^{\frac{3}{2}}}[1+\phi _{2}(z,\lambda )]+[1+\phi _{2}(z,\lambda )]-
\frac{\phi _{2}^{\prime }(z,\lambda )}{\lambda }.
\end{equation*}
We have $\phi _{2}(z,\lambda )\rightarrow 0$ as $z\rightarrow \infty ,z\in
\gamma _{\varepsilon }^{+}$ for $\lambda \in \Lambda _{a,\varepsilon
}(\theta )$. From (\ref{asymp derive}) $\phi _{2}^{\prime }(z,\lambda
)\rightarrow 0$ as $z\rightarrow \infty ,z\in \gamma _{\varepsilon }^{+}$
for $\lambda \in \Lambda _{a,\varepsilon }(\theta )$ (or as $\left\vert
\lambda \right\vert \rightarrow +\infty $, uniformly for $z\in
D_{\varepsilon }$). By the fact that $\frac{p_{a}^{\prime }(z)}{(p_{a}(z))^{
\frac{3}{2}}}\rightarrow 0$ as $z\rightarrow \infty $, we achieve the proof.
\end{proof}

\bigskip

\texttt{University of Gabes, Preparatory Institute of Engineering Studies,}

\texttt{Avenue Omar Ibn El Khattab 6029. Gabes. Tunisia.}

\texttt{Research Laboratory Mathematics and Applications LR17ES11, Gabes, Tunisia.}

\texttt{m.marwa.ens@gmail.com}

\texttt{braekgliia@gmail.com}

\texttt{chouikhi.mondher@gmail.com}

\texttt{faouzithabet@yahoo.fr}\bigskip

\end{document}